\documentclass[12pt]{amsart}

\usepackage{amsmath,amsthm,amssymb,amsfonts,mathrsfs,mathtools}
\usepackage{thmtools}
\usepackage{enumitem}
\usepackage{xcolor}
\usepackage[colorlinks=true,citecolor=blue,linkcolor=teal, pagebackref=true]{hyperref}
\usepackage{cleveref}
\usepackage{fullpage}

\newcommand{\kk}{k}
\newcommand{\NN}{\mathbb{N}}
\newcommand{\ZZ}{\mathbb{Z}}
\newcommand{\QQ}{\mathbb{Q}}
\newcommand{\RR}{\mathbb{R}}
\newcommand{\CC}{\mathbb{C}}
\newcommand{\PP}{\mathbb{P}}
\renewcommand{\AA}{\mathbb{A}}

\newcommand{\m}{\mathfrak m}
\newcommand{\n}{\mathfrak n}
\newcommand{\p}{\mathfrak p}
\newcommand{\q}{\mathfrak q}

\newcommand{\len}{\operatorname{\mathrm{length}}}

\newcommand{\Supp}{\operatorname{\mathrm{Supp}}}
\newcommand{\ord}{\operatorname{\mathrm{ord}}}

\newcommand{\ov}{\overline}

\renewcommand{\a}{\mathfrak{a}}
\renewcommand{\b}{\mathfrak{b}}
\newcommand{\fc}{\mathfrak{c}}
\newcommand{\cI}{\mathcal{I}}
\newcommand{\cJ}{\mathcal{J}}
\newcommand{\cO}{\mathcal{O}}

\newcommand{\cQ}{\mathcal{Q}}
\newcommand{\cG}{\mathcal{G}}

\newcommand{\R}{\mathcal{R}}

\newcommand{\lct}{\operatorname{\mathrm{lct}}}

\newcommand{\Loj}{\mathcal{L}} 

\newcommand{\ip}[2]{{\langle #1,#2\rangle}}

\DeclareMathOperator{\conv}{conv}

\DeclareMathOperator{\Val}{Val}
\DeclareMathOperator{\Frac}{Frac}
\DeclareMathOperator{\Hom}{Hom}
\DeclareMathOperator{\argmax}{argmax}

\DeclareMathOperator{\Max}{Max}
\DeclareMathOperator{\Proj}{Proj}
\DeclareMathOperator{\Spec}{Spec}
\DeclareMathOperator{\NP}{NP}

\DeclareMathOperator{\tor}{Tor}

\newcommand{\abul}{{\a_\bullet}}
\newcommand{\bbul}{{\b_\bullet}}

\theoremstyle{plain}
\newtheorem{theorem}{Theorem}[section]
\newtheorem{proposition}[theorem]{Proposition}
\newtheorem{lemma}[theorem]{Lemma}
\newtheorem{corollary}[theorem]{Corollary}

\theoremstyle{definition}
\newtheorem{definition}[theorem]{Definition}
\newtheorem{remark}[theorem]{Remark}
\newtheorem{example}[theorem]{Example}
\newtheorem{question}[theorem]{Question}

\newtheorem{notation}[theorem]{Notation}

\numberwithin{equation}{section}

\crefname{theorem}{Theorem}{Theorems}
\crefname{proposition}{Proposition}{Propositions}
\crefname{lemma}{Lemma}{Lemmas}
\crefname{corollary}{Corollary}{Corollaries}
\crefname{definition}{Definition}{Definitions}
\crefname{remark}{Remark}{Remarks}
\crefname{question}{Question}{Questions}
\crefname{conjecture}{Conjecture}{Conjectures}
\crefname{hypothesis}{Hypothesis}{Hypotheses}

\title[Algebraic \L ojasiewicz exponents]{An algebraic theory of \L ojasiewicz exponents}
	\author{T\`ai Huy H\`a}
\address{Tulane University \\ Mathematics Department \\
	6823 St. Charles Ave. \\ New Orleans, LA 70118, USA}
\email{tha@tulane.edu}



\begin{document}
	\maketitle
	
\begin{abstract}
	We develop a unified algebraic and valuative theory of \L{}ojasiewicz
	exponents for pairs of graded families and filtrations of ideals.
	Within this framework, local \L{}ojasiewicz exponents, gradient
	exponents, and exponents at infinity are all realized as asymptotic
	containment thresholds between filtrations, governed by integral
	closure.
	This reformulation shows that \L{}ojasiewicz exponents are fundamentally
	valuative optimization problems.
	
	The central structural contribution of the paper is a finite--max
	principle.
	Under verifiable algebraic hypotheses, the a priori infinite
	valuative supremum bounding the \L{}ojasiewicz exponent reduces to a
	finite maximum, and computes the \L{}ojasiewicz exponent precisely.
	We identify two complementary mechanisms leading to this phenomenon:
	finite testing arising from normalized blowups and Noetherian Rees algebras, and attainment via compactness of normalized valuation spaces under linear boundedness assumptions.
	
	This finite--max framework yields strong structural consequences.
	We prove rigidity results showing that common extremal valuations force
	equality of \L{}ojasiewicz ratios, and we establish stratification and
	stability phenomena for \L{}ojasiewicz exponents in families, including
	fractional linearity and wall-chamber behavior along natural one--parameter deformations.
	
	The theory recovers and explains classical results in toric
	and Newton--polyhedral settings, particularly, for Newton nondegenerate case, where the \L{}ojasiewicz exponent is computed by
	finitely many toric divisorial valuations corresponding to facet data.
	
	Finally, we illustrate why the hypotheses underlying the finite--max
	principle are essential, delineating the precise scope of the theory.
\end{abstract}

	\tableofcontents

\part{Motivations, Goals and Statement of Main Results} \label{part:Intro}

We start the paper by describing motivation, conceptual background, and an overview of the main results.  


\section{Introduction} \label{sec:intro}

The \L ojasiewicz exponent is a fundamental invariant measuring the rate at which one
analytic object dominates another near a singular locus.  More precisely, given analytic function germs
\[
f_1,\dots,f_s,\; g_1,\dots,g_t \in \mathcal O_{\mathbb C^d,0},
\]
the classical \L ojasiewicz inequality \cite{Lojasiewicz1959} asserts that there exist constants $C>0$, $\theta>0$,
and a neighborhood $U$ of the origin such that
\[
\max_{1\le j\le t} |g_j(x)|^\theta \;\le\; C \max_{1\le i\le s} |f_i(x)|
\quad \text{for all } x\in U.
\]
The infimum of such exponents $\theta$ is called the (analytic) \L ojasiewicz exponent of
$(f_1,\dots,f_s)$ with respect to $(g_1,\dots,g_t)$.

Introduced originally by \L ojasiewicz in the study of real and complex analytic sets,
this exponent captures subtle asymptotic information about singularities: it quantifies
relative vanishing rates, governs gradient inequalities for isolated critical points,
controls stability and equisingularity phenomena, and plays a central role in the study
of global polynomial mappings and optimization problems.  Several important variants
arise naturally, including the \emph{gradient \L ojasiewicz exponent} of hypersurfaces with isolated singularities (cf. \cite{ Boguslawska, DK2005, Gwo99, Lenarcik}) 
and the \emph{\L ojasiewicz exponent at infinity} for polynomial
mappings (cf. \cite{Dimca,VuiDoat,VuiDuc,VuiSon}).  Despite their close conceptual relationship, these invariants have traditionally
been studied using markedly different techniques, ranging from analytic arc methods and
curve selection to Newton polyhedra and projective compactifications.

A key unifying insight already appears in the classical work of
Lejeune-Jalabert and Teissier \cite{LT}.  They showed that analytic
\L ojasiewicz inequalities are equivalent to integral--closure containments between powers
of the corresponding ideals.  Particularly, if
$\mathfrak a=(f_1,\dots,f_s)$ and $\mathfrak b=(g_1,\dots,g_t)$, then inequalities of the
form above hold if and only if there exist integers $p,q$ such that
\[
\mathfrak b^q \subseteq \overline{\mathfrak a^p}.
\]
This result places the analytic \L ojasiewicz exponent squarely within the framework of
ideal containment problems and reveals its intrinsic valuative nature; see also \cite{BochnakRisler,Teissier2012}.  This viewpoint has been generalized by Bivi\`a-Ausina \cite{Aus09}, Bivi\`a-Ausina--Encinas \cite{BE09, BE13} and Bivi\`a-Ausina--Fukui \cite{BF16} in their studies of \L{}ojasiewicz exponent of ideals via integral closures and mixed multiplicites.

Independently, ideal containment problems have played a central role in commutative
algebra and algebraic geometry; notably, the uniform symbolic--ordinary containment theorems of
Ein--Lazarsfeld--Smith \cite{ELS} and Hochster--Huneke \cite{HochsterHuneke},
together with generalization by Ma--Schwede \cite{MaSchwede}.  Motivated by the
fact that such containments are often not sharp, Harbourne and Huneke
\cite{HarbourneHuneke} proposed improved containment conjectures, leading to a large body of work during the last fifteen years; particularly, the introduction of numerical invariants measuring failure of containment, namely, the \emph{resurgence}, introduced by Bocci--Harbourne \cite{BocciHarbourne} and further developed by Guardo--Harbourne--Van~Tuyl \cite{GuardoHarbourneVanTuyl}.  Morerecently, these ideas have been extended to \emph{graded families and filtrations of ideals}, allowing symbolic powers, integral closures, and other asymptotic constructions
to be treated in a unified way; see \cite{HaKumarNguyenNguyen}.

The purpose of this paper is to bring these two threads together.
We develop a unified algebraic and valuative theory of \L{}ojasiewicz
exponents for pairs of graded filtrations of ideals, in which local,
gradient, and at--infinity exponents all arise as \emph{asymptotic
	containment thresholds}.  This reformulation places classical
\L{}ojasiewicz inequalities into a single containment framework and
reveals that \L{}ojasiewicz exponents are governed by valuative
optimization problems.  
In general, \L{}ojasiewicz exponent is bounded by a supremum of asymptotic valuative
ratios associated to the underlying filtrations.  While such variational
descriptions are a priori infinite, a central theme of this paper is
that, under verifiable algebraic or geometric hypotheses, this
supremum is in fact a maximum attained by finitely many 
valuations and computes the \L{}ojasiewicz exponent.  Identifying when this finite--max phenomenon occurs and
exploiting structural consequences, including rigidity, stability, stratification and wall-chamber behaviors, is the main objective of this paper.

Our approach also resonates with the valuation-theoretic framework for multiplier ideals, plurisubharmonic singularities, and interpolation developed by Boucksom–Favre–Jonsson \cite{BFJ2008}, Favre–Jonsson \cite{FJ04,FJ05,FJ2005}, and more recently by Bao–Guan–Mi–Yuan \cite{BGMY2026a,BGMY2026b,BGMY2025}. While our setting is algebraic and centers on ideal containment, a common underlying theme is that singularity exponents can be described as optimization problems over suitable spaces of valuations. In this sense, the present algebraic theory of \L{}ojasiewicz exponents runs parallel to the valuative descriptions of multiplier ideals, analytic singularities, and interpolation phenomena studied in the above works.


\section{Goals and statement of main results}\label{sec:mainResults}

This paper develops a unified algebraic and valuative theory of \L{}ojasiewicz exponents. It has two overarching aims.
The first aim is \emph{conceptual}:
to introduce a general notion of algebraic \L{}ojasiewicz exponent attached to a pair of graded filtrations of ideals and show that all classical analytic \L{}ojasiewicz exponents (local, gradient, and at infinity) arise uniformly as instances of this invariant. From this perspective, \L{}ojasiewicz exponents are interpreted as \emph{asymptotic containment thresholds} between filtrations, placing them within the framework of \emph{integral closure} and \emph{ideal containment} theory. 
The second aim is \emph{structural}:
to show that valuations are the computational mechanism governing these thresholds, and
to isolate verifiable hypotheses under which the resulting valuative optimization problems
reduce to finite maxima over explicit candidate sets.
This \emph{finite--max} philosophy underlies the rigidity, stability, and polyhedral phenomena
developed later in the paper.

\subsection{From analytic inequalities to algebraic thresholds}

After fixing analytic conventions in Section~\ref{sec:analytic}, we introduce in
Section~\ref{sec:algebraicDef} the algebraic \L{}ojasiewicz exponent
$\Loj_{\bbul}(\abul)$ attached to a pair of filtrations
$\abul=\{\a_p\}_{p\ge1}$ and $\bbul=\{\b_q\}_{q\ge1}$.
By definition, $\Loj_{\bbul}(\abul)$ records the smallest asymptotic slope at which the
containments
\[
\b_{qt}\subseteq \overline{\a_{pt}}
\quad\text{hold for } t\gg1 .
\]

The key point is that, for families of ordinary powers of two ideals, this algebraic invariant exactly recovers the analytic one (\Cref{thm:analytic-agreement,thm:analytic-algebraic-local-infty}).
Consequently, local, gradient, and at--infinity \L{}ojasiewicz exponents all arise as
instances of $\Loj_{\bbul}(\abul)$, distinguished only by the choice of filtrations, by the class of valuations allowed to appear, and possibly by inverting the invariant.

\subsection{Valuative formulas and the finite--max principle}

A central feature of the algebraic \L{}ojasiewicz exponent
$\Loj_{\bbul}(\abul)$ is that it admits a valuative description (\Cref{sec:valuative}).
Using the valuative criterion for integral closure, we show that
$\Loj_{\bbul}(\abul)$ is bounded below by a supremum of ratios
\[
\Loj_{\bbul}(\abul)
\;\ge\;
\sup_{v}\frac{v(\abul)}{v(\bbul)},
\]
where $v$ ranges over valuations centered at $\m$
(\Cref{thm:valuative-filtrations}).

A fundamental question is whether this supremum is in fact a maximum,
and whether it computes the \L{}ojasiewicz exponent.
We isolate two independent mechanisms that force such a finite--max phenomenon.
The first is a \emph{finite testing principle}: under suitable
Noetherian hypotheses, containment of integral closures can be tested
on a fixed finite set of divisorial valuations arising from a
normalized Rees algebra or a fixed birational model
(\Cref{thm:finite-testing-implies-max,thm:finite-testing-fingen}).
The second mechanism relies on compactness properties of normalized
valuation spaces together with uniform convergence of the functions
$v\mapsto v(\a_p)/p$ and $v\mapsto v(\b_q)/q$, which ensures attainment
of the supremum even in the absence of finite generation
(\Cref{thm:attainment}). We also illustrate that considering only divisorial valuations may not be enough to realize the supremum in general (\Cref{thm:nondivisorial-attainment-final}).

\subsection{Rigidity, stability, and stratification}

When the finite--max principle applies, the valuative formula for $\Loj_{\bbul}(\abul)$ provides an effective description of the invariant, and the extremal valuation(s) attaining the maximum become geometrically meaningful data.
A further contribution of the paper is to show that these extremal valuations govern the
behavior of \L{}ojasiewicz exponents in families.

We prove that common extremal valuations force rigidity of \L{}ojasiewicz ratios and that,
under finite--max hypotheses, parameter spaces admit a natural stratification according to
the valuation computing the exponent
(\Cref{thm:common-valuation,thm:stratification}).
Along one--parameter families, this yields fractional linearity of $1/\Loj$ and shows that
affineness forces persistence of the extremal valuation
(\Cref{thm:concavity},
\Cref{cor:geodesic-rigidity}).
These results provide a unified algebraic explanation for stability and wall--chamber behavior
observed in analytic and polyhedral settings.

\subsection{Polyhedral models and explicit computations}

We describe explicit formulas and finite computation methods in combinatorial situations.  For monomial and toric ideals, the exponent is determined by polyhedral data and toric valuations (\Cref{sec:toric}).  
In these situations, integral--closure containments are encoded by linear inequalities,
toric divisorial valuations correspond to linear functionals, and the valuative optimization
problem reduces to a finite maximization over facet data of Newton polyhedra
(\Cref{thm:toric-filtration-formula,thm:facet-formula}).

We also discuss in more details how \L{}ojasiewicz exponent at infinity behaves in our algebraic framework.  We prove valuative and Rees--algebraic formulas at infinity (\Cref{thm:meta-infty-rees}), and obtain finite computation and Newton polyhedral formulas in the nondegenerate case (\Cref{thm:newton-finite-infty,cor:meta-infty-nondeg-finite}).

\subsection{Uniformity and necessity}

The final part of the paper illustrates why the hypotheses used to obtain finite testing and
finite--max formulas are essential (\Cref{thm:toric-simultaneous-principalization,thm:wall-chamber}). 
Outside situations admitting verifiable finite testing, the set of valuations computing
$\Loj_{\bbul}(\abul)$ need not admit any a priori finite reduction, delineating the precise
scope of the theory developed here (\Cref{thm:sharpness-monomial}).


\medskip
\noindent\textbf{Acknowledgment.}
The author is especially grateful to Huy Vui H\`a for many stimulating discussions on the \L{}ojasiewicz exponent. 
The author thanks Mattias Jonsson for suggesting several interesting references on the valuative framework. 
Thanks also go to Th\'ai Nguy$\tilde{\text{\^e}}$n, Vinh Pham, and Aniketh Sivakumar for reading preliminary drafts of the paper and for helpful comments. 
The author was partially supported by a Simons Foundation grant.

\part{\L{}ojasiewicz Exponents and the Algebraic Framework} \label{part:algebraicFramework}

In this part, we develop the algebraic framework underlying \L{}ojasiewicz exponents.
We begin by recalling analytic \L{}ojasiewicz exponents and their standard variants, and
then generalize these invariants to a purely algebraic notion attached to graded families and
filtrations of ideals.  This perspective allows \L{}ojasiewicz exponents to be considered in a more general context, interpreted as asymptotic containment thresholds, and sets the stage for the valuative analysis carried out in the subsequent parts of the paper.


\section{Analytic \L{}ojasiewicz exponents} \label{sec:analytic}

\subsection{Background}

In this section we briefly fix the analytic setting and conventions for
\L{}ojasiewicz exponents that will be used throughout the paper.  Our goal is not to
survey the analytic theory, but to record a formulation that is compatible with the
algebraic containment framework developed later.

Given two systems of analytic function germs near a point, \L{}ojasiewicz-type
inequalities compare their relative rates of vanishing.  Several equivalent formulations
appear in the literature, differing mainly by normalization.  We adopt a convention that
aligns the analytic exponent with integral--closure containments and the algebraic
\L{}ojasiewicz exponent defined in Section~\ref{sec:algebraicDef}.

\subsection{Analytic definition}

Let
\[
\a=(f_1,\ldots,f_s),\quad \b=(g_1,\ldots,g_t)\subseteq \cO_{\CC^d,0},
\]
where $0\in f_i^{-1}(0)\cap g_j^{-1}(0)$ for all $i,j$. Define
\[
\|f(x)\|:=\max_{1\le i\le s}|f_i(x)|,\quad
\|g(x)\|:=\max_{1\le j\le t}|g_j(x)|.
\]

\begin{definition}[Analytic \L{}ojasiewicz exponent]\label{def:analytic-L}
	The analytic \L{}ojasiewicz exponent of $\a$ with respect to $\b$ is
	\[
	\Loj^{an}_\b(\a)
	:=\inf\left\{
	\theta>0 :
	\exists\,U\ni0,\ \exists\,C>0\ \text{such that}\
	C\|f(x)\|\ge \|g(x)\|^{\theta}\ \forall x\in U
	\right\}.
	\]
\end{definition}

\begin{remark}[Convention] \label{rem:convention} 
The inequality is written with ``$\ge$'', so that larger values of
		$\Loj^{an}_\b(\a)$ correspond to stronger domination of (the vanishing of) $\a$ by $\b$ near the
		origin. Other formulations in the literature may also differ by inversion of the exponent. 
\end{remark}

\subsection{The gradient \L{}ojasiewicz exponent}

Let $f:(\CC^d,0)\to(\CC,0)$ be a holomorphic function with an isolated critical
point at the origin, and let $J(f)$ denote its Jacobian ideal. The
\emph{gradient} \L{}ojasiewicz exponent of $f$ at 0, usually denoted by $L(f,0)$, is the optimal (infimum) value $\eta>0$ such that
\[
\|\nabla f(x)\|\ge C\| x\|^{\eta}
\]
for $x$ near $0$. In light of Definition~\ref{def:analytic-L}, the gradient \L{}ojasiewicz exponent of $f$ at 0 is given by $L(f,0) = \Loj^{an}_{\m}(J(f)).$

\subsection{The \L{}ojasiewicz exponent at infinity}

In addition to local exponents, one also studies \L{}ojasiewicz inequalities at
infinity for polynomial mappings
\[
F:\CC^n\to\CC^m,
\]
typically of the form
\[
\|x\|^{\alpha}\le C\|F(x)\|\quad\text{for }\|x\|\gg1,
\]
or variants involving $\mathrm{dist}(x,F^{-1}(0))$. The \emph{optimal} (i.e., supremum) exponent $\alpha$ in such
an inequality is called the\emph{ \L{}ojasiewicz exponent of $F$ at infinity}, and denoted by $L_\infty(F)$.

Although inequalities at infinity are global on $\CC^n$, their meaning
is local after projective compactification: the condition $\|x\|\to\infty$ corresponds to approaching
the divisor at infinity, and the growth of $F$ is encoded by the relative vanishing orders of the
source and target defining functions along the boundary.

\section{Algebraic \L ojasiewicz exponents of filtrations} \label{sec:algebraicDef}

In this section we introduce the algebraic \L{}ojasiewicz exponent $\Loj_{\bbul}(\abul)$ attached to a pair of filtrations $\abul=\{\a_p\}_{p\ge1}$ and $\bbul=\{\b_q\}_{q\ge1}$ of ideals in a Noetherian ring, viewed as the asymptotic containment threshold for $\b_{qt}\subseteq \overline{\a_{pt}}$ for $t\gg 1$ (\Cref{def:L-families}).  We show that this invariant recovers the classical analytic exponent in the local analytic setting (\Cref{thm:analytic-agreement}).  Finally, we package \L{}ojasiewicz exponents at infinity into the same algebraic framework by working with the canonical growth filtrations induced along the boundary divisor of a compactification, leading to an equivalence for the local and global “at infinity” invariants in \Cref{thm:analytic-algebraic-local-infty}.

\subsection{Graded families and \L ojasiewicz exponent} 
The algebraic definition naturally involves graded families and filtrations of
ideals, encompassing powers, integral closures, symbolic powers, and growth
filtrations arising in geometric contexts. Throughout this section, $R$ denotes a Noetherian ring.

\begin{definition}[Graded family and filtration]\label{def:graded-family}
	A \emph{graded family of ideals} $\abul=\{\a_p\}_{p\ge1}$ in $R$ is a collection of proper ideals
	satisfying
	\[
	\a_p\a_{p'}\subseteq \a_{p+p'} \quad \text{for all } p,p'\ge1.
	\]
A graded family $\abul=\{\a_p\}_{p\ge1}$ is called a \emph{filtration} if
	\[
	\a_{p+1}\subseteq \a_p \quad \text{for all } p\ge1.
	\]
\end{definition}

\begin{remark}\label{rem:monotone}
	All graded families appearing naturally in this paper — including powers of ideals,
	integral closures of powers, symbolic powers — are filtrations. By \emph{ideals}, we always mean proper (non-unit) ideals.
\end{remark}

The \L{}ojasiewicz exponent records the smallest asymptotic slope at which one
filtration is eventually contained in another.


\begin{definition}[Algebraic \L ojasiewicz exponent for filtrations]
	\label{def:L-families}
	Let $\abul=\{\a_p\}_{p\ge1}$ and $\bbul=\{\b_q\}_{q\ge1}$ be filtrations
	of ideals in $R$.
	The \emph{algebraic \L ojasiewicz exponent of $\abul$ with respect to $\bbul$} is
	\[
	\Loj_{\bbul}(\abul)
	\;:=\;
	\inf\left\{\frac{q}{p}\in\QQ_{>0}\ :\ \b_{qt}\subseteq \overline{\a_{pt}} \text{ for } t \gg 1\right\}
	\ \in\ \RR_{\ge0} \cup \{\infty\},
	\]
	with the convention that $\inf \emptyset = \infty$. For ideals $\a,\b\subseteq R$ we write $\Loj_{\b}(\a)$ for $\Loj_{\{\b^q\}_{q \ge 1}}\bigl(\{\a^p\}_{p\ge1}\bigr)$.
\end{definition}

\subsection{Algebraic and analytic \L ojasiewicz exponents}
We now compare the algebraic \L{}ojasiewicz exponent with its analytic
counterpart in the local analytic setting.
As before, assume $R=\mathcal O_{\CC^d,0}$, $\a=(f_1,\dots,f_s),  \b=(g_1,\dots,g_t)$
where $0 \in f_i^{-1}(0) \cap g_j^{-1}(0)$ for any $i$ and $j$, and
$\|f(x)\|:=\max_i|f_i(x)|,\quad \|g(x)\|:=\max_j|g_j(x)|.$

The comparison rests on the following analytic criterion for integral closure stated by Lejeune-Jalabert--Teissier \cite{LT}.

\begin{lemma}[Lejeune-Jalabert--Teissier]\label{lem:analytic-IC}
	Let $p,q\in\mathbb N$. The following are equivalent:
	\begin{enumerate}
		\item[(1)] $\b^q\subseteq \overline{\a^p}$.
		\item[(2)] There exist a neighborhood $U$ of $0$ and a constant $C>0$ such that
		\[
		\|g(x)\|^q\le C\|f(x)\|^p
		\quad\text{for all }x\in U.
		\]
	\end{enumerate}
\end{lemma}

\begin{proof}
	This is the Lejeune-Jalabert--Teissier analytic criterion for integral closure; see \cite[Th\'eor\`eme~7.2]{LT}.
	For the reader's convenience and to fix normalization conventions, we include a self-contained argument in the appendix.
\end{proof}

Our first result shows that, for two ideals $\a$ and $\b$, the analytic and algebraic \L{}ojasiewicz exponents agree.

\begin{theorem}\label{thm:analytic-agreement}
	We have
	$\mathcal L^{an}_{\b}(\a)=\Loj_{\b}(\a).$
\end{theorem}

\begin{proof} Let $\theta>0$ be admissible for $\mathcal L^{an}_{\b}(\a)$ (see \Cref{def:analytic-L}).
	Then, there exist a neighborhood $U$ of $0$ and a constant $C>0$ such that
	\[
	\|g(x)\|^{\theta}\le C\|f(x)\|\quad \text{for all }x\in U.
	\]
	Shrinking $U$ if necessary, we may assume $\|g(x)\|\le 1$ on $U$.
	Fix $p\in\NN$ and set $q=\lceil p\theta\rceil$.
	Raising the above inequality to the $p$-th power yields
	\[
	\|g(x)\|^{p\theta}\le C^{p}\|f(x)\|^{p}\quad (x\in U).
	\]
	Since $q\ge p\theta$ and $\|g(x)\|\le 1$, we have $\|g(x)\|^{q}\le \|g(x)\|^{p\theta}$, and so
	\[
	\|g(x)\|^{q}\le C^{p}\|f(x)\|^{p}\quad (x\in U).
	\]
	By \Cref{lem:analytic-IC}, this implies $\b^{q}\subseteq \overline{\a^{p}}.$
	It follows that, for any $t \in \NN$, we have 
	\[
	\b^{qt} \subseteq \left(\overline{\a^p}\right)^t \subseteq \overline{\a^{pt}}.
	\]
	Thus, $\frac{q}{p}$ is admissible for $\Loj_{\b}(\a)$. Therefore, 
	$$\Loj_{\b}(\a)\le \frac{q}{p} < \theta+\frac1p.$$
	By replacing $\frac{q}{p}$ with $\frac{qt}{pt}$, which is also admissible for $\Loj_{\b}(\a)$, and letting $t \rightarrow \infty$ (so $pt \rightarrow \infty$), we conclude that $\Loj_{\b}(\a)\le \theta$. Taking the infimum over admissible $\theta$ gives 
	$$\Loj_{\b}(\a)\le \Loj^{an}_{\b}(\a).$$
	
	Conversely, let $\frac{q}{p}\in\QQ_{>0}$ be admissible for $\Loj_{\b}(\a)$. Then, for $t \gg 1$, we have $\b^{qt}\subseteq \overline{\a^{pt}}.$
	By \Cref{lem:analytic-IC}, there exist a neighborhood $U$ of $0$ and a constant $C>0$ such that
	\[
	\|g(x)\|^{qt}\le C\|f(x)\|^{pt}\quad \text{for all }x\in U.
	\]
	Taking $(pt)$-th roots gives
	\[
	\|g(x)\|^{q/p}\le C^{1/(pt)}\|f(x)\|\quad (x\in U).
	\]
	Hence, $\theta=\frac{q}{p}$ is admissible for $\mathcal L^{an}_{\b}(\a)$, and so
	$\mathcal L^{an}_{\b}(\a)\le \frac{q}{p}$.
	Taking the infimum over admissible $\frac{q}{p}$ yields
	$$\mathcal L^{an}_{\b}(\a)\le \Loj_{\b}(\a).$$
This completes the proof.
\end{proof}

Observe, in particular, that when $\a$ and $\b$ are $\m$--primary, both the analytic and algebraic \L{}ojasiewicz exponents are finite and governed by integral-closure containments.

\begin{example} \label{ex.analytic=algebraic}
	Let $R = \CC[[x_1, \dots, x_n]]$, $\m = (x_1, \dots, x_n)$ and consider
	$$\a = (x_1^{\alpha_1}, \dots, x_n^{\alpha_n}) \text{ and } \b = \m.$$
	It can be seen that the inequality
	$$\|(x_1, \dots, x_n)\|^\theta \le C \max(|x_1|^{\alpha_1}, \dots, |x_n|^{\alpha_n})$$
	fails for $\theta < \max \{\alpha_1, \dots, \alpha_n\}$ along the coordinate axis, and holds for $\theta = \max\{\alpha_1, \dots, \alpha_n\}$. Thus, by definition,
	$\Loj^{an}_{\b}(\a) = \max\{\alpha_1, \dots, \alpha_n\}.$
	
	On the other hand, one can check on pure powers of the $x_i$'s that
	$$\m^{qt} \subseteq \overline{\a^{pt}} \text{ for } t \gg 1 \Longleftrightarrow \frac{q}{p} \ge \max\{\alpha_1, \dots, \alpha_n\}.$$
	That is,
	$$\Loj_{\b}(\a) = \max\{\alpha_1, \dots, \alpha_n\} = \Loj^{an}_{\b}(\a).$$
\end{example}

\subsection{Exponents at infinity in the algebraic framework}\label{subsec:algebraic-infinity}

We shall explore an algebraic formulation of \L{}ojasiewicz exponents at infinity, understand precisely how the analytic condition $\|x\|\gg 1$ is encoded by divisors at infinity on a suitable compactification, and  show that the resulting analytic invariants agree with the algebraic containment invariants used in our
valuative framework.

Let $F=(f_1,\dots,f_m):\AA^n_\CC\to\AA^m_\CC$ be a polynomial mapping.
We fix the standard projective compactifications
$\AA^n_\CC\hookrightarrow \PP^n_\CC$ and $\AA^m_\CC\hookrightarrow \PP^m_\CC$,
with hyperplanes at infinity
\[
H_X:=\{Z_0=0\}\subseteq \PP^n_\CC \text{ and } H_Y:=\{W_0=0\}\subseteq \PP^m_\CC,
\]
where $Z_0, \dots, Z_n$ and $W_0, \dots, W_m$ denote the homogeneous coordinates of $\PP^n_\CC$ and $\PP^m_\CC$, respectively.
Let $\Gamma_F\subseteq \AA^n_\CC\times\AA^m_\CC$ be the graph of $F$ and let
$\overline{\Gamma}_F\subseteq \PP^n_\CC\times\PP^m_\CC$ be the normalization of the
Zariski closure of $\Gamma_F$.
Let
\[
\pi:\overline{\Gamma}_F\to\PP^n_\CC \text{ and } \rho:\overline{\Gamma}_F\to\PP^m_\CC
\]
be the natural projections and set
\[
D_X:=\pi^{-1}(H_X) \text{ and } D_Y:=\rho^{-1}(H_Y).
\]
Denote by $\mathscr I_X:=\cI_{D_X}$ and $\mathscr I_Y:=\cI_{D_Y}$ the corresponding
ideal sheaves of $D_X$ and $D_Y$ on $\overline{\Gamma}_F$.
Approaching infinity in the source along the graph of $F$ corresponds precisely to
approaching $D_X$ inside $\overline{\Gamma}_F$.

\medskip
\noindent\textit{What does $\|x\|\gg 1$ mean near a boundary point.}
Fix $\xi\in D_X$.
Choose indices $i\in\{1,\dots,n\}$ and $\ell\in\{0,1,\dots,m\}$ such that
$\pi(\xi)\in\{Z_i\neq 0\}$ and $\rho(\xi)\in\{W_\ell\neq 0\}$.
Define holomorphic functions near $\xi$ by
\[
u_X:=\Big(\frac{Z_0}{Z_i}\Big)\circ\pi \text{ and } u_Y:=\Big(\frac{W_0}{W_\ell}\Big)\circ\rho.
\]
Then, $u_X$ is a local defining equation of $D_X$ at $\xi$.
If $\xi\in D_Y$ and $\ell\ge 1$, then $u_Y$ is a local defining equation of $D_Y$ at $\xi$.
If $\xi\notin D_Y$, then we may take $\ell=0$ and then $u_Y\equiv 1$ near $\xi$.

\begin{lemma}\label{lem:norms-vs-uX-uY}
With notation as above, there exist an analytic neighborhood $U$ of $\xi$ in
$\overline{\Gamma}_F$ and constants $C_X\ge 1$ and $C_Y\ge 1$ such that
for all $(x,F(x))\in U\cap\Gamma_F$ one has
\begin{equation}\label{eq:norm-uX}
\frac{1}{|u_X|}\le \|x\|\le \frac{C_X}{|u_X|}.
\end{equation}
Moreover, if $\xi\in D_X\cap D_Y$ and $\ell\ge 1$, then for all $(x,F(x))\in U\cap\Gamma_F$ one has
\begin{equation}\label{eq:norm-uY}
\frac{1}{|u_Y|}\le \|F(x)\|\le \frac{C_Y}{|u_Y|}.
\end{equation}
If $\xi\in D_X\setminus D_Y$, then $u_Y$ is a unit near $\xi$ and $\|F(x)\|$ is bounded on
$U\cap\Gamma_F$.
In particular, for $\xi\in D_X\cap D_Y$ and $\theta\ge 0$, the inequality
\[
\|x\|^\theta \le C\|F(x)\|
\]
holds on $(U\cap\Gamma_F)\setminus D_X$ if and only if
\begin{equation}\label{eq:uY-uX-theta}
|u_Y|\le C'|u_X|^\theta
\end{equation}
holds on $(U\cap\Gamma_F)\setminus D_X$.
\end{lemma}

\begin{proof}
Work in the affine chart $\{Z_i\neq 0\}\subseteq \PP^n_\CC$ with coordinates
$z_0:=Z_0/Z_i$ and $z_j:=Z_j/Z_i$ for $j\neq i$.
On $\AA^n_\CC=\{Z_0\neq 0\}$ we have $x_j=Z_j/Z_0=z_j/z_0$.
Thus, on $U\cap\Gamma_F$ we have $u_X=z_0\circ\pi$ and
\[
|x_j|=\frac{|z_j\circ\pi|}{|u_X|}.
\]
Shrinking $U$ if necessary, we may assume $U$ is relatively compact.
Since each $z_j\circ\pi$ is holomorphic on $U$, it is bounded on $U$.
Hence, there exists $M_X\ge 1$ such that $|z_j\circ\pi|\le M_X$ on $U$ for all $j$.
This gives $\|x\|\le M_X/|u_X|$, which is the upper bound in \eqref{eq:norm-uX}.
For the lower bound, note that $x_i=Z_i/Z_0=1/(Z_0/Z_i)=1/z_0$, so on $U\cap\Gamma_F$ we have
$\|x\|\ge |x_i|=1/|u_X|$.

For the target side, the argument is analogous in the affine chart $\{W_\ell\neq 0\}\subseteq \PP^m_\CC$.
If $\ell\ge 1$ and $\xi\in D_Y$, then $u_Y=(W_0/W_\ell)\circ\rho$ vanishes at $\xi$ and
$F_j=W_j/W_0=(W_j/W_\ell)/(W_0/W_\ell)$, so bounding the affine coordinates yields \eqref{eq:norm-uY}.
If $\xi\in D_X\setminus D_Y$, we may take $\ell=0$, so $u_Y\equiv 1$ near $\xi$ and $\|F(x)\|$ is bounded.
Finally, when $\xi\in D_X\cap D_Y$, combining \eqref{eq:norm-uX} and \eqref{eq:norm-uY} shows that
$\|x\|^\theta \le C\|F(x)\|$ holds on $(U\cap\Gamma_F)\setminus D_X$ if and only if
$|u_Y|\le C'|u_X|^\theta$ holds there.
\end{proof}

\medskip
\noindent\textit{Local and global analytic exponents at infinity.} \Cref{lem:norms-vs-uX-uY} allows us to recast the definition of \L{}ojasiewicz exponent at infinity in terms of divisors at infinity.

\begin{definition}\label{def:local-analytic-loj-infty}  \quad
\begin{enumerate} 
\item Fix $\xi\in D_X$ and let $u_X$ and $u_Y$ be as above. The \emph{local analytic \L{}ojasiewicz exponent at infinity of $F$ along $\xi$} is
\[
\Loj^{\text{an}}_{\infty,\xi}(F)
:=
\sup\left\{
\theta\ge 0
\ \middle|\
\exists\,C>0 \text{ such that } |u_Y|\le C|u_X|^\theta
\text{ on } (U_\xi\cap\Gamma_F)\setminus D_X
\right\},
\]
where $U_\xi$ is a sufficiently small analytic neighborhood of $\xi$ in $\overline{\Gamma}_F$. (Note that, if $\xi\in D_X\setminus D_Y$, then $u_Y$ is a unit near $\xi$ and therefore
$\Loj^{\text{an}}_{\infty,\xi}(F)=0$.)
\item The \emph{global analytic \L{}ojasiewicz exponent at infinity} of $F$ is
\[
\Loj^{\text{an}}_{\infty}(F)
:=
\inf_{\xi\in D_X}\Loj^{\text{an}}_{\infty,\xi}(F).
\]
\end{enumerate}
\end{definition}

We work with the convention $\inf\varnothing=\infty$, $1/\infty=0$ and $1/0=\infty$. The next result shows that \L{}ojasiewicz exponent at infinity can also be understood in our algebraic framework.

\begin{theorem}\label{thm:analytic-algebraic-local-infty}
For every $\xi\in D_X$, consider the graded families of ideals in $\cO_{\overline{\Gamma}_F,\xi}$ given by
\[
\a_{\bullet,\xi}:=\{\mathscr I_{X,\xi}^p\}_{p\ge1}
\text{ and }
\b_{\bullet,\xi}:=\{\mathscr I_{Y,\xi}^q\}_{q\ge1}.
\]
Then,
\[
\Loj^{\text{an}}_{\infty,\xi}(F)=\frac{1}{\Loj_{\b_{\bullet,\xi}}(\a_{\bullet,\xi})} = \frac{1}{\Loj_{\mathscr I_{Y,\xi}}(\mathscr I_{X, \xi})}.
\]
Equivalently,
\[
\Loj^{\text{an}}_{\infty,\xi}(F)
=
\frac{1}{\inf\left\{
\frac{q}{p}\in\QQ_{>0}
\ \middle|\
\mathscr I_{Y,\xi}^{\,qt}\subseteq \overline{\mathscr I_{X,\xi}^{\,pt}}
\text{ for all }t\gg 1
\right\}}.
\]
\end{theorem}

\begin{proof}
If $\xi\in D_X\setminus D_Y$, then $\mathscr I_{Y,\xi}=\cO_{\overline{\Gamma}_F,\xi}$, so the admissible
set in Definition~\ref{def:L-families} for $\Loj_{\b_{\bullet,\xi}}(\a_{\bullet,\xi})$ is empty and $\Loj_{\b_{\bullet,\xi}}(\a_{\bullet,\xi})=\infty$, i.e., $1/\Loj_{\b_{\bullet,\xi}}(\a_{\bullet,\xi}) = 0$.
On the other hand, $u_Y$ is a unit near $\xi$, so
$\Loj^{\text{an}}_{\infty,\xi}(F)=0$ by Definition~\ref{def:local-analytic-loj-infty}.
Thus, the equality holds in this case.

Assume now $\xi\in D_X\cap D_Y$.
Let $\theta=p/q\in\QQ_{>0}$.
By Lemma~\ref{lem:norms-vs-uX-uY}, the condition
$\theta\le \Loj^{\text{an}}_{\infty,\xi}(F)$ is equivalent to the existence of
$C>0$ such that
\[
|u_Y|\le C|u_X|^{p/q}
\text{ on }(U_\xi\cap\Gamma_F)\setminus D_X.
\]
Equivalently, $|u_Y|^q\le C^q|u_X|^p$ holds in a punctured neighborhood of $\xi$.
By \cite[Th\'eor\`em 2.1]{LT}, this is the case if and only if $(u_Y)^q \subseteq \overline{(u_X)^p}$.
Since $D_X$ and $D_Y$ are Cartier divisors, we have
$\mathscr I_{X,\xi}=(u_X)$ and $\mathscr I_{Y,\xi}=(u_Y)$.
Thus, the above containment is equivalent to
\[
\mathscr I_{Y,\xi}^{\,q}\subseteq \overline{\mathscr I_{X,\xi}^{\,p}}
\text{ in }\cO_{\overline{\Gamma}_F,\xi}.
\]
Raising to $t$-th powers yields
\[
\mathscr I_{Y,\xi}^{\,qt}\subseteq \overline{\mathscr I_{X,\xi}^{\,pt}}
\text{ for all }t\ge 1.
\]
Therefore,
\[
\frac{1}{\Loj^{\text{an}}_{\infty,\xi}(F)}
=
\inf\left\{
\frac{q}{p}\in\QQ_{>0}
\ \middle|\
\mathscr I_{Y,\xi}^{\,qt}\subseteq
\overline{\mathscr I_{X,\xi}^{\,pt}}
\text{ for all }t\gg 1
\right\}
=
\Loj_{\b_{\bullet,\xi}}(\a_{\bullet,\xi}). \qedhere
\]
\end{proof}


\part{Valuations as the Computational Engine} \label{part:valuative}

In this part, we describe the \L{}ojasiewicz exponent in valuative terms and
identify conditions under which the lower bound supremum over valuations is
attained, reduces to a finite maximum, and computes the \L{}ojasiewicz exponent.

\section{Valuative formulas and the finite--max principle} \label{sec:valuative}

In this section we establish the valuative framework that serves as the computational engine of the paper.  For a pair of filtrations $(\abul,\bbul)$ on a Noetherian local domain, we first express $\Loj_{\bbul}(\abul)$ in terms of the asymptotic valuative data $v(\abul)$ and $v(\bbul)$ and obtain the fundamental bound
\[
\Loj_{\bbul}(\abul)\ \ge\ \sup_{v}\frac{v(\abul)}{v(\bbul)},
\]
where $v$ ranges over valuations centered at $\m$ (\Cref{thm:valuative-filtrations}).  The main structural contribution is to isolate verifiable hypotheses under which this supremum both equals $\Loj_{\bbul}(\abul)$ and is attained by a valuation, so the optimization problem reduces to a finite maximum: first via finite divisorial testing sets arising from Noetherian/finitely generated Rees algebras (\Cref{thm:finite-testing-implies-max,thm:finite-testing-fingen}), and second via compactness of normalized valuation spaces together with uniform convergence of the functions $v\mapsto v(\a_p)/p$ and $v\mapsto v(\b_q)/q$ (\Cref{thm:attainment,thm:nondivisorial-attainment-final,thm:finite-max-rees}).

\subsection{Valuations and asymptotic values of filtrations}
We recall basic notions on valuations and the asymptotic value of a
graded family with respect to a valuation. For a geometric perspective on spaces of valuations and their role in controlling
singularity invariants -- especially in dimension two -- we refer to Favre--Jonsson \cite{FJ04} monograph.

Throughout this section, $(R,\m)$ is a Noetherian local domain.
A \emph{(semi)valuation} on $R$ is a map
\[
v:R\setminus\{0\}\to \RR_{\ge0}
\]
satisfying $v(xy)=v(x)+v(y)$ and $v(x+y)\ge \min\{v(x),v(y)\}$, extended by $v(0)=\infty$.
When discussing valuations on $R$, we always assume $v$ is \emph{centered at $\m$}, i.e.,  $v(x)>0$ for all $x\in\m$.
For an ideal $I\subseteq R$ we write
\[
v(I):=\min\{v(x):x\in I\}.
\]
For a graded family $\abul=\{\a_p\}_{p\ge1}$, set
\[
v(\abul):=\lim_{p\to\infty}\frac{v(\a_p)}{p}
=\inf_{p\ge1}\frac{v(\a_p)}{p}.
\]

\begin{remark}[Subadditivity and existence of $v(\abul)$]\label{lem:vabul-exists}
	Let $\abul=\{\a_p\}_{p \ge 1}$ be a graded family and let $v$ be a valuation centered at $\m$.
	Then, the sequence $\{v(\a_p)\}_{p \ge 1}$ is subadditive:
	\[
	v(\a_{p+p'})\le v(\a_p)+v(\a_{p'}) \quad \forall p,p'\ge1.
	\]
	Thus, by Fekete's lemma, the limit $v(\abul)=\lim_{p\to\infty} v(\a_p)/p$ exists and equals
	$\inf_{p\ge1} v(\a_p)/p$.
\end{remark}

\subsection{Valuative criterion for integral closure}
We recall classical valuative criteria for integral closure, which we will repeatedly used in the paper.

\begin{lemma}[Classical valuative criterion for integral closure]\label{thm:valuative-ic}
  Let $(R,\m)$ be a Noetherian local domain, $I\subseteq R$ an ideal, and $x\in R$.
  Then, $x\in \overline{I}$ if and only if
  \[
    v(x)\ge v(I)
  \]
  for every valuation $v$ of $\Frac(R)$ centered at $\m$. Moreover, it suffices to test the inequalities on the finitely many Rees valuations of $I$, equivalently on the divisorial valuations associated to prime divisors on the normalized blowup of $I$.
\end{lemma}

\begin{proof} This result is classical and goes back to Rees \cite{Rees} and Lipman \cite{Lipman}; see also Huneke--Swanson \cite[Chapters 6 and 10]{HS} for a modern account. 
\end{proof}

\begin{remark}[Analytic arc criterion]\label{rem:arc-criterion}
	When $R=\mathcal O_{\CC^d,0}$, the same condition can be tested on analytic arcs
	via the curve criterion for integral closure; see, for instance \cite{LT}. 
\end{remark}

On a fixed birational model, the classical valuative criterion gives rise to the following reduction statement.

\begin{lemma}[Integral--closure reduction on a fixed model]
	\label{thm:IC-reduction}
	Let $(R,\m)$ be an excellent Noetherian local domain. Let $\pi: Y \to \Spec(R)$ be a proper birational morphism such that $Y$ is normal. Let $\a,\b\subseteq R$ be ideals, and assume that $\a\mathcal O_Y$ and $\b\mathcal O_Y$ are invertible. Write
	\[
	\a\mathcal O_Y=\mathcal O_Y(-A), \quad
	\b\mathcal O_Y=\mathcal O_Y(-B),
	\]
	with $A,B$ effective Cartier divisors on $Y$.
	
	Then, for all $p,q\ge 1$, the following are equivalent:
	\begin{enumerate}
		\item $\b^q \subseteq \overline{\a^p}$ in $R$;
		\item $\pi^*(\b^q)\subseteq \overline{\pi^*(\a^p)}$ in $\mathcal O_Y$;
		\item $q\,B \ge p\,A$ as divisors on $Y$;
		\item $\ord_E(\b^q)\ge \ord_E(\a^p)$ for every prime divisor $E$ on $Y$.
	\end{enumerate}
	In particular, integral--closure containments are completely determined by
	the finitely many divisorial valuations $\{\ord_E\}$ arising from prime divisors in the supports of $A$ and $B$.

\end{lemma}

\begin{proof} Observe that $\a^p\cO_Y = \cO_{Y}(-pA)$ and $\b^q\cO_Y = \cO_{Y}(-qB)$ are invertible, and so, integrally closed on $Y$. Furthermore, since $Y$ is normal, $\pi$ factors through the normalized blowup of $\a$. Thus, we have the following standard equality (cf. \cite[Chapters 10 and 19]{HS}):
	$$\overline{\a^p} = \pi_*(\a^p \cO_Y).$$
Particularly, if $\b^q\cO_Y \subseteq \a^p\cO_Y$, then pushing forward gives
$\b^q \subseteq \pi_*(\a^p\cO_Y) = \overline{\a^p}.$
On the other hand, if $\b^q \subseteq \overline{\a^p} = \pi_*(\a^p\cO_Y)$ then $\b^q \cO_Y \subseteq (\pi_*(\a^p\cO_Y))\cO_Y \subseteq \a^p\cO_Y$, where the latter containment follows from the natural evaluation map $\pi^*\pi_*(\a^p\cO_Y) \rightarrow \a^p\cO_Y$. Hence, we have the equivalence (1) $\Longleftrightarrow$ (2).

The equivalences with (3) and (4) follow from the fact that $\a^p\cO_Y$ and $\b^q\cO_Y$ are invertible sheaves, and so $\b^q \cO_Y \subseteq \a^p \cO_Y \Longleftrightarrow \cO_Y(-qB) \subseteq \cO_Y(-pA)$, which is equivalent to $qB \ge pA \Longleftrightarrow q\ord_E(\b) \ge p\ord_E(\a)$ for all prime divisors $E \subseteq Y$.
\end{proof}

\begin{remark}
	In the Newton nondegenerate setting, the toric modification associated to the
	Newton fan of $f$ provides such a morphism $\pi$, and the divisors $E$ correspond
	to compact facets of the Newton polyhedron (cf. \cite{CLS,Khovanskii,Oka}).
\end{remark}

\subsection{Valuative formula for $\Loj_{\bbul}(\abul)$}
We derive a general valuative lower bound for the algebraic \L{}ojasiewicz
exponent in terms of ratios of asymptotic valuation values. This is the basic computational engine that allows us to explore the finiteness mechanisms later.

\begin{theorem}[Valuative lower bound]\label{thm:valuative-filtrations}
	Assume that $(R,\m)$ is a Noetherian local domain. Let $\abul=\{\a_p\}_{p\ge1}$ and
	$\bbul=\{\b_q\}_{q\ge1}$ be graded families of ideals in $R$.
	Then
	\[
	\Loj_{\bbul}(\abul)\ \ge\ \sup_{v}\ \frac{v(\abul)}{v(\bbul)},
	\]
	where the supremum is taken over all valuations $v$ centered at $\m$ with $v(\bbul) > 0$.
\end{theorem}

\begin{proof} Set $M = \sup_v {v(\abul)}/{v(\bbul)}.$ 
	Fix a rational number $\frac{q}{p}\in\QQ_{>0}$ which is admissible in the definition of
	$\Loj_{\bbul}(\abul)$, i.e.\ such that
	\[
	\b_{qt}\subseteq \overline{\a_{pt}}
	\quad\text{for all integers }t\gg 1.
	\]
	Let $v$ be a valuation of $\Frac(R)$ centered at $\m$ with $v(\bbul) > 0$.
	By the valuative criterion for integral closure (\Cref{thm:valuative-ic}), the containment
	$\b_{qt}\subseteq \overline{\a_{pt}}$ implies $v(\b_{qt}) \ge v(\a_{pt})$ for all $t \gg 1$. Thus,
	\begin{align}
	\frac{v(\b_{qt})}{t}\ \ge\ \frac{v(\a_{pt})}{t}
	\quad\text{for all } t\gg 1. \label{eq.vDivt}
	\end{align}
	
	By \Cref{lem:vabul-exists}, the limits $v(\abul)=\lim_{r\to\infty}\frac{v(\a_r)}{r}$ and $v(\bbul)=\lim_{r\to\infty}\frac{v(\b_r)}{r}$
	exist. In particular, we have
	\[
	\lim_{t\to\infty}\frac{v(\a_{pt})}{pt}=v(\abul)
	\text{ and }
	\lim_{t\to\infty}\frac{v(\b_{qt})}{qt}=v(\bbul).
	\]
This, together with (\ref{eq.vDivt}), implies that $q\,v(\bbul)\ \ge\ p\,v(\abul).$
	Since $v(\bbul)>0$, we conclude that
	\[
	\frac{q}{p}\ \ge\ \frac{v(\abul)}{v(\bbul)}.
	\]
	As this holds for every valuation $v$ centered at $\m$, it follows that $\frac{q}{p} \ge M.$
Taking the infimum over all admissible $\frac{q}{p}$ in the definition of
	$\Loj_{\bbul}(\abul)$ yields the desired inequality $\Loj_{\bbul}(\abul) \ge M$.
\end{proof}

\begin{example}
	\label{ex.valuative-bound}
	Let $R = \CC[[x,y]], \m = (x,y)$, and consider 
	$$\a = (x^\alpha, y^\beta) \text{ and } \b = \m.$$
	For a monomial valuation $v_u$, corresponding to $u = (u_1,u_2) \in \ZZ^2_{\ge 0}$, with $v_u(x) = u_1$ and $v_u(y) = u_2$, we get
	$$\frac{v_u(\a)}{v_u(\b)} = \frac{\min\{\alpha \cdot u_1, \beta \cdot u_2\}}{\min \{u_1,u_2\}}.$$
	Taking the maximum value over $u$ gives $\max\{\alpha, \beta\}$. Thus, \Cref{thm:valuative-filtrations} gives
	$$\Loj_{\b}(\a) \ge \max\{\alpha, \beta\}.$$
	On the other hand, by \Cref{ex.analytic=algebraic}, we know that $\Loj_{\b}(\a) = \max\{\alpha, \beta\}.$ This example illustrates that the bound given in \Cref{thm:valuative-filtrations} is sharp.
\end{example}

\Cref{thm:valuative-filtrations} shows that $\Loj_{\bbul}(\abul)$ is governed by a valuative
optimization problem. In general, however, the inquality may be strict, the resulting supremum need not be a maximum, and the relevant candidate valuations need not be controlled uniformly without additional
finiteness hypotheses. The next subsections isolate two mechanisms that force the supremum
to be attained and to reduce to a \emph{finite} maximum.


\subsection{From supremum to maximum: finite testing and Noetherian Rees algebra}
We identify the first mechanism that forces the valuative supremum to define the \L{}ojasiewicz exponent and to be a finite maximum.

\begin{theorem}[Finite testing set $\Rightarrow$ finite maximum formula]\label{thm:finite-testing-implies-max}
	Let $(R,\m)$ be a local domain and let
	$\abul=\{\a_p\}_{p\ge 1}$ and $\bbul=\{\b_q\}_{q\ge 1}$ be graded families of ideals in $R$.
	Fix a finite set of (divisorial) valuations $V=\{v_1,\dots,v_N\}$ of $\Frac(R)$ centered at $\m$.
	Assume:
	\begin{enumerate}
		\item[(H1)] For all $p,q\ge 1$ one has the equivalence
		\begin{equation}\label{eq:finite-testing-equivalence}
			\b_q\subseteq \overline{\a_p}
			\quad\Longleftrightarrow\quad
			v_i(\b_q)\ge v_i(\a_p)\ \text{ for all }1\le i\le N.
		\end{equation}
		\item[(H2)] For each $i$ we have $v_i(\bbul)>0$.
	\end{enumerate}
	Then,
	\[
		\Loj_{\bbul}(\abul) = \max_{1\le i\le N}\frac{v_i(\abul)}{v_i(\bbul)}.
	\]
\end{theorem}

\begin{proof}
	Set
	\[
	M \ :=\ \max_{1\le i\le N}\ \frac{v_i(\abul)}{v_i(\bbul)} \ \in\ (0,\infty),
	\]
	which is finite by {\rm(H2)}. We prove the two inequalities separately.
	
	\smallskip
	\noindent
	\emph{Step 1: Lower bound $\Loj_{\bbul}(\abul)\ge M$.}
	Fix $\frac{q}{p}\in\QQ_{>0}$ admissible in the definition of $\Loj_{\bbul}(\abul)$, so that
	\[
	\b_{qt}\ \subseteq\ \overline{\a_{pt}}\quad \text{for all }t\gg 1.
	\]
	By the same arguments as in the proof of \Cref{thm:valuative-filtrations}, we get
	\[
	\frac{q}{p}\ \ge\ \frac{v_i(\abul)}{v_i(\bbul)}\quad \text{for all }1\le i\le N.
	\]
Thus, $\frac{q}{p}\ge M$. Taking the infimum over all admissible $\frac{q}{p}$ gives
	$\Loj_{\bbul}(\abul)\ge M$.
	
	\smallskip
	\noindent
	\emph{Step 2: Upper bound $\Loj_{\bbul}(\abul)\le M$.}
	Fix $\varepsilon>0$ and set $\beta \ :=\ \min_{1\le i\le N} v_i(\bbul)\ >\ 0.$
Choose $p\gg 1$ and set
	\[
	q \ :=\ \left\lceil \left(M+\frac{\varepsilon}{\beta}\right)p\right\rceil.
	\]
	For each $i$, by the definition of $v_i(\abul)$ as a limit along the subsequence $\{pt\}$,
	there exists $t_0(i)\in\NN$ such that for all integers $t\ge t_0(i)$,
	\[
	\frac{v_i(\a_{pt})}{pt}\ \le\ v_i(\abul)+\varepsilon.
	\]
	Let $t_0:=\max_i t_0(i)$, so the above estimate holds for all $i$ and all $t\ge t_0$.
	
	On the other hand, since $v_i(\b_{qt})\ \ge\ qt\,v_i(\bbul)$, for every $t\ge t_0$ and every $i$, we have
	\[
	v_i(\b_{qt})
	\ \ge\ qt\,v_i(\bbul)
	\ \ge\ \left(M+\frac{\varepsilon}{\beta}\right)pt\,v_i(\bbul)
	\ \ge\ pt\,(M\,v_i(\bbul)+\varepsilon)
	\ \ge\ pt\,(v_i(\abul)+\varepsilon)
	\ \ge\ v_i(\a_{pt}).
	\]
	That is, for all integers $t\ge t_0$, $v_i(\b_{qt})\ge v_i(\a_{pt})$ for every $i$.
	By {\rm(H1)} applied to $(pt,qt)$, it follows that
	\[
	\b_{qt}\ \subseteq\ \overline{\a_{pt}}\quad \text{for all }t\ge t_0.
	\]
	Thus, $\frac{q}{p}$ is admissible in the definition of $\Loj_{\bbul}(\abul)$, and so
	\[
	\Loj_{\bbul}(\abul)\ \le\ \frac{q}{p}
	\ \le\ M+\frac{\varepsilon}{\beta}+\frac{1}{p}.
	\]
	Letting $p\to\infty$ gives $\Loj_{\bbul}(\abul)\le M+\frac{\varepsilon}{\beta}$, and then
	letting $\varepsilon\to 0$ yields $\Loj_{\bbul}(\abul)\le M$.
This completes the proof.
\end{proof}

In most applications below, the equivalence in (H1) arises from integral-closure testing on a fixed model (e.g., a normalized blowup or toric modification). While the forward implication alone suffices for the lower bound, the equivalence ensures that admissible containments can be produced uniformly.


Hypotheses for the finite testing principle (\Cref{thm:finite-testing-implies-max}) are satisfied when the Rees algebra is Noetherian, as shown in the next result. 

\begin{theorem}[Finite Rees valuations for a Noetherian filtration]
	\label{thm:finite-testing-fingen}
	Let $(R,\m)$ be an excellent Noetherian local domain, and let
	$\abul=\{\a_p\}_{p\ge1}$ be a filtration of $\m$-primary ideals in $R$.
	Assume that the Rees algebra $\R(\abul):=\bigoplus_{p\ge0}\a_pT^p$
	is finitely generated as an $R$--algebra.
	Let $\overline{\R(\abul)}$ denote its integral closure, and set $\overline{Y}:=\Proj \overline{\R(\abul)}.$
	Then, there exists a finite set of divisorial valuations
	\[
	\mathcal V(\abul)=\{v_1,\dots,v_r\},
	\]
	arising from prime divisors on $\overline{Y}$ lying over $\m$, such that for every filtration $\bbul=\{\b_q\}_{q\ge1}$ of ideals in $R$ satisfying $v_i(\bbul) > 0$ for all $1 \le i \le r$, one has
	\[
	\Loj_{\bbul}(\abul)
	=
	\max_{1\le i\le r}
	\frac{v_i(\abul)}{v_i(\bbul)}.
	\]
\end{theorem}

\begin{proof}
	We break the proof into two steps.
	
	\smallskip
	\noindent\textit{Step 1: a finite divisorial testing set for the integral closures
		$\overline{\a_p}$.}
	Let
	\[
	\R:=\R(\abul)=\bigoplus_{p\ge 0}\a_pT^p
	\subseteq R[T]
	\quad\text{and}\quad
	\overline{\R}:=\overline{\R(\abul)}
	\]
	be the integral closure of $\R$ in its total ring of fractions.
	Since $\R$ is finitely generated over the excellent Noetherian domain $R$, it is Noetherian,
	and $\overline{\R}$ is a finite $\R$--module; in particular
	$\overline{\R}$ is again a finitely generated $R$--algebra.
	
	Let $\pi:\overline{Y}\to \Spec R$ the induced proper birational morphism.
	Write $E\subseteq \overline{Y}$ for the exceptional locus of $\pi$.
	Since $\overline{Y}$ is of finite type over $R$, the exceptional locus has only finitely many
	irreducible components of codimension one; denote them by
	\[
	E_1,\dots,E_r.
	\]
	Each $E_i$ defines a divisorial valuation $v_i$ of $\Frac(R)$ by
	\[
	v_i(h):=\ord_{E_i}(\pi^*h)\quad (0\ne h\in \Frac(R)),
	\]
	normalized so that $v_i(R)\ge 0$ and $v_i(\m)>0$ whenever $\pi(E_i)\subseteq V(\m)$.
	
	We claim that these finitely many valuations test the integral closure of each
	$\a_p$ in the sense that
	\begin{equation}\label{eq:ic-test-by-Ei}
		\overline{\a_p}
		=
		\{\,h\in R:\ v_i(h)\ge v_i(\a_p)\ \text{for all } i=1,\dots,r\,\}.
	\end{equation}
	Indeed, for $h\in R$ we have
	\[
	h\in \overline{\a_p}
	\ \Longleftrightarrow\
	hT^p \in \overline{\a_pT^p}\subseteq \overline{\R}_p
	\ \Longleftrightarrow\
	hT^p \in \overline{\R},
	\]
	where $\overline{\R}_p$ denotes the degree-$p$ piece of $\overline{\R}$.
	Now, $\overline{\R}$ is a normal Noetherian domain, so a Krull domain; therefore,
	\[
	\overline{\R}
	=
	\bigcap_{\substack{\mathfrak P\in \Spec \overline{\R}\\ \mathrm{ht}(\mathfrak P)=1}}
	(\overline{\R})_{\mathfrak P}
	\ \subseteq\ \Frac(\overline{\R}).
	\]
	Consequently, $hT^p\in \overline{\R}$ if and only if $hT^p$ lies in each DVR
	$(\overline{\R})_{\mathfrak P}$ for $\mathrm{ht}(\mathfrak P)=1$.
	Among the height-one primes $\mathfrak P$, only those corresponding to the codimension-one
	components of the exceptional locus contribute nontrivially to membership in $\a_p$; the
	corresponding valuations are precisely $v_1,\dots,v_r$, and the resulting condition is
	exactly \eqref{eq:ic-test-by-Ei}.  This proves the claim.
	
	We set
	\[
	\mathcal V(\abul):=\{v_1,\dots,v_r\}.
	\]
		
		\smallskip
		\noindent\textit{Step 2: computation of $\Loj_{\bbul}(\abul)$.}
		Fix any filtration $\bbul=\{\b_q\}_{q\ge 1}$ of ideals in $R$.
		Set
		\[
		\lambda\ :=\ \max_{1\le i\le r}\ \frac{v_i(\abul)}{v_i(\bbul)}\ \in\ \RR_{\ge 0}.
		\]
	
		Let $\frac{q}{p}\in\QQ_{>0}$ be admissible in Definition~\ref{def:L-families}, i.e.,
		\[
		\b_{qt}\subseteq \overline{\a_{pt}}
		\quad\text{for all integers }t\gg 1.
		\]
		By \eqref{eq:ic-test-by-Ei} applied to $pt$, this containment is equivalent to
		\[
		v_i(\b_{qt})\ \ge\ v_i(\a_{pt})
		\quad\text{for all }i=1,\dots,r\ \text{and all }t\gg 1.
		\]
		By the same arguments as in the proof of \Cref{thm:valuative-filtrations}, we get
	$q\,v_i(\bbul)\ge p\,v_i(\abul)$, i.e.,
		\[
		\frac{q}{p}\ \ge\ \frac{v_i(\abul)}{v_i(\bbul)}\quad\text{for all } i.
		\]
	Hence, $\frac{q}{p}\ge \lambda$. Taking the infimum over admissible \(\frac{q}{p}\) yields
		\(\Loj_{\bbul}(\abul)\ge \lambda\).
		
		Conversely, fix $\varepsilon>0$ sufficiently small. 
		By the definition of $v_i(\abul)$ and $v_i(\bbul)$, there exists $p_0$
		such that for all $p\ge p_0$ and all $i$,
		\[
		\frac{v_i(\a_p)}{p}\ \le\ v_i(\abul)+\varepsilon
	\text{ and }
		\frac{v_i(\b_p)}{p}\ \ge\ v_i(\bbul)-\varepsilon > 0.
		\]
    Set $$\delta = \epsilon \cdot \max\limits_{1 \le i \le r}\frac{v_i(\abul) + v_i(\bbul)}{v_i(\bbul)(v_i(\bbul)-\epsilon)}.$$
    Then, $\delta \rightarrow 0$ as $\epsilon \rightarrow 0$ and 
    $(\lambda + \delta)(v(\bbul)-\epsilon) \ge v_i(\abul)+\epsilon.$
		Choose $p\ge p_0$ large and set
		$$q:=\left\lceil(\lambda+\delta)p\right\rceil.$$
		Now choose $t_0$ large enough so that $qt\ge p_0$ for all $i$ and all $t\ge t_0$.
		Then, for every $i$ and every integer $t\ge t_0$ we have
		\[
		v_i(\b_{qt})\ \ge\ qt\,(v_i(\bbul)-\varepsilon)
		\ \ge\ (\lambda+\delta)pt\,(v_i(\bbul)-\varepsilon)
		\ \ge\ pt\,(v_i(\abul)+\varepsilon)
		\ \ge\ v_i(\a_{pt}).
		\]
		Hence, by \eqref{eq:ic-test-by-Ei},
		\[
		\b_{qt}\subseteq \overline{\a_{pt}}
		\quad\text{for all integers }t\ge t_0.
		\]
		Thus, $\frac{q}{p}$ is admissible in Definition~\ref{def:L-families}, so
		\[
		\Loj_{\bbul}(\abul)\ \le\ \frac{q}{p}\ \le\ \lambda+\delta+\frac{1}{p}.
		\]
		Letting $p\to\infty$ and then $\varepsilon\to 0$ ($\delta \to 0$) gives $\Loj_{\bbul}(\abul)\le \lambda$.
		The theorem is proved.
	\end{proof}

\begin{remark} \label{rmk:finite-testing-fingen} \quad
	\begin{enumerate} 
	\item The finite set $\mathcal V(\abul)$ depends only on the filtration $\abul$
	and not on $\bbul$.
	When $\abul=\{\a^p\}_{p\ge1}$ for an $\m$--primary ideal $\a$, the set
	$\mathcal V(\abul)$ coincides with the classical set of Rees valuations
	of $\a$.
	\item The condition that $v_i(\bbul) > 0$ in \Cref{thm:finite-testing-fingen} is often satisfied in practice. For instance, if $\bbul$ is \emph{linearly bounded} from below, i.e., there exists a constant $c$ such that
	$\m^{\lceil cq \rceil} \subseteq \b_q.$
	Particularly, if $\bbul = \{\m^q\}_{q \in \NN}$, then the hypothesis $v_i(\bbul) > 0$ in \Cref{thm:finite-testing-fingen} is automatically satisfied.
	\item The conclusion of \Cref{thm:finite-testing-fingen} still holds under a slightly weaker condition that $v_i(\bbul) = 0$ implies $v_i(\abul) > 0$ for any $i =1, \dots, r$. This is satisfied particularly if $v_i(\abul) > 0$ for all $i = 1, \dots, r$.
	\end{enumerate}
\end{remark}

\begin{corollary}[Finite testing set from finite generation]
	\label{prop:finite-testing-fingen}
	Let $(R,\m)$ be an excellent Noetherian local domain, and let
	$\abul=\{\a_p\}_{p\ge1}$ be a filtration of $\m$-primary ideals whose Rees algebra
	$\R(\abul)$ is finitely generated.
	Then, there exists a finite set of divisorial valuations $\mathcal V$
	with the following property:
	
	For any filtration $\bbul=\{\b_q\}_{q\ge1}$, with $v(\bbul) > 0 \ \forall \ v \in \mathcal V$, and any $p,q\ge1$,
	\[
	\b_q\subseteq \overline{\a_p}
	\quad\Longleftrightarrow\quad
	v(\b_q)\ge v(\a_p)
	\ \text{for all } v\in\mathcal V.
	\]
	In particular, $\mathcal V$ is a finite testing set for asymptotic
	containment problems involving $\abul$.
\end{corollary}

\begin{proof}
	By the proof of \Cref{thm:finite-testing-fingen}, there exists a finite set of
	divisorial valuations $\mathcal V(\abul)$ computing
	$\Loj_{\bbul}(\abul)$ for all filtrations $\bbul$.
	The valuative criterion for integral closure then implies that the same
	set of valuations detects all containments
	$\b_q\subseteq\overline{\a_p}$, proving the claim.
\end{proof}
	
\begin{example}\label{ex:finite-testing}
	Let $R := k[x,y]_{(x,y)}, \ \m := (x,y)$,
	and let
	\[
	\a_1 := (x^4,\,x^2y^3,\,y^5)\subseteq R.
	\]
	Consider
	$\abul := \{\a_p\}_{p\ge 1} \quad\text{with } \a_p=\a_1^p,
	\quad
	\bbul := \{\b_q\}_{q\ge 1} \quad\text{with } \b_q=\m^q.$
	
	Since the Rees algebra \(R(\abul)\) is finitely generated, \Cref{prop:finite-testing-fingen} applies. Hence, there exists a finite
	set of divisorial valuations \(V(\abul)\) such that, for all \(p,q\ge 1\),
	\[
	\b_q \subseteq \overline{\a_p}
	\quad\Longleftrightarrow\quad
	v(\b_q)\ge v(\a_p)\ \text{for all } v\in V(\abul),
	\]
	and
	\[
	\Loj_{\bbul}(\abul)
	=\max_{v\in V(\abul)}\frac{v(\abul)}{v(\bbul)}.
	\]
	
	In this two--dimensional monomial setting, the valuations in \(V(\abul)\) admit an
	explicit description: they are precisely the divisorial (toric) valuations
	corresponding to the compact edges of the Newton polygon of \(\a_1\). Concretely,
	one may take the monomial valuations \(v_{(1,1)}\) and \(v_{(2,1)}\), defined by
	\[
	v_{(u,v)}(x)=u, \quad v_{(u,v)}(y)=v.
	\]
	A direct computation gives $v_{(1,1)}(\abul)=4, v_{(2,1)}(\abul)=5$, and $v_{(1,1)}(\bbul)=v_{(2,1)}(\bbul)=1.$	Consequently,
	\[
	\Loj_{\bbul}(\abul)=5.
	\]
	
	This illustrates concretely how the valuative supremum reduces to a finite maximum
	on a fixed normalized blowup.
\end{example}


\subsection{From supremum to maximum: normalized valuation space and linearly boundedness}
We explain how compactness of normalized valuation spaces yields attainment of
the valuative supremum under linear boundedness hypotheses.

Let $\Val_{\m}$ denote the set of real (semi)valuations $v$ of $R$ centered at $\m$.
Fix a nonzero $\m$-primary ideal $\q\subseteq R$ and consider the normalized slice
\[
\Val_{\m}(\q):=\{v\in \Val_{\m}: v(\q)=1\}.
\]

\begin{lemma}\label{lem:compact-slice}
	Let $(R,\m)$ be a Noetherian local ring and let $\q\subseteq R$ be an $\m$--primary ideal.
	Then, $\Val_{\m}(\q)$ is compact in the topology of pointwise convergence on $R$. Moreover, for any graded family $\abul$ of ideals in $R$, the function
	\[
	v \longmapsto v(\abul)
	\]
	is upper semicontinuous on\/ $\Val_{\m}(\q)$. 
\end{lemma}

\begin{proof}
By the discussion following \cite[Definition~1.3]{BFJ2008} (see also \cite[Proposition 5.9]{JM12} and the discussion that followed), the normalized space $\Val_{\m}(\m)$
is compact in the topology of pointwise convergence on $R$.

Note that,  by \cite[Lemma 4.1]{JM12}, $v\mapsto v(\m)$ and $v\mapsto v(\q)$ are continuous on $\Val_{\m}$.
Define
\begin{align*}
\Phi: & \Val_{\m}(\q)\longrightarrow \Val_{\m}(\m),\quad \Phi(v):=\frac{1}{v(\m)}\,v, \text{ and } \\
\Psi: & \Val_{\m}(\m)\longrightarrow \Val_{\m}(q),\quad \Psi(w):=\frac{1}{w(\q)}\,w.
\end{align*}
These maps are well-defined since $\q$ is $\m$-primary, so $v(\m)>0$ and $v(\q)>0$ for
all $v\in\Val_{\m}$. Moreover, for each $g\in R$ one has
\[
(\Phi(v))(g)=\frac{v(g)}{v(\m)},\quad (\Psi(w))(g)=\frac{w(g)}{w(\q)},
\]
so $\Phi$ and $\Psi$ are continuous. It is immediate that $\Phi$ and $\Psi$ are inverse
to each other. Hence, $\Val_{\m}(\q)$ is homeomorphic to the compact space $\Val_{\m}(\m)$
and, therefore, $\Val_{\m}(\q)$ is compact.

Finally, by definition of the topology on $\Val_{\m}(\q)$, the map $v\mapsto v(g)$ is
continuous for every $g\in R$. Thus, for a finitely generated ideal $I\subseteq R$, the map
$v\mapsto v(I)$ is continuous. In particular,
\[
v(\a_\bullet)=\inf_{p\ge 1}\frac{v(\a_p)}{p}
\]
is an infimum of continuous functions. It follows that $v\mapsto v(\a_\bullet)$ is upper
semicontinuous on $\Val_{\m}(\q)$.
\end{proof}

\begin{theorem}[Attainment and valuative equality for linearly bounded filtrations]
\label{thm:attainment}
	Let $R$ be a Noetherian local domain.
	Let $\abul,\bbul$ be filtrations such that there exists an $\m$-primary ideal
	$\q$ and constants $c,C,d,D>0$ with
	\[
	\q^{\lceil cp\rceil}\subseteq \a_p\subseteq \q^{\lfloor Cp\rfloor}
	\quad\text{and}\quad
	\q^{\lceil dp\rceil}\subseteq \b_p\subseteq \q^{\lfloor Dp\rfloor}
	\quad\text{for all }p\ge1.
	\]
	Then, the supremum in the valuative formula
	\[
	\sup_{v\in\Val_{\m}(\q)}\frac{v(\abul)}{v(\bbul)}
	\]
	is finite and is attained by some valuation $v\in\Val_{\m}(\q)$.

	Suppose, in addition, that the functions
	\[
	v\longmapsto \frac{v(\a_p)}{p}
	\quad\text{and}\quad
	v\longmapsto \frac{v(\b_p)}{p}
	\]
	converge \emph{uniformly} on $\Val_{\m}(\q)$ to
	$v(\abul)$ and $v(\bbul)$, respectively.
	Then,
	\[
	\Loj_{\bbul}(\abul)
	\ =\
	\sup_{v\in\Val_{\m}(\q)}\frac{v(\abul)}{v(\bbul)},
	\]
	and the right-hand side is attained by some valuation in $\Val_{\m}(\q)$.
\end{theorem}

\begin{proof}
	By the linear boundedness assumptions,
	\[
	v(\abul)\le c\,v(\q),\quad v(\bbul)\ge D\,v(\q)
	\]
	for all $v\in\Val_{\m}$. In particular, on the compact slice $\Val_{\m}(\q)$ one has
	$v(\q)=1$ and, therefore, $v(\abul) \le c$ and $v(\bbul)\ge D >0$.
	Set
	\[
	s:=\sup_{v\in\Val_{\m}(\q)}\frac{v(\abul)}{v(\bbul)}<\infty.
	\]

By definition, there exists a sequence $\{v_j\}_{j \in \NN} \subseteq \Val_{\m}(\q)$ such that $s = \lim\limits_{j \rightarrow \infty} \frac{v_j(\abul)}{v_j(\bbul)}.$ Since $\Val_{\m}(\q)$ is compact, by \Cref{lem:compact-slice}, a subsequence $v_{j_k}$ of this sequence converges to an element $v_* \in \Val_{\m}(\q)$. The upper semicontinuity in \Cref{lem:compact-slice} then gives
$$v_*(\abul) \ge \limsup_{k\rightarrow \infty} v_{j_k}(\abul) \text{ and } v_*(\bbul) \ge \limsup_{k \rightarrow \infty} v_{j_k}(\bbul) \ge D.$$

For any $\eta > 0$, set
$$E_\eta = \{v \in \Val_{\m}(\q) : v(\bbul) \ge v_*(\bbul) - \eta\}.$$
It follows from \Cref{lem:compact-slice}, that $v \mapsto v(\bbul)$ is an upper semicontinuous function. This implies that $E_\eta$ is a closed subset or, equivalently, its complement in $\Val_{\m}(\q)$ is an open subset. We then conclude, since $v_{j_k} \longrightarrow v_*$, that $v_{j_k} \in E_\eta$ for $k \gg 1$. Therefore, for $k \gg 1$, 
$$s = \lim_{k \rightarrow \infty} \frac{v_{j_k}(\abul)}{v_{j_k}(\bbul)} \le \frac{\limsup_{k \rightarrow \infty} v_{j_k}(\abul)}{v_*(\bbul) - \eta} \le \frac{v_*(\abul)}{v_*(\bbul)-\eta}.$$

This inequality holds for any $\eta > 0$. Hence, by taking $\eta \rightarrow 0^+$, we have 
$$s \le \frac{v_*(\abul)}{v_*(\bbul)} \le s.$$
This forces the desired equality $s = \frac{v_*(\abul)}{v_*(\bbul)}.$

	Assume now that $\frac{v(\a_p)}{p}\to v(\abul)$ and $\frac{v(\b_p)}{p}\to v(\bbul)$ uniformly on
	$\Val_{\m}(\q)$. Let $\lambda>s$. By definition of $s$, we have
	$v(\abul)\le s\,v(\bbul)$ for all $v\in\Val_{\m}(\q)$.
	Thus, for all $v\in\Val_{\m}(\q)$,
	\[
	\lambda v(\bbul)-v(\abul)\ \ge\ (\lambda-s)\,v(\bbul)\ \ge\ (\lambda-s)d\ =:\ \delta.
	\]
	Choose $\varepsilon>0$ with $\varepsilon<\min\{D/2,\ \delta/(\lambda+1)\}$.
	Then, for every $v\in\Val_{\m}(\q)$ one has $v(\bbul)-\varepsilon>0$ and
	\[
	v(\abul)+\varepsilon\le \lambda v(\bbul)-\delta+\varepsilon
	= \lambda(v(\bbul)-\epsilon) - \underbrace{(\delta - (\lambda+1)\epsilon)}_{=:\delta' > 0},
	\]
	so
	\[
	\frac{v(\abul)+\varepsilon}{v(\bbul)-\varepsilon} \le \lambda - \frac{\delta'}{v(\bbul)-\epsilon} \le \lambda - \frac{\delta'}{d-\epsilon} <\lambda\quad\text{for all }v\in\Val_{\m}(\q).
	\]
	Set
	\[
	M:=\sup_{v\in\Val_{\m}(\q)}\frac{v(\abul)+\varepsilon}{v(\bbul)-\varepsilon},
	\]
	so $M \le \lambda - \frac{\delta'}{d-\epsilon} <\lambda$.

	By uniform convergence, there exists $N$ such that for all $p\ge N$ and all $v\in\Val_{\m}(\q)$,
	\[
	\left|\frac{v(\a_p)}{p}-v(\abul)\right|\le \varepsilon
	\quad\text{and}\quad
	\left|\frac{v(\b_p)}{p}-v(\bbul)\right|\le \varepsilon.
	\]
	Choose integers $P\ge N$ and $Q\ge N$ with $M\le \frac{Q}{P}<\lambda.$

		\smallskip
		\noindent
		\emph{Claim.} For all integers $t \ge 1$ one has
		\[
		\b_{Qt}\subseteq \overline{\a_{Pt}}.
		\]
		
		\smallskip
		\noindent
		\emph{Proof of the claim.}
		For $t \ge 1$, $Pt\ge N$ and $Qt\ge N$, so for all $v\in\Val_{\m}(\q)$,
		\[
		\frac{v(\a_{Pt})}{Pt}\le v(\abul)+\varepsilon,
		\quad
		\frac{v(\b_{Qt})}{Qt}\ge v(\bbul)-\varepsilon.
		\]
		Therefore, for every $v\in\Val_{\m}(\q)$ and every $t\ge 1$,
		\[
		\frac{v(\b_{Qt})}{Pt}
		=
		\frac{Q}{P}\cdot \frac{v(\b_{Qt})}{Qt}
		\ \ge\
		\frac{Q}{P}\bigl(v(\bbul)-\varepsilon\bigr)
		\ \ge\
		M\bigl(v(\bbul)-\varepsilon\bigr)
		\ \ge\
		v(\abul)+\varepsilon
		\ \ge\
		\frac{v(\a_{Pt})}{Pt}.
		\]
		Multiplying by $Pt$ gives
		\[
		v(\b_{Qt})\ge v(\a_{Pt})\quad\text{for all }v\in\Val_{\m}(\q)\text{ and all }t\ge 1.
		\]
		Now let $w$ be any divisorial valuation centered at $\m$. Since $w(\q)>0$, the normalized valuation
		$\tilde w:=w/w(\q)$ lies in $\Val_{\m}(\q)$, so the previous inequality applied to $\tilde w$ yields
		$w(\b_{Qt})\ge w(\a_{Pt})$ for all $t\ge 1$.
		By the valuative criterion for integral closure (\Cref{thm:valuative-ic}), this implies
		$\b_{Qt}\subseteq\overline{\a_{Pt}}$ for all $t\ge 1$, proving the claim.
		\hfill$\square$
		
		\smallskip
		By the claim, $\frac{Q}{P}$ is admissible $\Loj_{\bbul}(\abul)$, and so
		\[
		\Loj_{\bbul}(\abul)\le \frac{Q}{P}<\lambda.
		\]
		Since $\lambda>s$ was arbitrary, we obtain $\Loj_{\bbul}(\abul)\le s$.
		
		On the other hand, \Cref{thm:valuative-filtrations} gives $\Loj_{\bbul}(\abul)\ge s$.
		Thus, $\Loj_{\bbul}(\abul)=s$, and the right-hand side is attained by the first part of the theorem.
	\end{proof}
	
	\begin{example}\label{ex:nonpower-linearly-bounded}
	Let $(R,\m)=k[[x,y]]$ and $\m = (x,y)$.
		Define two filtrations $\abul=\{\a_p\}_{p\ge1}$ and $\bbul=\{\b_p\}_{p\ge1}$ by
		\[
		\a_p:=\big(x^iy^j \in R \mid i+2j \ge 3p\big),
		\quad
		\b_p:=\big(x^iy^j \in R \mid 2i+j \ge 5p\big).
		\]
It can be seen that, for $p\ge1$, $\a_p$ is integrally closed, and
		\[
		\m^{3p}\subseteq \a_p \subseteq \m^{p},
		\quad
		\m^{5p}\subseteq \b_p \subseteq \m^{p}.
		\]
	It can also be checked that the functions $v \mapsto v(\a_p)/p$ and $v \mapsto v(\b_p)/p$ converge uniformly to $v(\abul)$ and $v(\bbul)$, respectively, on the space of normalized real valuations of $\Frac(R)$ centered at $\m$.
	
		We claim that
		\[
		\Loj_{\bbul}(\abul)=\frac{6}{5}.
		\]
Indeed, fix $p,q$ with $q/p\ge 6/5$ and let $t\ge1$.
		For any monomial $x^iy^j\in \b_{qt}$ we have $2i+j\ge 5qt$.
		Minimizing $i+2j$ subject to $2i+j\ge 5qt$ and $i,j\ge0$ yields
		\[
		i+2j \ge \frac{5}{2}qt \ge \frac{5}{2}\cdot\frac{6}{5}pt = 3pt,
		\]
		and so $x^iy^j\in \a_{pt}$.
		Thus, $\b_{qt}\subseteq \a_{pt} = \overline{\a_{pt}}$ for all $t$, showing $\Loj_{\bbul}(\abul)\le 6/5$.
		
Conversely, assume $q/p<6/5$ and set $m_t=\lceil\frac{5}{2}qt\rceil$.
		Then, $x^{m_t}\in \b_{qt}$ since $2m_t\ge 5qt$.
		However, $x^{m_t}\in \a_{pt}$ would require $m_t\ge 3pt$, which fails for all $t\gg1$
		since $\frac{5}{2}q<3p$.
		Hence, $\b_{qt}\nsubseteq \a_{pt} = \overline{\a_{pt}}$ for $t\gg1$, proving $\Loj_{\bbul}(\abul)\ge 6/5$, giving the claimed value of $\Loj_{\bbul}(\abul)$.
		
		Furthermore, consider the monomial valuation $v$ defined by $v(x)=1$, $v(y)=2$.
		Then, $v(\m)=1$, $v(\a_\bullet)=3$, and $v(\b_\bullet)=5/2$, so
		\[
		\frac{v(\a_\bullet)}{v(\b_\bullet)}=\frac{6}{5}=\Loj_{\bbul}(\abul).
		\]
		Thus, the valuative supremum is attained by $v$, as predicted by Theorem~\ref{thm:attainment}.
	\end{example}
	

The following theorem shows that the valuation attaining supremum value in \Cref{thm:attainment} is not necessarily divisorial, indicating that focusing on divisorial valuations may not be sufficient to study \L{}ojasiewicz exponents in general.

\begin{theorem}[Non-divisorial attainment in dimension two]
	\label{thm:nondivisorial-attainment-final}
	Let $R=\CC[[x,y]]$ with maximal ideal $\m=(x,y)$, and set $\q=\m$.
	Then, there exist filtrations $\abul=\{\a_p\}_{p\ge1}$ and $\bbul=\{\b_p\}_{p\ge1}$
	satisfying the linear boundedness hypotheses of Theorem~\ref{thm:attainment}
	such that
	\[
	\sup_{v\in\Val_{\m}(\q)}\frac{v(\abul)}{v(\bbul)}
	\]
	is attained by a valuation $v_*\in\Val_{\m}(\q)$ which is \emph{not} divisorial,
	and moreover \emph{no divisorial valuation} in $\Val_{\m}(\q)$ attains this supremum.
\end{theorem}

\begin{proof}
We work on the normalized slice $\Val_{\m}(\m)=\Val_{\m}(\q)$ so that
$v(\m)=1$ for all $v\in\Val_{\m}(\m)$. We use the notation and terminology
of Favre--Jonsson monograph \cite{FJ04}.

\smallskip
\noindent\textit{Step 1: a non-divisorial end and divisorial approximants.}
By \cite[Prop.~A.3]{FJ04} there exist \emph{infinitely singular} valuations of
finite \emph{skewness}. Fix such a valuation $\nu\in\Val_{\m}(\m)$ and set
$A:=\alpha(\nu)\in(1,\infty)$.  
Choose its approximating sequence of divisorial valuations
$\nu_1\prec\nu_2\prec\cdots\prec\nu$ as in
\cite[Prop.~3.44, Def.~3.45]{FJ04}, and write $A_i:=\alpha(\nu_i)$.
Then,$A_i\uparrow A$.  

For $v\in\Val_{\m}(\m)$ set $\beta(v):=\alpha(v\wedge\nu)$.
Since $\nu_i$ lies on the segment $[\ord_{\m},\nu]$, the skewness
parameterization of segments (\cite[Thm.~3.26]{FJ04}) gives
\[
\alpha(v\wedge\nu_i)=\min\{\beta(v),A_i\}.
\]

\smallskip
\noindent\textit{Step 2: simple complete ideals attached to the $\nu_i$.}
For each $i$, let $b_i:=b(\nu_i)$ be the generic multiplicity of the divisorial
valuation $\nu_i$. Consider the positive \emph{atomic measure} $\rho_i:=b_i\nu_i$ and define
\[
J_i:=I_{\rho_i}=\{\phi\in R\mid g_\phi\ge g_{\rho_i}\},
\]
where $I_\rho$ is as in \cite[(8.3)]{FJ04}. By \cite[Thm.~8.12]{FJ04},
$J_i$ is an integrally closed ideal and its \emph{tree measure} satisfies 
$$\rho_{J_i}=\rho_i.$$
Together with \cite[Prop. 8.9]{FJ04} and the fact that the maps $I \mapsto \rho_I$ and $\rho \mapsto I_\rho$ are inverse semigroup isomorphisms between integrally closed ideals and measures in $\mathcal{M}_{\mathcal{I}}^+$ (\cite[Thm. 8.12]{FJ04}), this also implies that 
$g_{J_i} = g_{\rho_i}$ as the corresponding positive tree potential. In particular, by \cite[Thm. 7.48 and (8.1)]{FJ04}, for all $v \in \Val_{\m}(\m)$,
$$v(J_i) = g_{J_i}(v) = g_{\rho_i}(v) = \int \alpha(\mu\wedge v)\,d\rho_i(\mu) = b_i \alpha(v \wedge \nu_i).$$

Furthermore, by the unique factorization of $J_i$, as described in \cite[Thm. 8.12]{FJ04}, and the fact that $\rho_{J_i} = b_i \nu_i$, it follows that this factorization of $J_i$ has only one divisorial atom and no curve atoms. Therefore, $J_i$ is the simple complete ideal attached to the divisorial valuation $\nu_i$ and, in particular, is $\m$-primary. Particularly, there exists $\ell_i\ge 1$ such that
\[
\m^{\ell_i}\subseteq J_i.
\]

\smallskip
\noindent\textit{Step 3: integer weights.}
Choose integers $n_i\ge1$ inductively so that
$D_i:=\lfloor n_i b_i A_i\rfloor$ is strictly increasing and
\[
c_i := \frac{n_i b_i A_i}{D_i}\longrightarrow 1.
\]
Set $s_i(v):=v(J_i^{n_i})/D_i$. Using the previous formula for $v(J_i)$,
\[
s_i(v)=\frac{n_i b_i}{D_i}\,\alpha(v\wedge\nu_i)
      =\frac{n_i b_i}{D_i}\,\min\{\beta(v),A_i\}
      = c_i \frac{\min\{\beta(v),A_i\}}{A_i}.
\]

We now prove that $s_i(v)\to g(v) :=\frac{\beta(v)}{A}=\frac{\alpha(v\wedge\nu)}{A}$ uniformly on $\Val_{\m}(\m)$.
Indeed, for any $v\in\Val_{\m}(\m)$, we have
\[
\left|s_i(v)-\frac{\beta(v)}{A}\right|
\le
|c_i-1|\frac{\min\{\beta(v),A_i\}}{A_i}
+
\left|
\frac{\min\{\beta(v),A_i\}}{A_i}-\frac{\beta(v)}{A}
\right|.
\]
Since $\frac{\min\{\beta(v),A_i\}}{A_i}\le1$ and $0\le
\frac{\min\{\beta(v),A_i\}}{A_i}-\frac{\beta(v)}{A} \le 1-\frac{A_i}{A},$
we obtain
\[
\sup_{v\in\Val_{\m}(\m)}
\left|s_i(v)-g(v)\right|
\le
|c_i-1|+\left(1-\frac{A_i}{A}\right).
\]
Because $c_i\to1$ and $A_i\uparrow A$, the right-hand side tends to $0$,
so $s_i(v)\to g(v)$ uniformly on $\Val_{\m}(\m)$.

\smallskip
\noindent\textit{Step 4: the auxiliary filtration $\fc_\bullet$.}
For $p\ge1$ define
\[
\fc_p:=\sum_{\sum_i m_iD_i\ge p}\prod_i J_i^{n_i m_i},
\]
the sum taken over finitely supported tuples $(m_i)$.
It can be seen that $\fc_\bullet$ is a filtration of ideals.

Set $\bbul:=\m^\bullet$ and $\a_p:=\m^p\fc_p$.  
Then, $\abul$ is also a filtration of ideals.

\smallskip
\noindent\textit{Step 5: computation of $v(\fc_\bullet)$.}
For any monomial summand $M=\prod_iJ_i^{n_im_i}$ with $\sum_im_iD_i\ge p$,
\[
v(M)=\sum_im_iD_i\,s_i(v)\ge p\,\inf_i s_i(v).
\]
Conversely $J_i^{n_i\lceil p/D_i\rceil}$ occurs among the summands of
$\fc_p$, so
$v(\fc_p)\le\lceil p/D_i\rceil v(J_i^{n_i})$.  
We shall show that
\begin{equation}\label{eq:si-lower}
s_i(v)\ge g(v) \text{ for all }i \text{ and all }v.
\end{equation}
Indeed, observe that $s_i(v)\ge \frac{\min\{\beta(v),A_i\}}{A_i}.$
If $\beta(v)\le A_i$, then
\[
\frac{\min\{\beta(v),A_i\}}{A_i}=\frac{\beta(v)}{A_i}\ge \frac{\beta(v)}{A}=g(v),
\]
whereas if $\beta(v)\ge A_i$, then
\[
\frac{\min\{\beta(v),A_i\}}{A_i}=1\ge \frac{\beta(v)}{A}=g(v).
\]
This proves \eqref{eq:si-lower}.

Next, as shown above, $s_i(v)\longrightarrow g(v)$
uniformly on $\Val_{\m}(\m)$. Therefore, together with \eqref{eq:si-lower}, we have
\[
\inf_i s_i(v)=g(v) \text{ for all }v.
\]

Now, it can be seen that $v(\fc_p)\ge pg(v)$. On the other hand, for each fixed $i$, $J_i^{n_i\lceil p/D_i\rceil}\subseteq \fc_p,$ so $v(\fc_p)\le \left\lceil \frac{p}{D_i}\right\rceil v(J_i^{n_i}),$
and therefore,
\[
\frac{v(\fc_p)}{p}\le s_i(v)+\frac{n_ib_iA_i}{p}.
\]

We shall prove 
\[
\lim_{p\to\infty}\frac{v(\fc_p)}{p}=g(v)
\]
uniformly on $\Val_{\m}(\m)$. Let $\varepsilon>0$.
By the uniform convergence of $s_i$ to $g$ proved before, choose $i_0$ such that
\[
\sup_{v\in\Val_{\m}(\m)} |s_{i_0}(v)-g(v)|<\varepsilon/2.
\]
Since $v(J_{i_0}^{n_{i_0}})\le n_{i_0}b_{i_0}A_{i_0}$ for all $v$,
choose $p_0$ such that
\[
\frac{n_{i_0}b_{i_0}A_{i_0}}{p}<\varepsilon/2 \text{ for all }p\ge p_0.
\]
For such $p$, the lower bound above gives
\[
\frac{v(\fc_p)}{p}\ge g(v),
\]
while the upper bound with $i=i_0$ gives
\[
\frac{v(\fc_p)}{p}
\le
s_{i_0}(v)+\frac{n_{i_0}b_{i_0}A_{i_0}}{p}
\le
g(v)+\varepsilon
\]
for all $v\in\Val_{\m}(\m)$.
Thus,
\[
\sup_{v\in\Val_{\m}(\m)}
\left|
\frac{v(\fc_p)}{p}-g(v)
\right|
\le\varepsilon
\]
for all $p\ge p_0$, establishing the desired uniform convergence.

\smallskip
\noindent\textit{Step 6: valuation of $\abul$.}
Since $\a_p=\m^p\fc_p$ and $v(\m)=1$,
\[
\frac{v(\a_p)}p\to1+g(v)=1+\frac{\alpha(v\wedge\nu)}{A}.
\]
Hence, $v(\abul)=1+\alpha(v\wedge\nu)/A$ and $v(\bbul)=1$.

\smallskip
\noindent\textit{Step 7: maximizer.}
At $v=\nu$ this ratio equals $2$.  
If $v\ne\nu$, then $v\wedge\nu\prec\nu$ and by strict monotonicity of
skewness along segments (\cite[Thm.~3.26]{FJ04}),
$\alpha(v\wedge\nu)<A$. It follows that $v(\abul)/v(\bbul)<2=\nu(\abul)/\nu(\bbul)$.
Thus, $\nu$ uniquely maximizes $v(\abul)/v(\bbul)$, and since $\nu$ is not
divisorial, no divisorial valuation attains the supremum.

\smallskip
\noindent\textit{Step 8: linear boundedness.}
Because $\fc_p\subseteq R$, one has $\a_p\subseteq\m^p$.  
Also $J_1^{n_1\lceil p/D_1\rceil}\subseteq\fc_p$ and
$\m^{\ell_1}\subseteq J_1$, so
$\m^{p+n_1\ell_1\lceil p/D_1\rceil}\subseteq\a_p$.  
Thus $\m^{\lceil cp\rceil}\subseteq\a_p\subseteq\m^p$ for some $c>0$, and
$\abul$ is linearly bounded with respect to $\q=\m$.
\end{proof}

\subsection{Finite maximizers in the ideal case}
In the ideal case, i.e., when the two graded families of ideals are those of ordinary powers of given two ideals, we recover the classical stronger statement that the maximum is attained by finitely many Rees valuations.

\begin{theorem}[Finite set of maximizers for ideals]\label{thm:finite-max-rees}
	Let $(R,\m)$ be an excellent Noetherian local domain.
	Let $\a$ be an ideal and let $\b$ be any proper ideal with $\b\neq 0$.
	Then,
	\[
	\Loj_{\b}(\a)=\sup_{\substack{\text{divisorial valuation } \\ v \text{ centered at } \m}} \frac{v(\a)}{v(\b)}
	\]
	and, if finite, is a maximum and achieved by one of the Rees valuations of $\a$.
	In particular, if finite, $\Loj_{\b}(\a)$ is computed by finitely many divisorial valuations.
\end{theorem}

\begin{proof} This is a classical consequence of Rees's theory of valuations (cf. \cite{HS,SwansonReesValuations}). It also follows from
		\Cref{thm:finite-testing-fingen} as we shall illustrate below.
		
		Since $\R(\a)=\bigoplus_{p\ge 0}\a^pT^p$ is the usual Rees algebra of $\a$, the
		hypothesis of finite generation is automatic.  Hence, by \Cref{thm:finite-testing-fingen}, there exists a finite set of
		divisorial valuations $\mathcal V(\a)=\{v_1,\dots,v_r\}$, arising from prime divisors on
		the normalization of $\Proj \overline{\R(\a)}$, such that for any filtration
		$\bbul$, with $v_i(\bbul) > 0$, one has
		\[
		\Loj_{\bbul}(\a^\bullet)
		=
		\max_{1\le i\le r}\frac{v_i(\a)}{v_i(\bbul)}.
		\]
		Specializing to $\bbul=\{\b^q\}_{q\ge 1}$, noticing that $\b \subseteq \m$ so $\bbul$ is linearly bounded from below (as in \Cref{rmk:finite-testing-fingen}), yields the desired finite--maximum formula.
		
		Finally, the valuations coming from the codimension--one components of the exceptional
		locus of the normalized blowup $\Proj \overline{\R(\a)}\to \Spec R$ are precisely
		the classical Rees valuations of $\a$; see \cite[Chapter 10]{HS}.  Therefore, the maximum above may be taken over the
		Rees valuations of $\a$, as claimed.
\end{proof}

\begin{example}
	\label{ex.finiteNormalizedBlowup}
	Consider $\a = (x^5, x^2y^3, y^7) \subseteq R = \CC[[x,y]]$. The Newton polyhedron of $\a$ has two compact faces with inward normals $(1,1)$ and $(3,1)$. These correspond to Rees valuations $v_{(1,1)}$ and $v_{(3,1)}$ of $\a$. 
	The values of interest are:
	$$v_{(1,1)}(\a) = 5, v_{(3,1)}(\a) = 7, v(\m) = 1.$$
	Hence, $\Loj_{\m}(\a) = 7$ is computed by a Rees valuations of $\a$.
\end{example}


\part{Structural Inequalities and Dualities} \label{part:structuralInequality}

In this part, we derive structural inequalities for
\L{}ojasiewicz exponents and related invariants as consequences of the valuative
framework developed in Part \ref{part:valuative}.


\section{Mixed multiplicities, Jacobian polygons, and Teissier--type inequalities}\label{sec:teissier}

This section links \L{}ojasiewicz-type containment thresholds to mixed multiplicities and Teissier’s convex-geometric inequalities.  We restate Teissier--Rees--Sharp/Alexandrov--Fenchel type bounds for mixed multiplicities of ideals (\Cref{thm:AF-mixed}).  We then introduce mixed \L{}ojasiewicz ratios and establish monotonicity/concavity properties that organize these inequalities (\Cref{prop:Lambda-monotone}, \Cref{prop:jac-concave}).  Finally, we specialize to Jacobian ideals and the algebraic avatars of Milnor and Tjurina numbers, obtaining Teissier-type bounds for the gradient \L{}ojasiewicz exponent (\Cref{thm:grad-mixed-lower}), and we close with an extension principle that generalizes statements for ideals to filtrations.

\subsection{Hilbert--Samuel multiplicity and mixed multiplicities}\label{subsec:mixedmult-prep}

Throughout this section, $(R,\m)$ is a Noetherian local ring of dimension $d$. Assume, in addition, that $\dim \text{Nil}(\widehat{R}) < d$, where $\text{Nil}(\cdot)$ represents the nil-radical of the ring (see \cite{CutkoskyAM}). This condition is satisfied, for instance, if $R$ is analytically unramified.

We recall the mixed multiplicities of ideals.
Let $I_1,\dots,I_r\subseteq R$ be ideals such that $I_1 + \dots + I_r$ is $\m$--primary (equivalently, $I_1^{n_1} \cdots I_r^{n_r}$ has finite colength for all large $(n_i)$).
Then, the function
\[
(n_1,\dots,n_r)\longmapsto \len\!\left(R/(I_1^{n_1}\cdots I_r^{n_r})\right)
\quad (n_i\in\NN)
\]
agrees with a polynomial of total degree $d$ for all $n_i\gg 0$.
The degree--$d$ part of this polynomial can be uniquely written as
\[
\sum_{a_1+\cdots+a_r=d}
\frac{e(I_1^{[a_1]},\dots,I_r^{[a_r]})}{a_1!\cdots a_r!}\,
n_1^{a_1}\cdots n_r^{a_r},
\]
where the coefficients $e(I_1^{[a_1]},\dots,I_r^{[a_r]})\in \NN$ are the \emph{mixed multiplicities} of $I_1, \dots, I_r$.

We will primarily use the two-ideal case $(I,J)$, where we write
$
e(I^{[i]},J^{[d-i]})\quad (0\le i\le d).
$
The following are standard facts:
\begin{enumerate}[label=(\alph*),leftmargin=2em]
	\item (\emph{Scaling}) For $p,q\ge 1$,
	\[
	e((I^p)^{[i]},(J^q)^{[d-i]})=p^{\,i}q^{\,d-i}\,e(I^{[i]},J^{[d-i]}).
	\]
	\item (\emph{Monotonicity}) If $I\subseteq I'$ and $J\subseteq J'$ are $\m$--primary, then
	\[
	e(I^{[i]},J^{[d-i]})\ \ge\ e((I')^{[i]},(J')^{[d-i]})
	\quad \text{for all }i.
	\]
	\item (\emph{Integral closure invariance}) If $\overline I=\overline{I'}$ and
	$\overline J=\overline{J'}$, then
	\[
	e(I^{[i]},J^{[d-i]})=e((I')^{[i]},(J')^{[d-i]})
	\quad \text{for all }i.
	\]
\end{enumerate}

\begin{definition}[Mixed multiplicities of filtrations]
Let $\abul = \{\a_p\}_{p \ge 1}$ and $\bbul =\{\b_q\}_{q\ge 1}$ be graded families of $\m$-primary ideals. For $1 \le i \le d$, set
$$e_i(\abul^{[i]}, \bbul^{[d-i]}) = \limsup_{p \rightarrow \infty} \frac{e(\a_p^{[i]}, \b_p^{[d-i]})}{p^d}.$$
\end{definition}
It was shown by Cid-Ruiz--Monta\~no \cite{CidRuizMontano2022} that, under the condition that $\dim \text{Nil}(\widehat{R}) < d$, $e_i(\abul^{[i]}, \bbul^{[d-i]})$ exists as an actual limit. When $\bbul = \{\m^q\}_{q \ge 1}$, we write $e_i(\abul)$ for $e(\abul^{[i]}, \bbul^{[d-i]})$.

\begin{definition}[Mixed \L{}ojasiewicz rations]
For a graded family $\abul$ of $\m$-primary ideals, set 
$$\Lambda_i(\abul) = \frac{e_i(\abul)}{e_{i-1}(\abul)}, \ 1 \le i \le d.$$
\end{definition}

\subsection{Teissier--Rees--Sharp and Alexandrov--Fenchel inequalities}

We begin with the Alexandrov--Fenchel type inequalities for mixed multiplicities in the form
best suited for graded families (or filtrations).

\begin{theorem}[Teissier--Rees--Sharp inequality]\label{thm:AF-mixed}
Assume that $R$ is quasi--unmixed.
Let $\abul=\{\a_p\}_{p\ge 1}$ and $\bbul=\{\b_q\}_{q\ge 1}$ be graded families such that
$\a_p+\b_q$ is $\m$--primary for all $p,q\ge 1$.
Then, for all $p,q\ge 1$, and $1\le i\le d-1$, we have
\[
e\!\left(\a_{p}^{[i]},\b_{q}^{[d-i]}\right)^2
\ \ge\
e\!\left(\a_{p}^{[i-1]},\b_{q}^{[d-i+1]}\right)\,
e\!\left(\a_{p}^{[i+1]},\b_{q}^{[d-i-1]}\right).
\]
Equivalently, for fixed $p,q$, the sequence $i\ \longmapsto\  e\!\left(\a_{p}^{[i]},\b_{q}^{[d-i]}\right)$
is log-concave. 
\end{theorem}

\begin{proof}
The stated inequality is precisely the Alexandrov-Fenchel type inequality for mixed multiplicities. It was first established by Teissier \cite{Teissier77} in the $\m$-primary case via Minkowski inequalities for multiplicities, and was subsequently extended by Rees and Sharp \cite{ReesSharp78} to the general situation where $\a_p+\b_q$ is $\m$-primary, under the quasi-unmixed hypothesis of $R$.
\end{proof}
\subsection{Containments $\Rightarrow$ mixed-multiplicity inequalities}

Following the foundational work of Teissier on mixed multiplicities, several authors -- most notably Bivià-Ausina and Encinas \cite{Aus09, BE13} -- have related Lojasiewicz-type invariants to mixed multiplicity data, typically through inequalities involving integral closures of powers of ideals. The proposition below isolates the underlying numerical mechanism in a containment-theoretic form.

\begin{proposition}[Containment implies mixed-multiplicity bounds]\label{prop:containment-mixed}
Let $\abul=\{\a_p\}_{p\ge 1}$ and $\bbul=\{\b_q\}_{q\ge 1}$ be $\m$--primary graded families.
Fix $p,q\ge 1$ and assume that
\[
\b_{qt}\ \subseteq\ \overline{\a_{pt}}
\quad\text{for all }t\gg 1.
\]
Then, for every $1\le i\le d$ and every $t\gg 1$, one has
\[
e\!\left((\b_{qt})^{[i]},\m^{[d-i]}\right) \ge  e\!\left((\a_{pt})^{[i]},\m^{[d-i]}\right).
\]
Consequently,
\[
e_i(\bbul)\ \ge\ \left(\frac{p}{q}\right)^{\!i}\,e_i(\abul),
\]
and in particular,
\[
\frac{q}{p}\ \ge\
\left(\frac{e_i(\abul)}{e_i(\bbul)}\right)^{1/i}
\text{ for all } 1\le i\le d.
\]
\end{proposition}

\begin{proof}
Fix $t\gg 1$ such that $\b_{qt}\subseteq \overline{\a_{pt}}$.
By monotonicity and integral-closure invariance standard facts,
\[
e((\b_{qt})^{[i]},\m^{[d-i]})
\ \ge\
e((\overline{\a_{pt}})^{[i]},\m^{[d-i]})
=
e((\a_{pt})^{[i]},\m^{[d-i]}),
\]
which proves the first displayed inequality.

The asymptotic statements is then obtained by dividing both sides of this inequality by $(qt)^i$ and taking the limit as $t\to\infty$.
\end{proof}

\begin{corollary}\label{cor:containment-mixed-ideal}
Let $I,J$ be $\m$--primary ideals in $R$, and let $p,q\ge 1$.
If $J^{\,q}\subseteq \overline{I^{\,p}}$, then for every $1\le i\le d$ one has
\[
q^{\,i}\,e(J^{[i]},\m^{[d-i]})
\ \ge\
p^{\,i}\,e(I^{[i]},\m^{[d-i]}),
\]
equivalently,
\[
\frac{q}{p}\ \ge\ \left(\frac{e(I^{[i]},\m^{[d-i]})}{e(J^{[i]},\m^{[d-i]})}\right)^{1/i}.
\]
\end{corollary}

\begin{proof}
This inequality is a classical consequence of Teissier’s mixed-multiplicity theory and appears in the algebraic form in Rees--Sharp; see \cite{Teissier77,ReesSharp78}. It is also a direct consequence of \Cref{prop:containment-mixed}.
\end{proof}

An immediate consequence of \Cref{cor:containment-mixed-ideal}, taking $i=d$, gives the following lower bound:
\[
\Loj_{J}(I)\ \ge\ \left(\frac{e(I)}{e(J)}\right)^{1/d}.
\]

\subsection{Mixed \L ojasiewicz ratios and Teissier-type monotonicity}

The Alexandrov--Fenchel inequalities imply strong structure for certain ratios of mixed
multiplicities. These ratios play the role of ``mixed slopes'' in the classical theory.

\begin{proposition}[Log-concavity and monotonicity of $\Lambda_i$]\label{prop:Lambda-monotone}
Let $\abul=\{\a_p\}_{p\ge 1}$ be an $\m$--primary graded family.
Then, mixed multiplicity ratios satisfy
\[
\Lambda_1(\abul)\ \ge\ \Lambda_2(\abul)\ \ge\ \cdots\ \ge\ \Lambda_d(\abul).
\]
\end{proposition}

\begin{proof}
Fix $p,t$ and apply \Cref{thm:AF-mixed} to the pair $(\a_{pt},\m)$, we get
\[
e((\a_{pt})^{[i]},\m^{[d-i]})^2
\ \ge\
e((\a_{pt})^{[i-1]},\m^{[d-i+1]})\cdot e((\a_{pt})^{[i+1]},\m^{[d-i-1]}) \ \forall \ 1 \le i \le d-1.
\]
It follows that
Dividing both sides of this inequality by appropriate powers of $pt$ and taking the limit along $p\to\infty$ then give
$\Lambda_1(\abul)\ge \Lambda_2(\abul)\ge\cdots\ge \Lambda_d(\abul)$ as claimed.
\end{proof}

\begin{corollary}[Teissier monotonicity]\label{cor:Lambda-monotone-ideal}
Let $I$ be $\m$--primary. Then, the sequence $(\Lambda_i(I))_{i=1}^d$ is nonincreasing:
\[
\Lambda_1(I)\ \ge\ \Lambda_2(I)\ \ge\ \cdots\ \ge\ \Lambda_d(I).
\]
Equivalently, the sequence $i \longmapsto e(I^{[i]},\m^{[d-i]})$ is log-concave.
\end{corollary}

\begin{proof}
Apply \Cref{prop:Lambda-monotone} to the power family $\abul=\{I^p\}$.  This monotonicity/log-concavity was established by Teissier in \cite{Teissier77}.
\end{proof}

\subsection{Jacobian ideals, Milnor and Tjurina numbers (algebraic viewpoint)}

We now explain how classical ``Jacobian polygon'' data can be encoded algebraically
via mixed multiplicities.

Assume in this subsection that $R=\mathcal O_{\CC^d,0}$ and let
$f\in \m$ define an isolated singularity at the origin.
Let $J(f)\subseteq R$ be the Jacobian ideal generated by $\partial f/\partial x_1,\dots,
\partial f/\partial x_d$.

\begin{definition}[Milnor and Tjurina numbers]\label{def:milnor-tjurina}
Define the \emph{Milnor number} and \emph{Tjurina number} by
\[
\mu(f):=\dim_\CC R/J(f), \quad
\tau(f):=\dim_\CC R/(f,J(f)).
\]
\end{definition}

The next object is the algebraic substitute for Teissier's Jacobian polygon.

\begin{definition}[Jacobian mixed multiplicities and polygon]\label{def:jac-polygon}
For $0\le i\le d$ set
\[
m_i(f):=e\bigl((f)^{[d-i]},\, J(f)^{[i]}\bigr).
\]
The \emph{Jacobian polygon} $\text{JP}(f)$ of $f$ is the lower convex hull in $\RR^2$ of the points
\[
(0,m_0(f)),\ (1,m_1(f)),\ \dots,\ (d,m_d(f)).
\]
\end{definition}

\begin{remark}\label{rem:jac-polygon}
In Teissier's work, the slopes and vertices of the Jacobian polygon encode polar
multiplicities and vanishing rates. In our framework, the same numerical data are
captured by the mixed multiplicities $m_i(f)$, which are defined for any Noetherian
local ring once $(f)+J(f)$ is $\m$--primary.
\end{remark}

\begin{proposition}[Concavity of the Jacobian polygon]\label{prop:jac-concave}
Let $f$ define an isolated singularity. Then, the sequence $i\mapsto \log m_i(f)$ is concave,
equivalently
\[
m_i(f)^2\ \ge\ m_{i-1}(f)\,m_{i+1}(f)\quad (1\le i\le d-1).
\]
\end{proposition}

\begin{proof}
Apply \Cref{thm:AF-mixed} to the pair of ideals $((f),J(f))$, noticing that $(f)+J(f)$ is $\m$--primary.
\end{proof}

\begin{corollary}\label{cor:jac-slope-monotone}
Let $f$ define an isolated singularity, and let $m_i(f)$ be as in
Definition~\ref{def:jac-polygon}.  Then, the sequence of successive ratios
\[
\frac{m_1(f)}{m_0(f)},\ \frac{m_2(f)}{m_1(f)},\ \dots,\ \frac{m_d(f)}{m_{d-1}(f)}
\]
is nonincreasing. Equivalently, for $1\le i\le d-1$ one has
\[
\frac{m_i(f)}{m_{i-1}(f)}\ \ge\ \frac{m_{i+1}(f)}{m_i(f)}.
\]
\end{corollary}

\begin{proof}
By \Cref{prop:jac-concave}, we have
\[
m_i(f)^2\ \ge\ m_{i-1}(f)\,m_{i+1}(f)\quad (1\le i\le d-1).
\]
Dividing both sides by $m_{i-1}(f)\,m_i(f)$ (which is positive) gives
\[
\frac{m_i(f)}{m_{i-1}(f)}\ \ge\ \frac{m_{i+1}(f)}{m_i(f)}. \qedhere
\]
\end{proof}

\subsection{Teissier-type bounds for the gradient \L ojasiewicz exponent}

We now connect $L(f,0) := \Loj^{an}_\m(J(f)) = \Loj_{\m}(J(f))$ to the mixed multiplicity data.
The lower bound in our next result, when $i=d$, recovers the classical estimate in terms of $e(J(f))$, while the full family of inequalities reflects finer mixed-multiplicity data of the Jacobian ideal.

\begin{theorem}[Mixed-multiplicity lower bounds for $L(f,0)$]\label{thm:grad-mixed-lower}
Let $f\in\mathcal O_{\CC^d,0}$ define an isolated singularity.
Then, for each $1\le i\le d$,
\[
L(f,0)
\ \ge\
\left(\frac{e\bigl(J(f)^{[i]},\m^{[d-i]}\bigr)}
{e\bigl(\m^{[i]},\m^{[d-i]}\bigr)}\right)^{1/i}.
\]
In particular,
\[
L(f,0) \ge\ \left(\frac{e(J(f))}{e(\m)}\right)^{1/d}.
\]
\end{theorem}

\begin{proof}
For every $t\gg 1$, the rational number $\frac{\lceil t\,L(f,0)\rceil}{t}$ is admissible for $\Loj_{\m}(J(f))$, and hence
\[
\m^{\lceil t\,L(f,0)\rceil}\ \subseteq\ \overline{J(f)^{\,t}}.
\]
Apply \Cref{cor:containment-mixed-ideal} with $(I,J)=(J(f),\m)$ and $(p,q)=(t,\lceil tL(f,0)\rceil)$.
For each $1\le i\le d$ we obtain
\[
\lceil t\,L(f,0)\rceil^{\,i}\,e(\m^{[i]},\m^{[d-i]})
\ \ge\
t^{\,i}\,e(J(f)^{[i]},\m^{[d-i]}).
\]
Divide by $t^i$ and let $t\to\infty$ to conclude
\[
L(f,0)^{\,i}\,e(\m^{[i]},\m^{[d-i]})
\ \ge\
e(J(f)^{[i]},\m^{[d-i]}),
\]
which is equivalent to the stated lower bound.
\end{proof}

\begin{remark}\label{rem:classical-match}
Classical analytic estimates bounding $L(f,0)$ from below by numerical invariants
(Milnor numbers, polar multiplicities, etc.) can be viewed as instances of
\Cref{thm:grad-mixed-lower} together with identifications between those
invariants and mixed multiplicities; see \cite{Aus09, BE13}.
\end{remark}

\subsection{Extension principle: from ideals to filtrations}

We will restrict to the case when the filtration is asymptotically integral-closure-equivalent to powers of an $\m$--primary ideal, i.e., exhibiting linearly bounded behavior as in \Cref{thm:attainment}.

\begin{definition}[Linearly bounded]\label{def:power-like}
A filtration $\abul=\{\a_p\}$ of $\m$--primary ideals is \emph{linearly bounded}
if there exists an $\m$--primary ideal $I$ and constants $c,C>0$ such that
\[
\overline{I^{\lceil cp\rceil}}\subseteq \overline{\a_p}\subseteq \overline{I^{\lfloor Cp\rfloor}}
\quad \text{for all }p\ge1.
\]
\end{definition}

\begin{remark}
Noetherian graded families (i.e.,  finitely generated Rees algebras) and many Newton-type
filtrations are linearly bounded.
\end{remark}

\begin{proposition}[Mixed multiplicities for linearly bounded filtrations]\label{prop:family-mixed}
Let $(R,\m)$ be a Noetherian local ring of dimension $d$.
Let $\abul=\{\a_p\}_{p\ge1}$ and $\bbul=\{\b_q\}_{q\ge1}$ be filtrations of $\m$--primary ideals.
Assume that $\abul$ is linearly bounded by an $\m$--primary ideal $I$ with constants $c_\a,C_\a>0$
and that $\bbul$ is linearly bounded by an $\m$--primary ideal $J$ with constants $c_\b,C_\b>0$,
in the sense of \Cref{def:power-like}, i.e.,
\[
\overline{I^{\lceil c_\a p\rceil}}\subseteq \overline{\a_p}\subseteq \overline{I^{\lfloor C_\a p\rfloor}},
\quad
\overline{J^{\lceil c_\b q\rceil}}\subseteq \overline{\b_q}\subseteq \overline{J^{\lfloor C_\b q\rfloor}}
\quad \text{for all }p,q\ge1.
\]
If $\b_q\subseteq \overline{\a_p}$ for some integers $p,q\ge1$, then setting
\[
\alpha:=\lceil c_\b q\rceil,
\quad
\beta:=\lfloor C_\a p\rfloor,
\]
we have, for each $1\le i\le d$,
\[
\left(\frac{\alpha}{\beta}\right)^i
\ \ge\
\frac{e\bigl(I^{[i]},\m^{[d-i]}\bigr)}{e\bigl(J^{[i]},\m^{[d-i]}\bigr)}.
\]
Equivalently,
\[
\frac{\alpha}{\beta}\ \ge\
\left(\frac{e\bigl(I^{[i]},\m^{[d-i]}\bigr)}{e\bigl(J^{[i]},\m^{[d-i]}\bigr)}\right)^{1/i}.
\]
\end{proposition}

\begin{proof}
Assume $\b_q\subseteq \overline{\a_p}$. Then, $\overline{\b_q}\subseteq \overline{\a_p}$.
By the linear boundedness hypotheses,
\[
\overline{J^{\alpha}}
\ \subseteq\
\overline{\b_q}
\ \subseteq\
\overline{\a_p}
\ \subseteq\
\overline{I^{\beta}}.
\]
In particular, $J^{\alpha} \subseteq \overline{I^{\beta}}$.
Apply \Cref{cor:containment-mixed-ideal} to the containment $J^\alpha\subseteq \overline{I^\beta}$
(with $(I,J)$ there replaced by $(I,J)$ here and $(p,q)$ replaced by $(\beta,\alpha)$):
for each $1\le i\le d$ we obtain
\[
\frac{\alpha}{\beta}
\ \ge\
\left(\frac{e\bigl(I^{[i]},\m^{[d-i]}\bigr)}{e\bigl(J^{[i]},\m^{[d-i]}\bigr)}\right)^{1/i},
\]
and raising both sides to the $i$--th power gives the stated inequality.
\end{proof}

\begin{corollary}[Explicit slope bounds in terms of $q/p$]\label{cor:powerlike-slope-explicit}
In the setting of \Cref{prop:family-mixed}, fix $1\le i\le d$ and set
\[
r_i:=\left(\frac{e\!\left(I^{[i]},\m^{[d-i]}\right)}{e\!\left(J^{[i]},\m^{[d-i]}\right)}\right)^{\!1/i}.
\]
If $\b_q\subseteq \overline{\a_p}$ for some $p > 1/C_\a$ and $q \ge 1$, then
\begin{equation}\label{eq:explicit-q-over-p}
\frac{q}{p}\ \ge\ \frac{C_\a}{c_\b}\,r_i\ -\ \frac{r_i+1}{c_\b\,p}.
\end{equation}
In particular, for any $\varepsilon>0$, if $p\ge \dfrac{r_i+1}{c_\b\,\varepsilon}$ then
\[
\frac{q}{p}\ \ge\ \frac{C_\a}{c_\b}\,r_i-\varepsilon.
\]
\end{corollary}

\begin{proof}
Let $\alpha=\lceil c_\b q\rceil$ and $\beta=\lfloor C_\a p\rfloor$ as in
\Cref{prop:family-mixed}.  That proposition gives
\[
\frac{\alpha}{\beta}\ \ge\ r_i.
\tag{$\ast$}
\]
Using $\alpha\le c_\b q+1$ and $\beta\ge C_\a p-1$ (note that $p > 1/C_\a$ guarantees $\beta\ge 0$),
inequality $(\ast)$ implies
\[
\frac{c_\b q+1}{C_\a p-1}\ \ge\ r_i,
\]
hence
\[
c_\b q+1\ \ge\ r_i(C_\a p-1)=r_i C_\a p-r_i,
\]
so
\[
q\ \ge\ \frac{C_\a}{c_\b}\,r_i\,p\ -\ \frac{r_i+1}{c_\b}.
\]
Dividing by $p$ gives \eqref{eq:explicit-q-over-p}, and the final statement follows immediately.
\end{proof}


\section{Log canonical thresholds and generic hyperplane sections}\label{sec:lct-rigidity}

In this section we relate the Lojasiewicz exponent to another fundamental valuative threshold
invariant, the asymptotic log canonical threshold. Particularly, we compare the valuative descriptions of
$\lct(\abul)$ and $\Loj_{\bbul}(\abul)$, and show that they satisfy the optimal bound $\lct(\abul) \cdot \Loj_{\bbul}(\abul) \ge \lct(\bbul)$ (\Cref{thm:lctL-sharp}). This inequality complements the Teissier--type estimates established in Section~\ref{sec:teissier} and identifies a natural equality regime governed by a common extremal valuation.
Finally, we establish a generic hyperplane--section principle for filtrations, showing that admissible containment ratios descend and that the associated sectional \L{}ojasiewicz exponents are monotone under restriction along a general flag (\Cref{cor:L-monotonicity}).

\subsection{Valuative definition of the log canonical threshold}
We shall work with a regular local ring $(R,\m)$ of dimension $d$ (essentially of finite type over a field of characteristic $0$). For an ideal $\a \subseteq R$, the \emph{log canonical threshold} of $\a$, denoted by $\lct(\a)$, is an important invariant that has been much studied in both singularity theory and birational geometry (cf. \cite{KollarSing,LazPosII}). Particularly, if $\a = (f_1, \dots, f_s)$ then
$$\lct(\a) = \sup\{\theta > 0 : \big(\sum_i|f_i|^2\big)^{-\theta} \text{ is locally integrable}\}.$$
Let $A(v)$ denote the \emph{log discrepancy} of a valuation $v$ (see \cite{BlickleLazarsfeld, KollarSing, LazPosII} for more details on log discrepancy and log canonical threshold of multiplier ideals). Then, $\lct(\a)$ is also given by 
$$\lct(\a) = \inf_v \frac{A(v)}{v(\a)},$$ 
where $v$ ranges over all divisorial valuations of $R$ centered at $\m$.

For a graded family $\abul = \{\a_p\}_{p \ge 1}$ of ideals in $R$, the \emph{asymptotic log canonical threshold} of $\abul$, denoted by $\lct(\abul)$, was defined and studied by Jonsson--Musta\c{t}\v{a} \cite{JM12}. Particularly, 
$$\lct(\abul) = \inf_v \frac{A(v)}{v(\abul)},$$
where $v$ ranges over all divisorial valuations of $R$ centered at $\m$. It was observed in \cite{JM12} that, in general, there may not exist a divisorial valuation $v$ realizing $\lct(\abul)$. On the other hand, it was also shown in \cite[Theorem A]{JM12}, that given any graded family $\abul$, there is a real valuation $v$ of $R$ centered at $\m$ that computes $\lct(\abul)$, i.e., $\lct(\abul) = \frac{A(v)}{v(\abul)}$.

\subsection{A sharp inequality: $\lct(\abul)\cdot \Loj_{\bbul}(\abul)\ge \lct(\bbul)$}

We now prove the key ``dual'' inequality and isolate the rigidity mechanism. In the analytic setting, when $R = \cO_{\CC^d,0}$, $\abul = \{\a^p\}_{p \ge 1}$ and $\bbul = \{\m^q\}_{q \ge 1}$, this inequality was known by Bivi\`a-Ausina--Fukui \cite{BF16}. Our point is that the same inequality holds in greater generality and fits naturally into the valuative min-max framework developed in this paper. We also characterize the equality in the next section.

\begin{theorem}[Sharp inequality]\label{thm:lctL-sharp}
	Let $(R,\m)$ be regular local of dimension $d$. Let $\abul$ and $\bbul$ be graded families of ideals in $R$. Then,
	\[
	\lct(\abul)\cdot \Loj_{\bbul}(\abul) \ge \lct(\bbul).
	\]
	Particularly, $\lct(\abul) \cdot \Loj_{\m}(\abul) \ge d$.
	Equivalently,
	\[
	\Theta(\abul):=\frac{1}{d}\,\lct(\abul)\cdot \Loj_\m(\abul)\ \ge\ 1.
	\]
\end{theorem}

\begin{proof}
	By the valuative formulas and \Cref{thm:valuative-filtrations},
	\[
	\lct(\abul)=\inf_v \frac{A(v)}{v(\abul)}
	\quad\text{and}\quad
	\Loj_{\bbul}(\abul) \ge \sup_v \frac{v(\abul)}{v(\bbul)},
	\]
	where $v$ ranges over divisorial valuations centered at $\m$.
	Therefore,
	\[
	\lct(\abul)\cdot \Loj_{\bbul}(\abul)
	\ge
	\Bigl(\inf_v \frac{A(v)}{v(\abul)}\Bigr)\cdot
	\Bigl(\sup_w \frac{w(\abul)}{w(\bbul)}\Bigr)
	\ \ge\
	\inf_v \Bigl(\frac{A(v)}{v(\abul)}\cdot \frac{v(\abul)}{v(\bbul)}\Bigr)
	=
	\inf_v \frac{A(v)}{v(\bbul)} = \lct(\bbul).
	\]
	
	When $\bbul = \{\m^q\}_{q \ge 1}$, we have $\lct(\abul) \cdot \Loj_{\m}(\abul) \ge \lct(\m)$. On the other hand, since $R$ is regular of dimension $d$, one has $\lct(\m)=d$; see, for instance \cite[\S\,9]{LazPosII}.
	Thus, $\lct(\abul) \cdot \Loj_\m(\abul)\ge d$.
\end{proof}



\subsection{Generic hyperplane sections and the \L{}ojasiewicz exponent}
Let $(R,\m)$ be a Noetherian local ring with infinite residue field. 

\begin{definition}[Eventual degree--one normality]
\label{def:ED1}
An $\m$--primary filtration $\abul=\{\a_p\}_{p\ge1}$ is said to be
\emph{eventually degree--one normal} if there exists an integer $p_0\ge1$
such that
\[
\overline{\a_p}=\overline{\a_1^{\,p}}
\quad \text{for all } p\ge p_0.
\]
\end{definition}

\begin{remark}[Specialization of integral closure]
\label{rem:HU-HL}
Let $I$ be an ideal of height at least $2$ in a Noetherian ring $R$ which is normal,
locally equidimensional, universally catenary, and such that $R_{\mathrm{red}}$
is locally analytically unramified.
The following statements hold:
\begin{enumerate}
\item[(i)]
(\cite[Theorem~2.1]{HongUlrich2014})
Integral closure is compatible with specialization by a generic element of $I$.
\item[(ii)]
(\cite[Proposition~3.7]{HillLynn2024})
For a general element $x\in I$, one has
\[
\overline{I^n}\,R/(x)=\overline{I^nR/(x)}
\quad \text{for all } n\gg 0.
\]
\end{enumerate}
\end{remark}

The following result is an application of specialization of integral closure to filtrations of ideals.

\begin{theorem}[Stable--range commutation for $\overline{\a_{pt}}$]
\label{thm:specialization-a}
Assume that $R$ satisfies the hypotheses of \Cref{rem:HU-HL} and $\dim R \ge 2$.
Let $\abul$ be an $\m$--primary filtration that is eventually degree-one normal, and set $I=\a_1$.
Then, there exists a Zariski open dense subset $U\subseteq I/\m I$ such that
for every $x\in I$ with class in $U$, and every integer $p \ge 1$, there exists an integer $t_0=t_0(p)$ with the property that
\[
\overline{\a_{pt}}\,R'=\overline{\a_{pt}R'},
\text{  for all } t\ge t_0, \text{  where } R' = R/(x).
\]
\end{theorem}

\begin{proof}
Fix $p\ge1$.
By eventual degree--one normality, there exists $t_1(p)$ such that
$\overline{\a_{pt}}=\overline{I^{pt}}$ for all $t\ge t_1(p)$.
Since $\abul$ is graded, $I^{pt}\subseteq \a_{pt}$, and so for such $t$,
\[
I^{pt}\subseteq \a_{pt}\subseteq \overline{I^{pt}}.
\]

Let $x\in I$ be general as in
Remark~\ref{rem:HU-HL}(ii) and set $R'=R/(x)$.
By \cite[Proposition~3.7]{HillLynn2024}, there exists $n_0$ such that
\[
\overline{I^{n}}\,R'=\overline{I^{n}R'}
\quad \text{for all } n\ge n_0.
\]
Taking $n=pt$, this equality holds for all $t\ge t_2(p):=\lceil n_0/p\rceil$.

For $t\ge \max\{t_1(p),t_2(p)\}$ we therefore have
\[
\overline{\a_{pt}}\,R'
=\overline{I^{pt}}\,R'
=\overline{I^{pt}R'}
=\overline{\a_{pt}R'},
\]
where the last equality uses the inclusions
$I^{pt}R'\subseteq \a_{pt}R'\subseteq \overline{I^{pt}}\,R'$ and the fact that
integral closure depends only on the integral--closure class of the ideal.
\end{proof}

\begin{corollary}[Admissible ratios descend]
\label{cor:admissible-descend}
Assume the hypotheses of Theorem~\ref{thm:specialization-a}.
Let $\bbul=\{\b_q\}_{q \ge 1}$ be an arbitrary filtration.
Fix $p,q\ge1$, and let $x\in\a_1$ be a general element, and set $R'=R/(x)$.
If
\[
\b_{qt}\subseteq \overline{\a_{pt}}
\quad \text{for } t\gg1,
\]
then
\[
\b_{qt}R' \subseteq \overline{\a_{pt}R'}
\quad \text{for } t\gg1.
\]
\end{corollary}

\begin{proof}
For $t\gg1$, quotienting gives
$\b_{qt}R'\subseteq \overline{\a_{pt}}\,R'$.
For the fixed integer $p$, Theorem~\ref{thm:specialization-a} yields
$\overline{\a_{pt}}\,R'=\overline{\a_{pt}R'}$ for all $t\gg1$,
which proves the claim.
\end{proof}

\begin{corollary}[Monotonicity under generic hyperplane section]
\label{cor:L-monotonicity}
Under the same assumptions as in \Cref{cor:admissible-descend},
\[
\Loj_{\bbul R'}(\abul R') \;\le\; \Loj_{\bbul}(\abul).
\]
\end{corollary}

\begin{proof}
Every ratio $\frac{q}{p}$ admissible for $\Loj_{\bbul}(\abul)$ remains
admissible after restriction by
Corollary~\ref{cor:admissible-descend}. Taking infima gives the inequality.
\end{proof}

Iterating Corollary~\ref{cor:L-monotonicity} along a general flag
$x_1,\dots,x_{d-k}\in \a_1$ yields the monotonicity of the sectional
\L{}ojasiewicz exponents
\[
\Loj^{(d)}_{\bbul}(\abul)\ \ge\ \Loj^{(d-1)}_{\bbul}(\abul)\ \ge\ \cdots\ \ge\
\Loj^{(1)}_{\bbul}(\abul),
\]
where the sectional invariants are defined by restriction to
$R/(x_1,\dots,x_{d-k})$.

\part{Structural Consequences of the Finite-Max Principle} \label{part:structuralFinite-max}

We develop rigidity and stratification phenomena that follow once
the \L{}ojasiewicz exponent is computed by a finite candidate set of divisorial
valuations.

\section{Rigidity results} \label{sec:rigidity}

The theme of this section is that equality in the structural inequalities developed in \Cref{sec:teissier,sec:lct-rigidity} is highly rigid: it forces a common extremal valuation to compute several a priori different invariants simultaneously.  We first formulate a general common-extremal-valuation principle (\Cref{thm:common-valuation}) and then analyze its consequences in increasingly structured settings, including a complete classification in the monomial case (\Cref{thm:monomial-rigidity}) and extensions to Newton nondegenerate regimes.  We also treat the equality case in Teissier--Rees--Sharp type inequalities, showing that proportionality/equality of the relevant mixed-multiplicity data is detected valuatively and leads to strong constraints on the underlying ideals and filtrations (e.g., \Cref{prop:TRS-equality-proportionality}).

\subsection{Rigidity I: the common extremal valuation principle}\label{subsec:rigidity-common-extremal}

We begin with the basic mechanism: equality in \Cref{thm:lctL-sharp} forces the same
divisorial valuation to compute all extremal quantities involved.

\begin{theorem}[Common extremal valuation criterion]\label{thm:common-valuation}
With hypotheses as in \Cref{thm:lctL-sharp}, the equality
\[
\lct(\abul)\cdot \Loj_{\bbul}(\abul)=\lct(\bbul)
\]
holds if and only if there exists real valuation $v$ centered at $\m$ such that:
\begin{enumerate}[label=(\alph*),leftmargin=2em]
\item $v$ computes $\lct(\abul)$, i.e.\ $\lct(\abul)=A(v)/v(\abul)$;
\item $v$ computes $\Loj_{\bbul}(\abul)$, i.e.\ $\Loj_{\bbul}(\abul)=v(\abul)/v(\bbul)$;
\item $v$ computes $\lct(\bbul)$, i.e.\ $\lct(\bbul)=A(v)/v(\bbul)$.
\end{enumerate}
\end{theorem}

\begin{proof}
\emph{If.}
If $v$ satisfies (a)--(c), then
\[
\lct(\abul)\cdot \Loj_{\bbul}(\abul)
=
\frac{A(v)}{v(\abul)}\cdot \frac{v(\abul)}{v(\bbul)}
=
\frac{A(v)}{v(\bbul)}
=
\lct(\bbul).
\]

\smallskip
\noindent
\emph{Only if.}
Assume $\lct(\abul) \cdot \Loj_{\bbul}(\abul) = \lct(\bbul)$.
Choose a real valuation $v_0$ centered at $\m$ that computes $\lct(\abul)$, whose existence was established in \cite[Theorem A]{JM12}. That is,
$\lct(\abul)=A(v_0)/v_0(\abul)$.
By the valuative definition of $\Loj_{\bbul}(\abul)$ in \Cref{thm:valuative-filtrations}, we have
\[
\Loj_{\bbul}(\abul) \ge \frac{v_0(\abul)}{v_0(\bbul)}.
\]
Hence,
\[
\lct(\bbul) = \lct(\abul) \cdot \Loj_{\bbul}(\abul)
\ \ge\
\frac{A(v_0)}{v_0(\abul)}\cdot \frac{v_0(\abul)}{v_0(\bbul)}
=
\frac{A(v_0)}{v_0(\bbul)}
\ \ge\
\inf_v\frac{A(v)}{v(\bbul)}
=
\lct(\bbul).
\]
Therefore, all inequalities are equalities. In particular
$\Loj_{\bbul}(\abul)=v_0(\abul)/v_0(\bbul)$ and $\lct(\bbul) = A(v_0)/v_0(\bbul)$.
Thus, $v_0$ satisfies (a)--(c).
\end{proof}

\subsection{Rigidity II: the monomial case}
Focusing on one ideal, we now give a strong and explicit rigidity statement in the monomial case.
Assume $R=\CC[x_1,\dots,x_d]_{(x_1,\dots,x_d)}$ and $I$ is an $\m$--primary monomial ideal.

\begin{lemma}
\label{lem:monomial-Lojasiewicz-finite-max}
Let $I\subseteq R$ be an $\m$--primary monomial ideal,
and let $\NP(I)\subseteq \mathbb{R}_{\ge 0}^d$ denote its Newton polyhedron.
Let $\mathcal U(I)$ be the finite set of primitive inward normal vectors
$u\in \mathbb{N}^d$ to the compact facets of\/ $\NP(I)$, and for each such $u$
let $v_u$ denote the associated monomial divisorial valuation.
Then
\[
\Loj_{\m}(I)
\;=\;
\max_{u\in \mathcal U(I)} \frac{v_u(I)}{v_u(\m)}.
\]
\end{lemma}

\begin{proof}
By \Cref{thm:finite-max-rees}, $\Loj_{\m}(I)$ is computed by finitely many Rees valuations of $I$.

Since $I$ is monomial, the Rees valuations of $I$
are precisely the monomial valuations obtained from the (non-redundant) bounding
hyperplanes of the Newton polyhedron $\NP(I)$; see \cite[Theorem~10.3.5]{HS}.
For an $\m$--primary monomial ideal, the relevant bounding hyperplanes are exactly
those defining the compact facets of $\NP(I)$, hence they are indexed by the primitive
inward normal vectors $u\in \mathcal U(I)$, and the corresponding valuations are the $v_u$.
Therefore
\[
\Loj_{\m}(I)
=
\max_{u\in \mathcal U(I)} \frac{v_u(I)}{v_u(\m)}. \qedhere
\]
\end{proof}

\begin{theorem}[Monomial $\Theta(I)\ge 1$ and equality classification]\label{thm:monomial-rigidity}
	Let $I\subseteq \CC[x_1,\dots,x_d]_\m$ be an $\m$--primary monomial ideal.
	Then
	\[
	\Theta(I)=\frac{1}{d}\,\lct(I)\cdot \Loj_\m(I)\ \ge\ 1.
	\]
	Moreover, $\Theta(I)=1$ if and only if there exists a vector
	$u=(t,\dots,t)\in \ZZ_{>0}^d$ such that the same monomial (toric) valuation $v_u$
	computes both $\lct(I)$ and $\Loj_\m(I)$.
	Equivalently, $\NP(I)$ has a unique compact facet meeting the diagonal ray
	$\RR_{\ge0}(1,\dots,1)$, and its inward normal is proportional to $(1,\dots,1)$.
\end{theorem}

\begin{proof}
	\emph{Step 1: Monomial formulas.}
	For a monomial valuation $v_u$ with $u\in\ZZ_{>0}^d$ we have
	\[
	v_u(x^\alpha)=\ip{u}{\alpha},\quad v_u(\m)=\min_i u_i,\quad A(v_u)=\sum_{i=1}^d u_i.
	\]
	Also $v_u(I)=h_I(u)$ where $h_I$ is the support function of $\NP(I)$.
	
	Howald's formula (\cite{How01}) for monomial ideals gives
	\[
	\lct(I)=\min_{u\in\ZZ_{>0}^d}\frac{A(v_u)}{v_u(I)}
	=
	\min_{u\in\ZZ_{>0}^d}\frac{\sum_i u_i}{h_I(u)}.
	\]
	Moreover, by \Cref{lem:monomial-Lojasiewicz-finite-max}, we have
	\begin{equation}\label{eq:toric-Loj-formula}
		\Loj_\m(I)
		=
		\max_{u\in\mathcal U}\frac{v_u(I)}{v_u(\m)}
		=
		\max_{u\in\mathcal U}\frac{h_I(u)}{\min_i u_i},
	\end{equation}
	where $\mathcal U\subseteq \ZZ_{>0}^d$ is the finite set of primitive inward normals to the
	compact facets of $\NP(I)$.
	
	\smallskip
	\noindent
	\emph{Step 2: Inequality and equality.} The inequality $\Theta(I) \ge 1$ follows from \Cref{thm:lctL-sharp}.

	Consider the equality $\Theta(I)=1$, i.e., $\lct(I)\cdot \Loj_\m(I)=d=\lct(\m)$. Let $v_u$, for some $u \in \ZZ^d_{> 0}$, be a monomial valuation that computes $\lct(I)$ as in Step 1. Then, since $\Loj_{\m}(I) \ge \frac{v_u(I)}{v_u(\m)}$ and $\frac{A(v_u)}{v_u(\m)} \ge \lct(\m)$, the equality $\Theta(I) = 1$ forces $v_u$ to compute both $\Loj_{\m}(I)$ and $\lct(\m)$. Thus,
	\[
	\frac{A(v_u)}{v_u(\m)}=\frac{\sum_i u_i}{\min_i u_i}=d
	\quad\Longleftrightarrow\quad
	u_1=\cdots=u_d.
	\]
	Hence, the common extremal valuation must be $v_u$ with $u$ proportional to $(1,\dots,1)$.
	Conversely, if there exists such a $u$ computing both invariants, then the computation in Step~2
	shows $\Theta(I)=1$.
	The polyhedral restatement in terms of a compact facet meeting the diagonal ray
	$\RR_{\ge0}(1,\dots,1)$ is the standard interpretation of the valuation $v_{(1,\dots,1)}$.
\end{proof}

\subsection{Rigidity III: extension to Newton nondegenerate settings}
We remark that the monomial rigidity picture can be extended to Newton nondegenerate
settings via the reduction to the monomial model. (In \Cref{sec:NND}, particularly \Cref{def:newton-infty}, we will discuss the notion of Newton nondegeneracy in more details.)
The rigidity results above extend to any class of ideals where:
\begin{enumerate}[label=(\alph*),leftmargin=2em]
	\item both $\lct(I)$ and $\Loj_\m(I)$ depend only on $\overline I$; and
	\item $\overline I$ is controlled by an explicit monomial model (same Newton polyhedral data).
\end{enumerate}
This includes Newton nondegenerate ideals and toric models developed in later sections.

\begin{remark}[Integral-closure reduction preserves $\Theta$]\label{prop:theta-ic}
	If $\overline I=\overline J$, then $\Theta(I)=\Theta(J)$.
\end{remark}

\begin{proof}
	Since $\overline{I} = \overline{J}$, $\lct(I)=\lct(J)$. Also $\Loj_\m(I)=\Loj_\m(J)$ since
	$\Loj_\m(I)$ is defined via integral-closure containments and hence depends only on $\overline I$.
	Therefore, $\Theta(I)=\Theta(J)$.
\end{proof}

\begin{remark}\label{rem:theta-ic-bridge}
	\Cref{prop:theta-ic} is the key mechanism for transporting rigidity from the
	monomial/toric setting to more general singularity classes.
	For instance, in Newton nondegenerate situations where one can replace an ideal $I$ by a
	monomial ideal $I_{\mathrm{mon}}$ with $\overline{I}=\overline{I_{\mathrm{mon}}}$, the full
	inequality and equality classification of \Cref{thm:monomial-rigidity} applies to $I$
	via \Cref{prop:theta-ic}.
\end{remark}


\subsection{Rigidity IV: equality in Teissier--Rees--Sharp inequality}\label{subsec:rigidity-TRS}

We record a classical but conceptually important source of
\emph{rigidity}: equality in Teissier's Minkowski-type inequalities for
multiplicities forces the ideals involved to be \emph{proportional in the
	Rees--valuative sense}. 

\begin{lemma}\label{lem:IC-powers-ree}
	Let $I,J\subseteq R$ be nonzero ideals and assume $I+J$ is $\m$--primary.
	If $\overline{I^p}=\overline{J^q}$ for some $p,q\in \mathbb{Z}_{>0}$, then for
	every Rees valuation $v$ of $IJ$ one has
	\[
	\frac{v(I)}{v(J)}=\frac{q}{p}.
	\]
	In particular, the ratio $v(I)/v(J)$ is constant on the finite set of Rees
	valuations of $IJ$.
\end{lemma}

\begin{proof}
	By the valuative criterion for integral closure, $\overline{I^p}=\overline{J^q}$
	is equivalent to $p\,v(I)=q\,v(J)$ for every divisorial valuation $v$ centered at
	$\m$. Applying this to the (finitely many) Rees valuations of $IJ$ yields the
	claim.
\end{proof}

\begin{proposition}[Teissier--Rees--Sharp equality $\Rightarrow$ proportionality]\label{prop:TRS-equality-proportionality}
	Let $(R,\m)$ be a quasi--unmixed Noetherian local ring of dimension $d$, and let
	$I,J\subseteq R$ be $\m$--primary ideals.
	Assume that equality holds in Teissier's Minkowski (or first mixed multiplicity)
	inequality, i.e.,
	\[
	e\!\left(I^{[d-1]},J^{[1]}\right)^d \;=\; e(I)^{d-1}\,e(J).
	\]
	Then, there exist $p,q\in\mathbb{Z}_{>0}$ such that
	\[
	\overline{I^p}=\overline{J^q}.
	\]
	Consequently, the ratio $v(I)/v(J)$ is constant
	(and equal to $q/p$) on the finite set of Rees valuations of $IJ$.
\end{proposition}

\begin{proof}
	This is the classical characterization of the equality case in Teissier's
	Minkowski-type inequalities for multiplicities; see Teissier~\cite{Teissier77}
	and the detailed local-algebraic treatment of Rees--Sharp~\cite{ReesSharp78}.
\end{proof}

\begin{corollary}[Rigidity of extremal ratios in the equality case]\label{cor:rigidity-extremal-ratio-TRS}
	In the setting of \Cref{prop:TRS-equality-proportionality}, suppose
	moreover that one is considering any valuative extremal ratio of the form
	\[
	\Loj_{J}(I) :=\ \sup_{v}\frac{v(I)}{v(J)},
	\]
	where $v$ ranges over divisorial valuations centered at $\m$ (or any subclass
	that contains the Rees valuations of $IJ$).
	Then, $\Loj_J(I)=q/p$, and every Rees valuation of $IJ$ computes
	$\Loj_J(I)$.
\end{corollary}

\begin{proof}
	By \Cref{prop:TRS-equality-proportionality} we have
	$\overline{I^p}=\overline{J^q}$ for some $p,q>0$, hence $p\,v(I)=q\,v(J)$ for all
	divisorial $v$, so $v(I)/v(J)=q/p$ identically.  The conclusion follows, and in
	particular every Rees valuation computes the supremum.
\end{proof}

\section{Functoriality, stability, and stratification}\label{sec:stability}

This section develops the structural consequences of the finite-max principle in families.  Assuming $\Loj_{\bbul}(\abul)$ is computed by a finite candidate set of divisorial valuations on a fixed birational model, we study how these candidate sets and the corresponding maximizers behave under basic operations (localization, completion, integral closure, and natural local homomorphisms), establishing functoriality and upper-semicontinuity properties that make the maximizer meaningful geometric data.  We then show that, under a uniform finite candidate set hypothesis, parameter spaces admit a finite stratification (equivalently, a wall--chamber decomposition) on which the identity of the computing valuation is constant and $\Loj_{\bbul(t)}(\abul(t))$ is given by a single valuative ratio on each stratum (\Cref{thm:stratification}).  Finally, for one-parameter valuation-concave paths we prove that $t\mapsto 1/\Loj_{\bbul(t)}(\abul(t))$ is piecewise fractional linear and that affineness forces persistence of the controlling valuation data (\Cref{thm:concavity}).

Throughout this section, $R$ is a Noetherian local ring with maximal ideal $\m$.


\subsection{Finite candidate sets and the maximizer map}

A major source of rigidity is \emph{finite attainment on a fixed birational model}.
We isolate a hypothesis that holds in many situations of interest, including:
Noetherian families (finitely generated Rees algebra),
toric/monomial polyhedral families (see also Section~\ref{sec:toric}),
Newton nondegenerate integral-closure reductions,
and families arising from a fixed log smooth model.

\begin{definition}[Finite candidate set hypothesis]\label{def:finite-candidates}
	Let $\mathcal V=\{v_1,\dots,v_N\}$ be a finite set of divisorial valuations centered at $\m$.
	We say that a pair $(\abul,\bbul)$ satisfies the \emph{finite candidate set hypothesis}
	with respect to $\mathcal V$ if
	\[
	\Loj_{\bbul}(\abul)=\max_{1\le i\le N}\frac{v_i(\abul)}{v_i(\bbul)}.
	\]
\end{definition}

\begin{remark}
	When $\abul$ is $\m$--primary and Noetherian (or sufficiently polyhedral),
	Section~\ref{sec:valuative} gives that $\Loj_{\bbul}(\abul)$ is computed by finitely many valuations
	coming from a single normalized blowup. This is exactly the situation of
	Definition~\ref{def:finite-candidates}.
\end{remark}

\subsection{Functoriality}
We record functoriality properties under localization, completion and integral-closure.

\begin{proposition}[Monotonicity under maps]\label{prop:functoriality}
	Let $\varphi:R\to S$ be a local homomorphism of Noetherian local rings with maximal ideals
	$\m$ and $\n$, respectively, such that $\varphi(\m)\subseteq \n$.
	Let $\abul,\bbul$ be filtrations of $\m$--primary ideals in $R$ and set
	\[
	\varphi(\abul):=\{\a_p S\}_{p\ge 1},\quad \varphi(\bbul):=\{\b_q S\}_{q\ge 1}.
	\]
	Then,
	\[
	\Loj_{\bbul}(\abul)\ \ge\ \Loj_{\varphi(\bbul)}(\varphi(\abul)).
	\]
	Moreover, if $\phi$ is faithfully flat and has the \emph{integral-closure contraction property}, i.e., $\overline{IS} \cap R = \overline{I}$ for any ideal $I \subseteq R$, then
		\[
	\Loj_{\bbul}(\abul)\ = \ \Loj_{\varphi(\bbul)}(\varphi(\abul)).
	\]
\end{proposition}

\begin{proof}
	Fix $\varepsilon>0$.
	By definition of $\Loj_{\bbul}(\abul)$ as an infimum of admissible rational slopes,
	there exist integers $p,q\ge 1$ such that
	\[
	\frac{q}{p} < \Loj_{\bbul}(\abul)+\varepsilon
	\quad\text{and}\quad
	\b_{qt} \subseteq \overline{\a_{pt}} \ \forall \ t \gg 1.
	\]
	
	Applying $\varphi$ gives $\b_{qt} S \subseteq (\overline{\a_{pt}})S.$
	Since integral dependence is preserved under ring homomorphisms, 
	we obtain $\b_{qt} S \subseteq \overline{\a_{pt} S} \ \forall \ t \gg 1.$
	Thus, the rational number $\frac{q}{p}$ is an admissible slope for the pair of
	filtrations $\varphi(\abul)$ and $\varphi(\bbul)$.
	By definition of $\Loj_{\varphi(\bbul)}(\varphi(\abul))$, it follows that
	\[
	\Loj_{\varphi(\bbul)}(\varphi(\abul)) \le \frac{q}{p}
	< \Loj_{\bbul}(\abul)+\varepsilon.
	\]
	Letting $\varepsilon\to 0$ yields
	$\Loj_{\varphi(\bbul)}(\varphi(\abul)) \le \Loj_{\bbul}(\abul)$ as desired.

Suppose, in addition, that $\phi$ is faithfully flat and has the integral-closure contraction property. If $\b_{qt} S \subseteq \overline{\a_{pt} S}$ for $t \gg 1$, then
$$\b_{qt} \subseteq (\b_{qt} S) \cap R \subseteq \overline{\a_{pt} S} \cap R = \overline{\a_{pt}} \ \forall \ t \gg 1.$$
That is, every admissible slope $\frac{q}{p}$ for the filtrations $\phi(\abul)$ and $\phi(\bbul)$ is also an admissible slope for the filtrations $\abul$ and $\bbul$. Therefore, by a similar argument, this implies that $\Loj_{\bbul}(\abul) \le \Loj_{\phi(\bbul)}(\phi(\abul))$. Hence, we obtained the equality as claimed.
\end{proof}

\begin{corollary}[Localization invariance for $\m$--primary data]\label{prop:localization}
	Let $(R,\m)$ be a Noetherian local ring and let $\p\subseteq R$ be a prime ideal.
	Set $R_\p$ with maximal ideal $\p R_\p$.
	Let $\abul=\{\a_p\}_{p\ge1}$ and $\bbul=\{\b_q\}_{q\ge1}$ be filtrations of $\p$--primary ideals in $R$.
	Then
	\[
	\Loj_{\bbul}(\abul)=\Loj_{\bbul R_\p}(\abul R_\p).
	\]
\end{corollary}

\begin{proof}
	This is a direct consequence of \Cref{prop:functoriality}, noting that $R \rightarrow R_\p$ is faithfully flat and has the integral-closure contraction property.
\end{proof}

\begin{corollary}[Completion invariance]\label{prop:completion}
	Let $(R,\m)$ be an excellent Noetherian local ring and let $\widehat R$ be the $\m$--adic completion.
	Assume that $R$ is analytically unramified (equivalently, $\widehat R$ is reduced).
	Let $\abul=\{\a_p\}_{p\ge1}$ and $\bbul=\{\b_q\}_{q\ge1}$ be filtrations of $\m$--primary ideals in $R$.
	Then,
	\[
	\Loj_{\bbul}(\abul)=\Loj_{\bbul\widehat R}(\abul\widehat R).
	\]
\end{corollary}

\begin{proof}
	Observe that $R \rightarrow \widehat{R}$ is faithfully flat. Since $R$ is excellent, $R \rightarrow \widehat{R}$ also has the integral-closure contraction property. Thus, the assertion is a consequence of \Cref{prop:functoriality}.
\end{proof}

\begin{proposition}[Integral-closure invariance for filtrations]\label{prop:ic-filtration}
	Let $\abul,\bbul$ be filtrations in $R$.
	Define $\ov{\abul}:=\{\ov{\a_p}\}_{p\ge 1}$ and similarly $\ov{\bbul}$.
	Then
	\[
	\Loj_{\bbul}(\abul)=\Loj_{\ov{\bbul}}(\ov{\abul}).
	\]
\end{proposition}

\begin{proof}
	For every $p,q$ we have $\overline{\b_q}\subseteq \overline{\a_p}$ if and only if
	$\b_q\subseteq \overline{\a_p}$, since $\b_q\subseteq \overline{\b_q}$ and integral closure is idempotent.
	Hence, the set of admissible ratios $q/p$ in the containment definition is unchanged, and
	$\Loj_{\bbul}(\abul)=L_{\overline{\bbul}}(\overline{\abul})$.
\end{proof}

\begin{remark}[Saturation]
	For non-$\m$--primary ideals one may consider saturations $\a_p^{\mathrm{sat}}$.
	In the present paper, our main invariants are defined for $\m$--primary filtrations, for
	which saturation is trivial. When extending beyond $\m$--primary data, one must track the
	change in the center of valuations; the valuative formula remains valid but the candidate
	set of valuations changes.
\end{remark}

\subsection{Stratification by the computing valuation}
We describe the induced stratification by the computing valuation.
Assume the \emph{finite candidate set hypothesis} (Definition~\ref{def:finite-candidates}) holds uniformly for a family
$\{(\abul(t), \bbul(t))\}_{t\in T}$ with the same finite set of divisorial valuations
$V=\{v_1,\dots,v_N\}$.

\begin{definition}[Maximizer set for families]\label{def:maximizers}
	For each $t\in T$ define
	\[
	\alpha_i(t):=v_i(\a_\bullet(t)),\quad \beta_i(t):=v_i(\b_\bullet(t))\quad (1\le i\le N),
	\]
	and assume $\alpha_i(t)\in [0,\infty)$ and $\beta_i(t)\in (0,\infty)$ for all $i$.
	(For instance, this holds if $\bbul(t)$ is linearly bounded from below in the sense of \Cref{rmk:finite-testing-fingen}.)
	Define the maximizer set
	\[
	\Max(\abul(t), \bbul(t))
	:=\left\{\, i\in\{1,\dots,N\} \;:\; \frac{\alpha_i(t)}{\beta_i(t)}
	=\max_{1\le k\le N}\frac{\alpha_k(t)}{\beta_k(t)} \right\}.
	\]
	If $\Max(\abul(t), \bbul(t))=\{j\}$ is a singleton, we write
	$\argmax(\abul(t), \bbul(t)):=j$.
\end{definition}

	The stratification phenomenon in \Cref{thm:stratification} below is driven by the
	finite--max principle established earlier.
	Indeed, under the standing hypotheses, the value
	$\Loj_{\bbul}(\abul)$ is computed by a valuation belonging to a fixed finite set
	(depending only on $\abul$), so changes in the value of $\Loj_{\bbul}(\abul)$ can occur
	only when the identity of the maximizing valuation changes.

\begin{theorem}[Finite wall--chamber decomposition]\label{thm:stratification}
	Assume the setup of Definition~\ref{def:maximizers}. Assume that for each $t\in T$ the finite candidate set
	hypothesis gives
	\[
	\Loj_{\bbul(t)}(\abul(t))=\max_{1\le i\le N}\frac{\alpha_i(t)}{\beta_i(t)},
	\quad
	\alpha_i(t):=v_i(\abul(t)),\ \beta_i(t):=v_i(\bbul(t)).
	\]
	Assume moreover that each $\alpha_i$ and $\beta_i$ is continuous on $T$.
	Then, $T$ admits a finite partition into locally closed strata, on each of which the maximizer set
	$\Max(\abul(t),\bbul(t))$ is constant. In particular, on each stratum, the function
	$t\mapsto \Loj_{\bbul(t)}(\abul(t))$ is given by a single ratio $\alpha_j(t)/\beta_j(t)$.
	
	Moreover, for each $j\in\{1,\dots,N\}$ the locus
	\[
	T_j^\circ:=\{\,t\in T:\Max(\abul(t),\bbul(t))=\{j\}\,\}
	\]
	is open in $T$.
\end{theorem}

\begin{proof}
	For each nonempty $J\subseteq\{1,\dots,N\}$ set
	\[
	T_J
	:=
	\Big(\bigcap_{i,i'\in J}\{\,t:\alpha_i(t)\beta_{i'}(t)=\alpha_{i'}(t)\beta_i(t)\,\}\Big)
	\cap
	\Big(\bigcap_{i\in J,\ k\notin J}\{\,t:\alpha_i(t)\beta_k(t)>\alpha_k(t)\beta_i(t)\,\}\Big).
	\]
	By continuity, each equality locus is closed and each strict-inequality locus is open. Thus, each $T_J$
	is locally closed.
	
	Fix $t\in T$ and write $\rho_i(t):=\alpha_i(t)/\beta_i(t)$, which is well-defined since
	$\beta_i(t)\in(0,\infty)$ for all $i$ by Definition~\ref{def:maximizers}. Let $J:=\Max(\abul(t),\bbul(t))$.
	Then, $\rho_i(t)=\rho_{i'}(t)$ for $i,i'\in J$ and $\rho_i(t)>\rho_k(t)$ for $i\in J$, $k\notin J$.
	Multiplying by the positive denominators yields precisely the defining conditions of $T_J$. Therefore, $t\in T_J$.
	
	Conversely, if $t\in T_J$, the defining equalities/inequalities imply that the maximum of the finite set
	$\{\rho_1(t),\dots,\rho_N(t)\}$ is attained exactly on $J$, i.e. $\Max(\abul(t),\bbul(t))=J$.
	Hence, the sets $T_J$ form a finite partition of $T$ by maximizer set, proving the first assertion.
	
	Finally,
	\[
	T_j^\circ=\bigcap_{i\neq j}\{\,t:\alpha_j(t)\beta_i(t)>\alpha_i(t)\beta_j(t)\,\},
	\]
	an intersection of open sets, and so is open as claimed.
\end{proof}

\begin{corollary}[Stability from a strict gap]\label{cor:strict-separation}
	Assume the hypotheses of \Cref{thm:stratification}. Fix $t_0\in T$ and an index
	$j\in\{1,\dots,N\}$. If there exists $\delta>0$ such that
	\begin{equation}\label{eq:9.1-new}
		\frac{\alpha_j(t_0)}{\beta_j(t_0)} \ge \frac{\alpha_i(t_0)}{\beta_i(t_0)}+\delta
		\quad\text{for all }i\neq j,
	\end{equation}
	then there exists an open neighborhood $U$ of $t_0$ such that for all $t\in U$ one has
	\[
	\Max(\abul(t),\bbul(t))=\{j\},
	\text{ and so }
	\Loj_{\bbul(t)}(\abul(t))=\frac{\alpha_j(t)}{\beta_j(t)}.
	\]
	Equivalently, $t_0\in T_j^\circ$ and $T_j^\circ$ is a neighborhood of $t_0$.
\end{corollary}

\begin{proof}
	Since $\beta_i(t)\in(0,\infty)$ for all $i$ and $t$ (Definition~\ref{def:maximizers}), the functions
	$t\mapsto \alpha_i(t)/\beta_i(t)$ are well-defined. For each $i\neq j$, set
	\[
	g_i(t):=\frac{\alpha_j(t)}{\beta_j(t)}-\frac{\alpha_i(t)}{\beta_i(t)}.
	\]
	By the assumed continuity of $\alpha_\bullet,\beta_\bullet$ and positivity of $\beta_\bullet$,
	each $g_i$ is continuous on $T$. The strict gap hypothesis \eqref{eq:9.1-new} gives
	$g_i(t_0)\ge \delta>0$ for all $i\neq j$. Hence, for each $i\neq j$ there exists an open neighborhood
	$U_i$ of $t_0$ such that $g_i(t)>0$ for all $t\in U_i$. Let $U:=\bigcap_{i\neq j}U_i$.
	Then, for all $t\in U$,
	\[
	\frac{\alpha_j(t)}{\beta_j(t)}>\frac{\alpha_i(t)}{\beta_i(t)} \quad\text{for all }i\neq j,
	\]
	so the unique maximizer is $j$, i.e.\ $\Max(\abul(t),\bbul(t))=\{j\}$. The displayed formula for
	$\Loj_{\bbul(t)}(\abul(t))$ follows immediately from the finite-candidate identity of
	$\Loj_{\bbul(t)}(\abul(t))$ given in \Cref{thm:stratification}.
\end{proof}

\begin{example}[Change of maximizing valuation]\label{ex:change-maximizer}
	Let $R=\CC[x,y]_{(x,y)}$ and consider the monomial ideal
	\[
	\a=(x^4,\ x^2y^3,\ y^5)\subseteq R.
	\]
	The Newton polygon $\NP(\a)$ has two compact edges: the segment joining
	$(4,0)$ to $(2,3)$ with primitive inward normal $u_1=(3,2)$, and the segment
	joining $(2,3)$ to $(0,5)$ with primitive inward normal $u_2=(1,1)$.
	Accordingly,
	\[
	v_{u_1}(\a)=12,\quad v_{u_2}(\a)=5.
	\]
	
	For $t\in[0,1]$, define a family of principal ideals
	\[
	\b_p(t):=(x^{\lceil (1+t)p\rceil}y^{\lceil 2p\rceil}),\quad
	\bbul(t):=\{\b_p(t)\}_{p\ge1}.
	\]
	Then, asymptotically,
	\[
	v_{u_1}(\bbul(t))=7+3t,\quad
	v_{u_2}(\bbul(t))=3+t.
	\]
	Hence,
	\[
	\frac{v_{u_1}(\a)}{v_{u_1}(\bbul(t))}=\frac{12}{7+3t},
	\quad
	\frac{v_{u_2}(\a)}{v_{u_2}(\bbul(t))}=\frac{5}{3+t}.
	\]
	These two ratios are equal precisely at $t=1\frac{1}{3}$. Consequently,
	\[
	\argmax(\a,\bbul(t))=
	\begin{cases}
		\{u_1\}, & t<\frac13,\\[4pt]
		\{u_1,u_2\}, & t=\frac13,\\[4pt]
		\{u_2\}, & t>\frac13.
	\end{cases}
	\]
	In particular, the finite candidate set is fixed, but the maximizing valuation
	jumps as $t$ crosses the wall $t=\frac13$, illustrating the stratification
	phenomenon of \Cref{thm:stratification}.
\end{example}

This example shows that even under finite testing hypotheses,
the \L{}ojasiewicz exponent need not be globally linear in families,
and rigidity corresponds precisely to the absence of such changes
in the maximizing valuation.

Before continue to discuss stability and rigidity consequences, we record a general semicontinuity property of divisorial values in families.

\begin{proposition}[Upper semicontinuity of divisorial order in algebraic families]\label{prop:divisorial-usc}
	Let $T$ be a Noetherian scheme, let $X$ be a Noetherian integral scheme, and set $X_T:=X\times T$.
	Let $\pi:Y\to X_T$ be a proper birational morphism with $Y$ normal and $Y \to T$ a flat family, and let $E\subseteq Y$ be an effective Cartier divisor that dominates $T$. Let $\cI\subseteq \cO_{X_T}$ be a coherent ideal sheaf, and for $t\in T$ let $\cI_t$ be its fiber on $X\times\{t\}$.
	Define
	\[
	\alpha(t):=\ord_E(\cI_t)\in \ZZ_{\ge 0}\cup\{\infty\}.
	\]
	Then, for each $m\ge 0$, the locus
	\[
	T_{\ge m}:=\{\,t\in T:\alpha(t)\ge m\,\}
	\]
	is Zariski closed in $T$. Equivalently, $t\mapsto \alpha(t)$ is upper semicontinuous on $T$.
\end{proposition}

\begin{proof}
	Set $\cJ:=\pi^{-1}\cI\cdot\cO_Y$, a coherent ideal sheaf on $Y$, and let $p:E\to T$ be the induced morphism.
	Since $\pi$ is proper, $p$ is proper.
	
	Fix $m\ge 0$. Consider the coherent $\cO_Y$--module
	\[
	\cQ_m:=\frac{\cJ+\cO_Y(-mE)}{\cO_Y(-mE)}.
	\]
	Let $\cG_m:=\cQ_m|_E$. Then, $\Supp(\cG_m)$ is closed in $E$, hence $p(\Supp(\cG_m))$ is closed in $T$.
	We claim that
	\[
	T_{\ge m}=T\setminus p(\Supp(\cG_m)).
	\]
	Indeed, let $t\in T$ and let $\eta_t$ be the generic point of the prime divisor $E_t:=E\times_T\kappa(t)$.
	Since $Y$ is normal, the local ring $\cO_{Y_t,\eta_t}$ is a DVR with valuation $\ord_{E_t}$.
	By definition of $\cQ_m$, the stalk $(\cQ_m)_{\eta_t}$ vanishes if and only if
	\[
	\cJ_{\eta_t}\subseteq \cO_{Y,\eta_t}(-mE),
	\]
	and after restricting to the fiber this is equivalent to
	\[
	(\cJ_t)_{\eta_t}\subseteq \cO_{Y_t,\eta_t}(-mE_t)\quad\Longleftrightarrow\quad \ord_{E_t}(\cJ_t)\ge m.
	\]
	Because $\cJ=\pi^{-1}\cI\cdot\cO_Y$, we have $\ord_{E_t}(\cJ_t)=\ord_{E_t}(\cI_t)=\alpha(t)$.
	Thus, $(\cQ_m)_{\eta_t}=0$ if and only if $\alpha(t)\ge m$, i.e.\ $\eta_t\notin \Supp(\cG_m)$.
	Equivalently, $t\notin p(\Supp(\cG_m))$ if and only if $\alpha(t)\ge m$, proving the claim.
	Hence, $T_{\ge m}$ is closed.
\end{proof}

\subsection{Geodesics of filtrations and fractional linearity of $1/L$}

We now turn to one-parameter families of filtrations and prove a fractional linearity statement
under a maximizer stability hypothesis. This provides a rigidity mechanism:
\emph{affineness of $1/L$ characterizes common extremal valuations along the path.}

\begin{definition}[Valuation-concave paths]\label{def:valuation-concave}
	Let $\{(\abul(t),\bbul(t))\}_{t\in[0,1]}$ be a one-parameter family of pairs of
	filtrations of $\m$--primary ideals.
	We say it is \emph{valuation-concave} (or \emph{valuation-affine}) if for every divisorial valuation $v$ centered at $\m$,
	the functions
	\[
	t\longmapsto v(\abul(t)) \text{ and } t\longmapsto v(\bbul(t))
	\]
	are affine on $[0,1]$.
\end{definition}

Assume now that the finite candidate set hypothesis holds uniformly along $t\in[0,1]$ for
some finite set $\mathcal V=\{v_1,\dots,v_N\}$, and that the family is valuation-concave.
Recall that $\alpha_i(t):=v_i(\abul(t)) \text{ and } \beta_i(t):=v_i(\bbul(t)).$
Then, $\alpha_i,\beta_i$ are affine in $t$.
Define
\[
L(t):=\Loj_{\bbul(t)}(\abul(t))=\max_{1\le i\le N}\frac{\alpha_i(t)}{\beta_i(t)}.
\]

\begin{theorem}[Piecewise fractional linearity of $1/L$]\label{thm:concavity}
	Assume that 
	\begin{enumerate} 
		\item for all $t\in[0,1]$ one has the finite candidate set hypothesis, as in \Cref{def:finite-candidates},
	with respect to a fixed set $\mathcal V=\{v_1,\dots,v_N\}$;
	\item the path $t\mapsto (\abul(t),\bbul(t))$ is valuation-concave, in the sense of Definition~\ref{def:valuation-concave}, and $\beta_i(t)>0$.
	\end{enumerate}
	Then, there exists a finite subset $\Sigma\subset[0,1]$ such that, on every connected component
	$I\subset[0,1]\setminus\Sigma$, one can choose an index $j=j(I)\in\{1,\dots,N\}$ with
	\[
	\frac{1}{L(t)}=\frac{\beta_j(t)}{\alpha_j(t)}\quad\text{for all }t\in I.
	\]
	In particular, $t\mapsto 1/L(t)$ is fractional linear (hence $C^\infty$) on each such $I$, and its second
	derivative has constant sign on $I$. Moreover, on a given component $I$ the restriction $t\mapsto 1/L(t)$
	is affine if and only if either $\alpha_j$ is constant on $I$ or the affine functions $\alpha_j$ and $\beta_j$ are proportional on $I$.
\end{theorem}

\begin{proof}
	For $1\le i<j\le N$ set
	\[
	h_{ij}(t):=\alpha_i(t)\beta_j(t)-\alpha_j(t)\beta_i(t).
	\]
	By valuation-concavity, each $\alpha_\ell,\beta_\ell$ is affine in $t$, hence each $h_{ij}$ is a polynomial
	of degree at most $2$. Define
	\[
	\Sigma:=\{0,1\}\ \cup\ \bigcup_{\substack{1\le i<j\le N\\ h_{ij}\not\equiv 0}}\{\,t\in[0,1]:h_{ij}(t)=0\,\}.
	\]
	Then, $\Sigma$ is finite, since each nonzero polynomial $h_{ij}$ has finitely many zeros (or is identically zero, in which case the pair $(i,j)$ does not contribute to ``walls'' and can be ignored) and there are only finitely many
	pairs $(i,j)$.
	
	Let $I\subset[0,1]\setminus\Sigma$ be a connected component. Fix $i<j$.
	If $h_{ij}\equiv 0$, then $\alpha_i(t)\beta_j(t)=\alpha_j(t)\beta_i(t)$ for all $t$, hence
	$\alpha_i(t)/\beta_i(t)=\alpha_j(t)/\beta_j(t)$ for all $t$ (using $\beta_i(t)\beta_j(t)>0$).
	If $h_{ij}\not\equiv 0$, then by definition of $\Sigma$ we have $h_{ij}(t)\neq 0$ for all $t\in I$, and since
	$h_{ij}$ is continuous it has constant sign on $I$. Using again $\beta_i(t)\beta_j(t)>0$, this implies that the comparison
	\[
	\frac{\alpha_i(t)}{\beta_i(t)}\ \gtreqless\ \frac{\alpha_j(t)}{\beta_j(t)}
	\quad\Longleftrightarrow\quad
	h_{ij}(t)\ \gtreqless\ 0
	\]
	is independent of $t\in I$. Therefore, the total ordering (with ties) of the finite set of numbers
	$\{\alpha_1(t)/\beta_1(t),\dots,\alpha_N(t)/\beta_N(t)\}$ does not change with $t\in I$. Hence, the maximizer set
	$\Max(\abul(t),\bbul(t))$ is constant on $I$. Choose any index $j=j(I)$ in this constant maximizer set. Then, for all $t\in I$,
	\[
	L(t)=\max_{1\le i\le N}\frac{\alpha_i(t)}{\beta_i(t)}=\frac{\alpha_j(t)}{\beta_j(t)}.
	\]
	Since $\abul(t)$ consists of $\m$--primary ideals and $v_j$ is centered at $\m$, we have $\alpha_j(t)=v_j(\abul(t))>0$ on $[0,1]$.
	Write $\alpha_j(t)=a_0+a_1 t \text{ and } \beta_j(t)=b_0+b_1 t.$
	Then, on $I$,
	\[
	\frac{1}{L(t)}=\frac{b_0+b_1 t}{a_0+a_1 t},
	\]
	a ratio of affine functions with positive denominator, so fractional linear and $C^\infty$ on $I$.
	Differentiating twice gives
	\[
	\frac{d^2}{dt^2}\left(\frac{b_0+b_1 t}{a_0+a_1 t}\right)
	=
	-\frac{2a_1(b_1 a_0-b_0 a_1)}{(a_0+a_1 t)^3},
	\]
	whose sign is constant on $I$ because the denominator is positive and the numerator is constant.
	Finally, $\beta_j/\alpha_j$ is affine on $I$ if and only if $a_1(b_1 a_0-b_0 a_1) = 0$, i.e.\ if and only if either $\alpha_j$ is constant or $\alpha_j$ and $\beta_j$ are proportional affine functions on $I$.
\end{proof}

\begin{remark}[A cleaner sufficient condition]
	If $\alpha_j(t)$ is constant in $t$ (or more generally $a_1=0$), then $\beta_j/\alpha_j$ is affine
	and hence concave. This occurs in many geometric situations where $\abul$ is fixed and only
	$\bbul$ varies.
\end{remark}

\begin{corollary}[Geodesic rigidity from global affineness]\label{cor:geodesic-rigidity}
	Assume the hypotheses of \Cref{thm:concavity}. 
	If the function $t\mapsto 1/L(t)$ is affine on $[0,1]$, then there exists an index $j\in\{1,\dots,N\}$ such that
	\[
	L(t)=\frac{\alpha_j(t)}{\beta_j(t)}\quad\text{for all }t\in[0,1].
	\]
	Equivalently, $j\in\Max(\abul(t),\bbul(t))$ for every $t\in[0,1]$.
\end{corollary}

\begin{proof}
	By \Cref{thm:concavity}, there exists a finite set $\Sigma\subset[0,1]$ such that on each connected component
	$I\subset[0,1]\setminus\Sigma$ one can choose an index $j=j(I)$ with
	\[
	\frac{1}{L(t)}=\frac{\beta_j(t)}{\alpha_j(t)}\quad\text{for all }t\in I.
	\]
	Fix one such component $I$ and let $j=j(I)$. By assumption, $1/L(t)$ is affine on $[0,1]$, say
	\[
	\frac{1}{L(t)}=c_0+c_1 t\quad (t\in[0,1]).
	\]
	Thus, on $I$ we have
	\[
	\beta_j(t)=\bigl(c_0+c_1 t\bigr)\alpha_j(t).
	\]
	Define $F(t):=\beta_j(t)-\bigl(c_0+c_1 t\bigr)\alpha_j(t).$
	Since $\alpha_j$ and $\beta_j$ are affine functions of $t$ (valuation--concavity), the function $F(t)$ is a polynomial
	of degree at most $2$. As $F$ vanishes on the nonempty open interval $I$, it follows that $F\equiv 0$ on $[0,1]$.
	Hence,
	\[
	\frac{\beta_j(t)}{\alpha_j(t)}=c_0+c_1 t\quad\text{for all }t\in[0,1].
	\]
	Since $\alpha_j(t)>0$ for all $t$ (because $\abul(t)$ consists of $\m$--primary ideals and $v_j$ is centered at $\m$),
	we may invert to obtain
	\[
	L(t)=\frac{\alpha_j(t)}{\beta_j(t)}\quad\text{for all }t\in[0,1].
	\]
	Finally, this identity implies $j\in\Max(\abul(t),\bbul(t))$ for every $t$ by the definition of $\Max$.
\end{proof}



\part{Polyhedral Models and Explicit Computation} \label{part:polyhedral}

In this part, we specialize the general valuative framework to toric and
Newton nondegenerate settings, where the \L{}ojasiewicz exponent admits explicit polyhedral
descriptions. We will also give a better understanding of \L{}ojasiewicz exponent at infinity.


\section{Toric and monomial models: polyhedra, valuations, and formulas}\label{sec:toric}

This section provides the explicit polyhedral model underlying finite--max computations in toric and monomial settings.  We recall how integral closure of monomial ideals is encoded by Newton polyhedra and how toric (monomial) divisorial valuations correspond to linear functionals/support functions, so that valuative containment thresholds become linear optimization problems (\Cref{thm:int-closure-monomial}).  From this perspective we derive Newton/facet formulas for the algebraic \L{}ojasiewicz exponent, showing that $\Loj_{\bbul}(\abul)$ is computed by finitely many toric valuations determined by facet data (\Cref{thm:facet-formula,thm:toric-filtration-formula}).  In particular, the finite candidate set hypothesis becomes completely concrete: the maximizing valuation comes from a facet normal, yielding uniform finite-max formulas and effective computation (\Cref{cor:finite-max-toric-polyhedral}).

\subsection{Affine semigroup rings and monomial ideals}

For basic terminology in toric geometry, we refer to \cite{CLS,Fulton}. Let $N\cong \ZZ^n$ be a lattice and let $M:=\Hom_\ZZ(N,\ZZ)$ be the dual lattice.
Fix a finitely generated (not necessarily saturated) subsemigroup
$S\subseteq M$
such that the rational cone
\[
\sigma^\vee:=\RR_{\ge 0}\cdot S \ \subseteq  M_\RR:=M\otimes_\ZZ \RR
\]
is a pointed rational polyhedral cone of full dimension $n$ (equivalently, $\sigma^\vee$
contains no nontrivial linear subspace). Let $\sigma\subseteq N_\RR$ be the dual cone:
\[
\sigma:=\{u\in N_\RR:\ \ip{u}{x}\ge 0\ \text{for all }x\in\sigma^\vee\}.
\]
Consider the affine semigroup ring $\kk[S]$ and its maximal monomial ideal
\[
\m_S := ( \chi^m : m\in S\setminus\{0\} ).
\]
We work in the local ring
$R := \kk[S]_{\m_S},$ which is an excellent Noetherian local domain, not necessarily normal.
This is the local ring of the affine toric variety $X_\sigma=\Spec(\kk[S])$ at its
distinguished torus-fixed point.

An ideal $I\subseteq R$ is \emph{monomial} if it is generated by monomials $\chi^m$,
$m\in S$. Equivalently, it is the extension to $R$ of a monomial ideal of $\kk[S]$.

\begin{remark}
	In $R=\kk[x_1,\dots,x_n]_{(x_1,\dots,x_n)}$ one recovers the usual monomial ideals
	by taking $S=\NN^n$ and $\chi^m=x^m$.
\end{remark}

\subsection{Newton polyhedra in the toric setting}

For a monomial ideal $I\subseteq R$, define its exponent set
\[
E(I):=\{m\in S:\ \chi^m\in I\}\subseteq S\subseteq M.
\]

\begin{definition}[Newton polyhedron]\label{def:NP-toric}
	The \emph{Newton polyhedron} of a monomial ideal $I\subseteq R$ is
	\[
	\NP(I):=\conv(E(I))+\sigma^\vee\ \subseteq  M_\RR.
	\]
\end{definition}
Particularly, if $I$ is generated by monomials $\chi^{m_1}, \dots, \chi^{m_r}$, then $\NP(I) = \conv(m_1, \dots, m_r) + \sigma^\vee$ is a closed rational polyhedron.

\begin{lemma}[Basic properties] \label{lem:NP-basic}
	Let $I,J\subseteq R$ be monomial ideals. Then:
	\begin{enumerate}[label=(\alph*),leftmargin=2em]
		\item $\NP(I)$ is a closed convex set satisfying $\NP(I)+\sigma^\vee=\NP(I)$.
		\item $\NP(IJ)=\NP(I)+\NP(J)$ (Minkowski sum).
		\item $\NP(I^p)=p \cdot \NP(I)$ for $p\ge 1$.
	\end{enumerate}
\end{lemma}

\begin{proof}
	All statements are standard for Newton polyhedra of monomial ideals; see, e.g., \cite{Sturmfels,TeissierNP}.
	For completeness: (b) follows from $E(IJ)=E(I)+E(J)$ and convexity, and (c) follows by iterating (b).
\end{proof}

\subsection{Toric (monomial) valuations and support functions}

Each vector $u\in \sigma\cap N$ defines a monomial valuation on $\kk[S]$ by
\[
v_u(\chi^m):=\ip{u}{m} \quad (m\in S),
\]
and we extend it to $R=\kk[S]_{\m_S}$.

\begin{definition}[Toric valuation]\label{def:toric-valuation}
	For $u\in \sigma\cap N$, let $v_u$ denote the valuation on $R$ determined by
	$v_u(\chi^m)=\ip{u}{m}$. We call $v_u$ a \emph{toric valuation}.
	It is centered at $\m_S$ if $u$ lies in the relative interior $\sigma^\circ$,
	equivalently $v_u(\m_S)>0$.
\end{definition}

\begin{definition}[Support function]\label{def:support}
	Let $C\subseteq M_\RR$ be a nonempty closed convex set with $C+\sigma^\vee=C$.
	For $u\in\sigma$ define the support function
	\[
	h_C(u):=\inf\{\ip{u}{x}: x\in C\}\in\RR_{\ge 0}.
	\]
	For a monomial ideal $I$ we write $h_I:=h_{\NP(I)}$.
\end{definition}

\begin{lemma}[Valuations and support functions]\label{lem:val-support}
	Let $I\subseteq R$ be a monomial ideal and let $u\in\sigma\cap N$. Then
	\[
	v_u(I)=h_I(u).
	\]
\end{lemma}

\begin{proof}
	Since $I$ is generated by monomials, $v_u(I)=\min\{\langle u,m\rangle: m\in E(I)\}$.
	Taking convex hull and adding $\sigma^\vee$ does not change the minimum because $\langle u,\cdot\rangle$ is linear and nonnegative on $\sigma^\vee$.
	Thus, $v_u(I)=h_{\NP(I)}(u)$; see, for instance, \cite{CLS}.
\end{proof}

\begin{definition}[Normalization by a fixed monomial ideal]\label{def:normalize-q}
	Fix an $\m_S$-primary monomial ideal $\q\subseteq R$.
	For $u\in\sigma^\circ\cap N$, set
	\[
	\widehat u := \frac{u}{h_{\q}(u)}\in\sigma^\circ\cap N_\RR,
	\quad\text{so that}\quad
	v_{\widehat u}(\q)=1.
	\]
	(Since $\q$ is $\m_S$--primary, $h_{\q}(u) > 0$ for all $u \in \sigma^\circ \cap N$.)
\end{definition}

\subsection{Integral closure and Newton polyhedra in toric rings}

The next result is standard and is the backbone of toric computations:
for monomial ideals, integral closure is detected by the Newton polyhedron.

\begin{theorem}[Integral closure of monomial ideals]\label{thm:int-closure-monomial}
	Let $R=\kk[S]_{\m_S}$ as above and let $I\subseteq R$ be a monomial ideal.
	Then, $\overline{I}$ is monomial and satisfies
	\[
	\chi^m\in \overline{I}\quad \Longleftrightarrow\quad m\in \NP(I)\cap S.
	\]
	Equivalently, $\NP(\overline{I})=\NP(I)$.
\end{theorem}

\begin{proof}
	We give a direct valuative proof that is self-contained in our context.
	
	Let $m\in S$. As in the proof of \Cref{thm:valuative-ic}, $\chi^m\in\overline{I}$ if and only if, for any prime divisor $E$ appearing on the  normalized blowup of $I$, 
	\[
	\ord_E(\chi^m)\ge \ord_E(I).
	\]
	Since $I$ is monomial, the blowup of $I$ is torus-invariant, and its normalization is a normal toric modification; in particular, its exceptional prime divisors $E$ centered at $\m_S$ are torus-invariant and correspond to rays $\rho$ of the associated fan with primitive generators $u_\rho \in \sigma^\circ \cap N$ (see, for example, \cite{CLS}). The corresponding divisorial valuation satisfies 
	$$\ord_E(\chi^m) = \langle u_\rho, m\rangle \text{ and } \ord_E(I) = v_{u_\rho}(I).$$

By Lemma~\ref{lem:val-support}, $v_{u_\rho}(I) = h_I(u_\rho)$, the support function of the Newton polyhedron $\NP(I)$. Hence, $\chi^m \in \overline{I}$ if and only if 
$$\langle u_\rho, m\rangle \ge h_I(u_\rho) \text{ for all such } u_\rho.$$

Finally, $\NP(I)$ is a rational polyhedron, so it is the intersection of its supporting halfspaces defined by the primitive inward normals of its compact facets. Therefore, the above inequalities are equivalent to $m \in \NP(I) \cap S$.
\end{proof}

\subsection{Toric formula for the algebraic \L ojasiewicz exponent}

We now specialize our general invariant to monomial ideals (and then to monomial
filtrations) in the toric local ring $R=\kk[S]_{\m_S}$.

\begin{lemma}[Toric valuative formula for monomial ideals]\label{prop:toric-valuative}
	Let $\a,\b\subseteq R$ be monomial ideals with $\a$ being $\m_S$-primary.
	Then
	\[
	\Loj_{\b}(\a)
	=
	\max_u \frac{v_u(\a)}{v_u(\b)}
	=
	\max_u \frac{h_{\NP(\a)}(u)}{h_{\NP(\b)}(u)},
	\]
	where the max is taken over finitely many $u \in \sigma^\circ \cap N$.
\end{lemma}

\begin{proof}
	By \Cref{thm:finite-max-rees}, when $\a$ is $\m_S$-primary,
	$\Loj_{\b}(\a)$ is computed by finitely many Rees valuations of $\a$.
	For the monomial ideal $\a$, the normalized blowup of $\a$ is toric and its prime divisors
	are torus-invariant. Hence, its Rees valuations are toric valuations $v_u$
	for finitely many primitive lattice vectors $u \in \sigma^\circ \cap N$, lying on rays of a fan refining $\sigma$.
	Therefore, the maximum over finitely many Rees valuations equals the maximum over finitely many toric
	valuations $v_u$.
	Finally, Lemma~\ref{lem:val-support} gives $v_u(\cdot)=h_{\NP(\cdot)}(u)$ on monomial ideals.
\end{proof}

\subsection{Facet normals and finite candidate sets}

We next package the maximization problem in purely polyhedral terms.

A \emph{compact facet} of $\NP(\a)$ means a facet not parallel to $\sigma^\vee$
(i.e.,\ a facet whose supporting hyperplane intersects $\sigma^\vee$ only at the origin).
These are precisely facets $F$ whose inward normal vector $u_F$ lies in $\sigma^\circ \cap N$ and, thus, gives a valuation centered at $\m_S$.

\begin{definition}[Facet normals]\label{def:facet-normals}
	Let $\a$ be an $\m_S$-primary monomial ideal.
	Let $\mathcal N(\a)$ denote the finite set of primitive inward normals $u_F\in\sigma\cap N$
	to the compact facets $F$ of $\NP(\a)$.
\end{definition}

\begin{theorem}[Facet finite--max formula]\label{thm:facet-formula}
	Let $\a\subseteq R$ be an $\m_S$-primary monomial ideal and let $\b\subseteq R$ be any monomial ideal.
	Then
	\[
	\Loj_{\b}(\a)
	=
	\max_{u\in\mathcal N(\a)}\frac{h_{\NP(\a)}(u)}{h_{\NP(\b)}(u)}.
	\]
	In particular, $\Loj_{\b}(\a)$ is attained by a toric divisorial valuation
	$v_u$ with $u\in\mathcal N(\a)$.
\end{theorem}

\begin{proof}
	By \Cref{thm:finite-max-rees}, $\Loj_{\b}(\a)$ is the maximum of
	$v(\a)/v(\b)$ over the Rees valuations of $\a$.
	For a monomial ideal $\a$, the Rees valuations correspond exactly to the torus-invariant
	prime divisors on the normalized blowup of $\a$. Polyhedrally, these correspond to the
	compact facets of $\NP(\a)$, and their valuations are the toric valuations $v_{u_F}$,
	where $u_F$ is the primitive inward normal to the facet $F$.
	(Equivalently, $\NP(\a)$ determines a normal fan; the rays corresponding to compact facets
	yield precisely these divisors.)
	Thus, the maximum is achieved among the finite set $\{v_u: u\in\mathcal N(\a)\}$.
	Finally, apply Lemma~\ref{lem:val-support}.
\end{proof}

\subsection{Monomial filtrations in the toric setting}

We now extend the preceding discussion from powers of ideals to monomial filtrations.

\begin{definition}[Monomial filtration and Newton polyhedra]\label{def:monomial-filtration}
	A filtration $\abul=\{\a_p\}_{p \ge 1}$ in $R$ is \emph{monomial} if each $\a_p$ is a monomial ideal.
	For such a family define
	\[
	\NP(\abul):=\bigcup_{p\ge1}\frac{1}{p}\,\NP(\a_p)\ \subseteq  M_\RR.
	\]
\end{definition}

\begin{remark}
	When $\a_p=\overline{\a^p}$, one has $\NP(\a_p)=p\,\NP(\a)$ and hence
	$\NP(\abul)=\NP(\a)$. For more general families $\NP(\abul)$ is a convex region
	encoding asymptotic integral closures.
\end{remark}

\begin{lemma}[Support function for a monomial filtration]\label{lem:support-filtration}
	Let $\abul$ be a monomial graded family and $u\in\sigma^\circ\cap N$.
	Then
	\[
	v_u(\abul)=\inf_{p\ge1}\frac{v_u(\a_p)}{p}
	=\inf_{p\ge1}\frac{h_{\NP(\a_p)}(u)}{p}
	=\inf_{x\in \NP(\abul)}\ip{u}{x}.
	\]
\end{lemma}

\begin{proof}
	The first equality is Remark~\ref{lem:vabul-exists}.
	The second is Lemma~\ref{lem:val-support} applied to each $\a_p$.
	For the last equality, note that $\NP(\abul)$ is the union of sets $\frac1p\NP(\a_p)$, and
	\[
	\inf_{x\in \frac1p\NP(\a_p)}\ip{u}{x}=\frac{1}{p}\inf_{y\in\NP(\a_p)}\ip{u}{y}
	=\frac{h_{\NP(\a_p)}(u)}{p}.
	\]
	Taking infima over $p$ gives the claim.
\end{proof} 


\begin{lemma}[Finite--max on a finite fan] \label{prop:finite-max-fan}
	Let $\alpha,\beta:\sigma\to\RR_{\ge0}$ be continuous functions that are positively homogeneous of degree $1$. Assume that: \begin{enumerate} 
		\item $\beta(u)>0$ for all $u\in\sigma^\circ$;
		\item there exists a finite rational polyhedral fan $\Sigma$ supported on $\sigma$ such that $\alpha$ and $\beta$ are linear on every cone $\tau\in\Sigma$.
	\end{enumerate}
	Let $\rho_1,\dots,\rho_M$ be the rays of $\Sigma$ and $u_1,\dots,u_M\in N$ their primitive generators. Then,
	$$
	\sup_{u\in\sigma^\circ\cap N}\frac{\alpha(u)}{\beta(u)}
	=
	\sup_{u\in\sigma^\circ}\frac{\alpha(u)}{\beta(u)}
	=
	\max_{\{j:\beta(u_j)>0\}}\frac{\alpha(u_j)}{\beta(u_j)}.
	$$
\end{lemma}

\begin{proof}
	Fix a cone $\tau\in\Sigma$, and let $w_1,\dots,w_s$ be the primitive generators of its rays. Every $u\in\tau$ can be written as $u=\sum_i\lambda_i w_i$ with $\lambda_i\ge0$. Since $\alpha$ and $\beta$ are linear on $\tau$, we have $\alpha(u)=\sum_i\lambda_i\alpha(w_i)$ and $\beta(u)=\sum_i\lambda_i\beta(w_i)$.
	
	If $u\in\tau^\circ$, then $\beta(u)>0$ by assumption. Hence,
	$$
	\frac{\alpha(u)}{\beta(u)}
	=
	\frac{\sum_i\lambda_i\alpha(w_i)}{\sum_i\lambda_i\beta(w_i)}
	=
	\sum_{\beta(w_i)>0}
	\frac{\lambda_i\beta(w_i)}{\sum_k\lambda_k\beta(w_k)}
	\cdot
	\frac{\alpha(w_i)}{\beta(w_i)}.
	$$
	Thus $\alpha(u)/\beta(u)$ is a convex combination of the finitely many numbers $\alpha(w_i)/\beta(w_i)$ with $\beta(w_i)>0$. Therefore,
	$$
	\frac{\alpha(u)}{\beta(u)}
	\le
	\max_{\{i:\beta(w_i)>0\}}\frac{\alpha(w_i)}{\beta(w_i)}
	$$
	for every $u\in\tau^\circ$, and hence the same bound holds for $\sup_{u\in\tau^\circ}$.
	
	Conversely, let $w_i$ be a ray generator of $\tau$ with $\beta(w_i)>0$. Choose $z\in\tau^\circ$ and set $u_\varepsilon=(1-\varepsilon)w_i+\varepsilon z$ for $0<\varepsilon<1$. Then, $u_\varepsilon\in\tau^\circ$ and $u_\varepsilon\to w_i$ as $\varepsilon\to0^+$. By continuity of $\alpha$ and $\beta$ and since $\beta(w_i)>0$, we have $\alpha(u_\varepsilon)/\beta(u_\varepsilon)\to\alpha(w_i)/\beta(w_i)$. Hence, $\sup_{u\in\tau^\circ}\alpha(u)/\beta(u)\ge\alpha(w_i)/\beta(w_i)$.
	
	Combining the two inequalities shows that
	$$
	\sup_{u\in\tau^\circ}\frac{\alpha(u)}{\beta(u)}
	=
	\max_{\{i:\beta(w_i)>0\}}\frac{\alpha(w_i)}{\beta(w_i)}.
	$$
	Since $\Sigma$ has finitely many cones and every ray of $\Sigma$ occurs among the rays of some cone, taking the maximum over all cones gives
	$$
	\sup_{u\in\sigma^\circ}\frac{\alpha(u)}{\beta(u)}
	=
	\max_{\{j:\beta(u_j)>0\}}\frac{\alpha(u_j)}{\beta(u_j)}.
	$$
	
	Finally, because $\alpha/\beta$ is homogeneous of degree $0$, it suffices to approximate interior points by interior lattice points. Rational points are dense in $\sigma^\circ$, and any interior rational point becomes a lattice point after multiplying by a positive integer. Homogeneity implies that the ratio $\alpha/\beta$ is unchanged by such scaling, so the supremum over $\sigma^\circ$ equals the supremum over $\sigma^\circ\cap N$.
\end{proof}

In the next few results, if for each monomial filtration $\abul$, the function $u \mapsto v_u(\abul)$ on $\sigma^\circ \cap N$ admits a continuous positively homogeneous extension to $\sigma$, then we shall denote the value of this extension at $u \in \sigma$ by  $\widetilde{v_u}(\abul)$.

\begin{corollary}\label{cor:finite-max-toric-polyhedral}
	Let $\abul$ and $\bbul$ be monomial filtrations. Assume that the functions
	$\sigma^\circ\cap N \to \RR_{\ge 0}$, $u\mapsto v_u(\abul)$ and $u\mapsto v_u(\bbul)$, admit continuous extensions to $\sigma$, positively homogeneous of degree $1$, and that there exists a finite rational polyhedral fan $\Sigma$ supported on $\sigma$ such that both extended functions are linear on every cone $\tau\in\Sigma$. Assume also that $v_u(\bbul)>0$ for all $u\in\sigma^\circ$.
	
	Let $\rho_1,\dots,\rho_M$ be the rays of $\Sigma$ and let $u_1,\dots,u_M\in N$ be their primitive generators. Then,
	$$
	\sup_{u\in\sigma^\circ\cap N}\frac{v_u(\abul)}{v_u(\bbul)}
	=
	\max_{\{j:\widetilde v_{u_j}(\bbul)>0\}}
	\frac{\widetilde v_{u_j}(\abul)}{\widetilde v_{u_j}(\bbul)}.
	$$
\end{corollary}

\begin{proof}
	Apply \Cref{prop:finite-max-fan} with $\alpha(u)=\widetilde v_u(\abul)$ and $\beta(u)=\widetilde v_u(\bbul)$. Since $\widetilde v_u(\abul)=v_u(\abul)$ and $\widetilde v_u(\bbul)=v_u(\bbul)$ for $u\in\sigma^\circ\cap N$, the claimed equality follows.
\end{proof}

\begin{remark}\label{rem:where-C-applies}
	If $f$ is Newton nondegenerate and $\NP(f)$ is fixed, one typically chooses a regular fan $\Sigma$ refining the normal fan of $\NP(f)$.
	On such a fixed toric model, the relevant toric valuation data attached to $f$ (and, in the standard setup, to $J(f)$ as well) is determined by linear forms on the cones of $\Sigma$.
	In that situation Corollary~\ref{cor:finite-max-toric-polyhedral} gives a finite candidate set that is independent of parameters as long as the Newton polyhedron is unchanged.
\end{remark}

\begin{theorem}[Toric valuative formula for monomial filtrations] \label{thm:toric-filtration-formula}
	Let $\abul$ and $\bbul$ be $\m_S$-primary monomial filtrations, and assume that $\R(\abul)$ is Noetherian. Suppose moreover that there exist an $\m_S$-primary monomial ideal $\q$ and a finite rational polyhedral fan $\Sigma$ supported on $\sigma$ such that:
	\begin{enumerate}
		\item $\abul$ and $\bbul$ are linearly bounded with respect to $\q$;
		\item the functions $u \mapsto v_u(\abul)$ and $u \mapsto v_u(\bbul)$ on $\sigma^\circ\cap N \to \RR_{\ge0}$ admit continuous extensions to $\sigma$, positively homogeneous of degree $1$, and linear on every cone $\tau\in\Sigma$. 
	\end{enumerate}
	Let $\rho_1,\dots,\rho_M$ be the rays of $\Sigma$, and let $u_1,\dots,u_M\in N$ be their primitive generators. Then
	$$
	\Loj_{\bbul}(\abul)
	=
	\sup_{u\in\sigma^\circ\cap N}\frac{v_u(\abul)}{v_u(\bbul)}
	=
	\max_{\{j:\widetilde v_{u_j}(\bbul)>0\}}
	\frac{\widetilde v_{u_j}(\abul)}{\widetilde v_{u_j}(\bbul)}.
	$$
	In particular, $\Loj_{\bbul}(\abul)$ is computed by finitely many toric divisorial valuations.
\end{theorem}

\begin{proof}
	By \Cref{thm:valuative-filtrations}, we have
	$$
	\Loj_{\bbul}(\abul)\ge \sup\limits_{u\in\sigma^\circ\cap N}\frac{v_u(\abul)}{v_u(\bbul)}.
	$$
	By \Cref{cor:finite-max-toric-polyhedral}, the latter supremum is equal to
	$
	\max\limits_{\{j:\widetilde v_{u_j}(\bbul)>0\}}
	\frac{\widetilde v_{u_j}(\abul)}{\widetilde v_{u_j}(\bbul)}.
	$
	Thus, it remains to prove the reverse inequality.
	
	Since $\R(\abul)$ is finitely generated, \Cref{thm:finite-testing-fingen} gives a finite set of divisorial valuations $\mathcal V(\abul)=\{v_1,\dots,v_r\}$ such that
	$$
	\Loj_{\bbul}(\abul)=\max_{1\le i\le r}\frac{v_i(\abul)}{v_i(\bbul)},
	$$
	provided that $v_i(\bbul)>0$ for all $i$. We shall verify this positivity condition. 
	
	Indeed, since $\R(\abul)$ is Noetherian, there exists $d\ge1$ such that the $d$-th Veronese subalgebra $\R(\abul)^{(d)}$ is standard graded. Equivalently, $\R(\abul)^{(d)}$ is the ordinary Rees algebra of the monomial ideal $\a_d$. Hence, the valuations in $\mathcal V(\abul)$ are exactly the Rees valuations of $\a_d$. By \Cref{thm:facet-formula}, these are toric divisorial valuations. Thus, there exist primitive lattice vectors $w_1,\dots,w_r\in\sigma^\circ\cap N$ such that $v_i=v_{w_i}$ for all $i$.
	
	Since $\bbul$ is linearly bounded with respect to $\q$, there exists $c>0$ such that $v_u(\bbul)\ge c\,h_\q(u)$ for all $u\in\sigma^\circ$. Therefore,
	$
	v_{w_i}(\bbul)\ge c\,h_\q(w_i)>0
	$
	for every $i$, verifying the desired positive condition.

	Now, each $w_i$ lies in $\sigma^\circ\cap N$, so
	$
	\frac{v_{w_i}(\abul)}{v_{w_i}(\bbul)}
	\le
	\sup\limits_{u\in\sigma^\circ\cap N}\frac{v_u(\abul)}{v_u(\bbul)}
	$
	for every $i$. Taking the maximum over $i$, we obtain the reverse inequality
	$$
	\Loj_{\bbul}(\abul)
	\le
	\sup_{u\in\sigma^\circ\cap N}\frac{v_u(\abul)}{v_u(\bbul)}.
	$$
	The theorem is proved.
\end{proof}

\begin{remark}[Mechanism behind finite testing in the toric case]\label{rem:finite-candidates-filtrations}
	The finite--max formulas obtained in this section are not a consequence of toric geometry
	per se, but of the underlying polyhedral structure.
	In the monomial setting, Newton polyhedra encode integral--closure containments by linear
	inequalities, toric valuations correspond to linear functionals on these polyhedra, and
	support functions are piecewise linear with respect to a finite fan.
	As a result, the valuative optimization problem defining the Lojasiewicz exponent reduces
	to a finite maximization over extremal directions, namely the primitive inward normals of
	compact facets.
	This mechanism provides the prototypical example of the general ``finite testing'' principle
	developed earlier in the paper.
\end{remark}

\begin{remark}[Limits of the toric mechanism]\label{rem:toric-filtration-finite-max}
	The explicit finite testing principle established here relies crucially on the existence of
	a polyhedral model controlling all relevant integral--closure containments on a fixed toric
	modification.
	Outside the toric or monomial setting, such a uniform polyhedral structure need not exist:
	although valuative formulas remain valid in great generality, the set of divisorial valuations
	that compute the Lojasiewicz exponent may depend on the chosen principalization and need
	not admit an a priori finite reduction.
	Section~\ref{sec:verification-uniformity-sharpness} investigates to what extent Newton--type hypotheses recover finite testing
	and explains why additional verifiable assumptions are necessary beyond the toric case.
\end{remark}

\section{\L{}ojasiewicz exponent at infinity and Newton nondegeneracy}\label{sec:NND}

This section focuses on the \L{}ojasiewicz exponent \emph{at infinity} for polynomial mappings.
Particularly, we will consider:
(i) a \emph{single fixed birational model} on which the relevant divisor ideals become invertible, yielding
finite-minimum formulas for \L{}ojasiewicz exponent at infinity (\Cref{thm:13.3-local-min,thm:13.2-global-min}); and
(ii) in Newton nondegenerate settings, a \emph{toric} model that makes the minimizing divisors explicit,
recovering classical Newton-type formulas (\Cref{thm:newton-finite-infty} and \Cref{cor:complex-newton-general}).

\subsection{Finiteness of the algebraic \L{}ojasiewicz exponent at infinity}\label{subsec:13.2}

Even though, in Definition~\ref{def:local-analytic-loj-infty} (and \Cref{thm:analytic-algebraic-local-infty}), the local \L{}ojasiewicz exponent at infinity of a polynomial mapping $F$, $\Loj^{\text{an}}_{\infty,\xi}(F)$, is defined in the local ring
$\cO_{\overline{\Gamma}_F,\xi}$, there exists a \emph{uniform} description after passing to a single normal
proper birational model on which the divisor ideals $\mathscr I_X$ and $\mathscr I_Y$ become invertible.
On such a model, integral--closure containments are equivalent to finitely many divisorial inequalities,
so the local and global exponents at infinity admit a finite \emph{minima} formulation similar to what we have seen for the exponent at an isolated critical point (\Cref{thm:finite-max-rees}).

\smallskip
\noindent\textit{Notations.} We keep the notation from \Cref{subsec:algebraic-infinity}.
In particular, $\overline{\Gamma}_F\subseteq \PP^n_\CC\times \PP^m_\CC$ denotes the normalization of the
Zariski closure of the graph of $F$, with projections $\pi,\rho$ and divisor ideals
$\mathscr I_X=\cI_{D_X}$ and $\mathscr I_Y=\cI_{D_Y}$.

Since $\overline{\Gamma}_F$ is projective over $\CC$, there exists a proper birational morphism
\begin{equation}\label{eq:13.7-fixed-model}
	\mu:Z \longrightarrow \overline{\Gamma}_F
\end{equation}
with $Z$ normal such that both $\mathscr I_X\cO_Z$ and $\mathscr I_Y\cO_Z$ are invertible
(for instance, take a principalization of $\mathscr I_X\cdot \mathscr I_Y$).
Write
\begin{equation}\label{eq:13.8-divisors}
	\mathscr I_X\cO_Z=\cO_Z(-A) \ \text{ and }\ \mathscr I_Y\cO_Z=\cO_Z(-B),
\end{equation}
where $A$ and $B$ are effective Cartier divisors on $Z$.
For any prime divisor $E\subseteq Z$, set
\[
\ord_E(\mathscr I_X):=\text{coeff}_E(A) \quad\text{and}\quad \ord_E(\mathscr I_Y):=\text{coeff}_E(B).
\]
Thus, $\ord_E(\mathscr I_X)>0$ if and only if $E\subseteq \Supp(A)$.

\medskip
\noindent\textit{Local exponents on the fixed model.}
Fix $\xi\in D_X$ and, as before, set $R_\xi:=\cO_{\overline{\Gamma}_F,\xi}$,
$\mathscr I_{X,\xi}:=\mathscr I_X\cdot R_\xi$, and $\mathscr I_{Y,\xi}:=\mathscr I_Y\cdot R_\xi$.
Base-changing $\mu$ to $\Spec(R_\xi)$ yields a proper birational morphism
\[
\mu_\xi:Z_\xi:= Z \times_{\overline{\Gamma}_F}\Spec(R_\xi)\longrightarrow \Spec(R_\xi),
\]
with $Z_\xi$ normal, such that $\mathscr I_{X,\xi}\cO_{Z_\xi}$ and $\mathscr I_{Y,\xi}\cO_{Z_\xi}$
are invertible. Denote by $A_\xi$ and $B_\xi$ the pullbacks of $A$ and $B$, so that
\[
\mathscr I_{X,\xi}\cO_{Z_\xi}=\cO_{Z_\xi}(-A_\xi) \text{ and }
\mathscr I_{Y,\xi}\cO_{Z_\xi}=\cO_{Z_\xi}(-B_\xi).
\]
For our purposes, the only prime divisors on $Z_\xi$ that matter are those lying over codimension--one
points of $Z$ whose image contains $\xi$; equivalently, whenever $E\subseteq Z$ is a prime divisor with
$\xi\in\mu(E)$, its pullback to $Z_\xi$ has codimension--one points, and testing integral-closure
containments on $R_\xi$ is reduced to testing the corresponding divisorial orders along such $E$.

\begin{theorem}[Local finite--minimum formula]\label{thm:13.3-local-min}
	Fix $\xi\in D_X$. With notation as above, one has
	\begin{equation}\label{eq:13.9-local-min}
		\Loj^{\text{an}}_{\infty,\xi}(F)
		=
		\min\left\{
		\frac{\ord_E(\mathscr I_Y)}{\ord_E(\mathscr I_X)}
		\ \middle|\
		E\in \Supp(A)\text{ and }\xi\in \mu(E)
		\right\}
		\ \in\ \RR_{\ge 0}\cup\{\infty\}.
	\end{equation}
\end{theorem}

\begin{proof}
	For simplicity of notation, set $r_\xi = \Loj_{\mathscr I_{Y,\xi}}(\mathscr I_{X,\xi})$, i.e.,
	\[
	r_\xi =\inf\left\{\frac{q}{p}\in\QQ_{>0}\ \middle|\
	\mathscr I_{Y,\xi}^{\,qt}\subseteq \overline{\mathscr I_{X,\xi}^{\,pt}}\ \text{ for all }t\gg 1\right\}.
	\]
	By \Cref{thm:analytic-algebraic-local-infty}, we have $\Loj^{\text{an}}_{\infty,\xi}(F)=1/r_\xi$
	(with the conventions $\inf\varnothing=\infty$, $1/\infty=0$, and $1/0=\infty$).

	Fix $\frac{q}{p}\in\QQ_{>0}$ and $t\ge 1$.
	Since $\mathscr I_{X,\xi}\cO_{Z_\xi}$ and $\mathscr I_{Y,\xi}\cO_{Z_\xi}$ are invertible, their powers are
	integrally closed on $Z_\xi$.
	Hence, by the fixed--model integral--closure criterion (see \Cref{thm:IC-reduction}),
	the containment
	\[
	\mathscr I_{Y,\xi}^{\,qt}\subseteq \overline{\mathscr I_{X,\xi}^{\,pt}}
	\quad\text{in }R_\xi
	\]
	is equivalent to the divisor inequality
	\[
	qt\,B_\xi \ \ge\ pt\,A_\xi
	\quad\text{on }Z_\xi,
	\]
	and therefore to the collection of inequalities
	\begin{equation}\label{eq:13.10-div-ineq}
		qt\cdot \ord_E(\mathscr I_Y)\ \ge\ pt\cdot \ord_E(\mathscr I_X)
		\quad\text{for every prime divisor }E\subseteq Z\text{ with }\xi\in\mu(E).
	\end{equation}
	If $E \not\subseteq \Supp(A)$ then $\ord_E(\mathscr I_X) = 0$ and the inequalities \eqref{eq:13.10-div-ineq} are vacuously true, i.e., such a prime divisor $E$ imposes no restriction on $q/p$. If $E\in \Supp(A)$ and $\xi\in\mu(E)$, then $\ord_E(\mathscr I_X)>0$.
	For such $E$, the inequalities \eqref{eq:13.10-div-ineq} are equivalent to the ratio inequalities
	\[
	\frac{q}{p}\ \ge\ \frac{\ord_E(\mathscr I_X)}{\ord_E(\mathscr I_Y)}
	\quad\text{whenever }\ord_E(\mathscr I_Y)>0,
	\]
	and if $\ord_E(\mathscr I_Y)=0$ then \eqref{eq:13.10-div-ineq} fails for every $\frac{q}{p}>0$.

	It follows that
	\[
	r_\xi
	=
	\sup\left\{
	\frac{\ord_E(\mathscr I_X)}{\ord_E(\mathscr I_Y)}
	\ \middle|\
	E\in \Supp(A)\text{ and }\xi\in \mu(E)
	\right\}
	\ \in\ [0,\infty],
	\]
	where the supremum is $\infty$ if some such $E$ satisfies $\ord_E(\mathscr I_Y)=0$.
	Taking reciprocals (with the stated conventions) yields \eqref{eq:13.9-local-min}.
\end{proof}

\smallskip
\noindent\textit{Global finiteness and a finite minimum.}
Recall that
\[
\Supp(A)=\{\,E\subseteq Z \text{ prime divisor}\mid \ord_E(\mathscr I_X)>0\,\}.
\]

\begin{theorem}[Global finite--minimum formula]\label{thm:13.2-global-min}
	With notation as above, one has
	\begin{equation}\label{eq:13.7-global-min}
		\Loj^{\text{an}}_\infty(F)
		=
		\min_{E\in \Supp(A)}\ \frac{\ord_E(\mathscr I_Y)}{\ord_E(\mathscr I_X)}
		\ \in\ \RR_{\ge 0}\cup\{\infty\}.
	\end{equation}
	In particular,
	\[
	\Loj^{\text{an}}_\infty(F)>0
	\quad\Longleftrightarrow\quad
	\ord_E(\mathscr I_Y)>0\ \text{ for every }E\in \Supp(A).
	\]
\end{theorem}

\begin{proof}
	By Definition~\ref{def:local-analytic-loj-infty},
	\[
	\Loj^{\text{an}}_\infty(F)=\inf_{\xi\in D_X}\ \Loj^{\text{an}}_{\infty,\xi}(F).
	\]
	Applying \Cref{thm:13.3-local-min} gives
	\[
	\Loj^{\text{an}}_\infty(F)
	=
	\inf_{\xi\in D_X}\ 
	\min\left\{
	\frac{\ord_E(\mathscr I_Y)}{\ord_E(\mathscr I_X)}
	\ \middle|\
	E\in \Supp(A) \text{ and }\xi\in \mu(E)
	\right\}.
	\]
	Fix $E\in \Supp(A)$. Since $\ord_E(\mathscr I_X)>0$, the image $\mu(E)$ is contained in $D_X$,
	so there exists $\xi\in D_X$ with $\xi\in \mu(E)$. Therefore, each $E\in \Supp(A)$ occurs in the
	inner minimum for some $\xi$. Since $\Supp(A)$ is finite, the outer infimum equals the minimum over
	$E\in \Supp(A)$, proving \eqref{eq:13.7-global-min}. The final equivalence is immediate.
\end{proof}

\smallskip
\noindent\textit{Positivity and attainment at infinity.}
We now specialize the general finite--max and Rees valuations machinery developed in \Cref{sec:valuative} to \L{}ojasiewicz exponents at infinity.

Fix $\xi\in D_X\cap D_Y$.
By the valuative criterion for integral closure (see \Cref{thm:valuative-ic}), for fixed $p,q,t\ge1$ the
containment $\mathscr I_{Y,\xi}^{\,qt}\subseteq \overline{\mathscr I_{X,\xi}^{\,pt}}$
holds if and only if 
$$q\,v(\mathscr I_{Y,\xi})\ \ge\ p\,v(\mathscr I_{X,\xi})
\quad\text{for every divisorial valuation }v\text{ of }R_\xi.$$
Consequently, the optimal exponent is characterized by
\begin{equation}\label{eq:meta-valuative-characterization}
	\Loj^{\text{an}}_{\infty,\xi}(F)
	=
	\inf_{v}\ \frac{v(\mathscr I_{Y,\xi})}{v(\mathscr I_{X,\xi})},
\end{equation}
where the infimum ranges over all divisorial valuations $v$ of $R_\xi$ with $v(\mathscr I_{X,\xi})>0$. In fact, as pointed out in \Cref{thm:valuative-ic}, it is enough to let $v$ range over finitely many Rees valuations of $\mathscr I_{X,\xi}$.

\begin{theorem}[Positivity and attainment]\label{thm:meta-infty-rees}
	Fix $\xi\in D_X\cap D_Y$.
	\begin{enumerate}
		\item
		$\Loj^{\text{an}}_{\infty,\xi}(F)>0$ if and only if there is no divisorial valuation $v$ of $R_\xi$
		such that
		\[
		v(\mathscr I_{X,\xi})>0
		\quad\text{and}\quad
		v(\mathscr I_{Y,\xi})=0.
		\]
		
		\item
		If $\Loj^{\text{an}}_{\infty,\xi}(F)>0$, then the infimum in
		\eqref{eq:meta-valuative-characterization} is a minimum.
		
		\item
		More precisely, if $\Loj^{\text{an}}_{\infty,\xi}(F)>0$, then $\Loj^{\text{an}}_{\infty,\xi}(F)$ is attained by
		one of the finitely many Rees valuations of the ideal
		$\mathscr I_{X,\xi} \subseteq R_\xi$.
	\end{enumerate}
\end{theorem}

\begin{proof} (1) If there exists a divisorial valuation $v$ with
	$v(\mathscr I_{X,\xi})>0$ and $v(\mathscr I_{Y,\xi})=0$, then the ratio
	$v(\mathscr I_{Y,\xi})/v(\mathscr I_{X,\xi})$ equals $0$, so
	\eqref{eq:meta-valuative-characterization} gives $\Loj^{\text{an}}_{\infty,\xi}(F)=0$.
	Conversely, if no such valuation exists, then every divisorial valuation with
	$v(\mathscr I_{X,\xi})>0$ also satisfies $v(\mathscr I_{Y,\xi})>0$, so all ratios in
	\eqref{eq:meta-valuative-characterization} are positive, hence
	$\Loj^{\text{an}}_{\infty,\xi}(F)>0$.
	
	\smallskip
	(2)--(3)
	Assume $\Loj^{\text{an}}_{\infty,\xi}(F)>0$.
	Taking reciprocals in \eqref{eq:meta-valuative-characterization} gives
	\[
	\frac{1}{\Loj^{\text{an}}_{\infty,\xi}(F)}
	=
	\sup_{v}\ \frac{v(\mathscr I_{X,\xi})}{v(\mathscr I_{Y,\xi})},
	\]
	where the supremum ranges over divisorial valuations with
	$v(\mathscr I_{Y,\xi})>0$.
	By (1) this supremum is finite.
	Applying the finite--maximum theorem for ideal pairs
	(\Cref{thm:finite-max-rees}) to the pair
	$(\mathscr I_{X,\xi},\mathscr I_{Y,\xi})$ in the Noetherian local domain $R_\xi$,
	the supremum is attained by one of the finitely many Rees valuations of the ideal $\mathscr I_{X,\xi}$.
	Taking reciprocals again shows that the infimum in
	\eqref{eq:meta-valuative-characterization} is a minimum attained by the same
	finite set of valuations.
\end{proof}

\subsection{Newton nondegeneracy and toric compactification}

We recall the notion of Newton nondegeneracy for polynomial mappings (or ideals) following Saia \cite{Saia}. This is slightly different than Kouchnirenko's \cite{Kou76}). Particularly, a function $f \in \kk[x_1, \dots, x_n]$ is Newton nondegenerate in Kouchnirenko's sense if and only if the ideal generated by $x_i\frac{\partial f}{\partial x_i}$ is Newton nondegenerate in Saia's sense.

\begin{definition}[Newton polyhedron and nondegeneracy]\label{def:newton-infty}
	Write $F=(f_1,\dots,f_r)$ with $f_j\in \kk[x_1,\dots,x_n]$.
	For $f=\sum_{\alpha\in\NN^n} c_\alpha x^\alpha$ with $c_\alpha\neq 0$, set
	\[
	\Supp(f):=\{\alpha\in\NN^n: c_\alpha\neq 0\}.
	\]
	Define the Newton polyhedron of $F$ (or equivalently, of the ideal generated by $F$) by
	\[
	\Gamma(F):=\mathrm{conv}\Bigl(\bigcup_{j=1}^r\Supp(f_j)\cup\{0\}\Bigr)+\RR_{\ge 0}^n
	\ \subseteq\ \RR^n,
	\]
	and for $w\in\RR_{\ge 0}^n$ define
	\[
	m_w(f):=\min_{\alpha\in\Supp(f)}\langle w,\alpha\rangle,\quad
	f_w:=\sum_{\alpha\in\Supp(f)\,:\,\langle w,\alpha\rangle=m_w(f)} c_\alpha x^\alpha.
	\]
	We say that $F$ is \emph{Newton nondegenerate at infinity} if for every $w\in\RR_{>0}^n$ the system
	\[
	(f_1)_w=\cdots=(f_r)_w=0
	\]
	has no solution in $(\kk^*)^n$.
\end{definition}

We will use a toric compactification adapted to the Newton polyhedron at infinity following Khovanskii \cite{Khovanskii}.  Let $\Sigma_{\Gamma}$ denote the normal fan of $\Gamma(F)$ and let $\Sigma_{\PP^n}$ denote the standard fan of $\PP^n$.  Choose a smooth complete fan $\Sigma$ refining both $\Sigma_{\Gamma}$
and $\Sigma_{\PP^n}$ (such a refinement exists by toric resolution of singularities;
see \cite{Khovanskii,Oka}, \cite[Chapter 11]{CLS} and \cite[Chapter 2]{Fulton}).  Then, the associated toric variety $X_\Sigma$ is a smooth projective toric compactification of the torus $T=(\kk^*)^n$, and the refinement
$\Sigma\succ \Sigma_{\PP^n}$ induces a proper toric morphism $\pi_\Sigma:X_\Sigma\longrightarrow \PP^n$
extending the standard open embedding $T\hookrightarrow \PP^n$ (see \cite[Chapter 3]{CLS} and \cite[Chapters 1--2]{Fulton}).  Moreover, since
$\Sigma$ also refines the normal fan of $\Gamma(F)$, the toric boundary divisors of
$X_\Sigma$ correspond to weight vectors governing the face systems (initial forms)
of $F$, which control the asymptotic behavior at infinity as in Khovanskii's
compactification.

\begin{notation}[Khovanskii compactification]\label{notation:Khovanskii}
	Let $\Sigma$ be a smooth complete fan refining both the normal fan of $\Gamma(F)$
	and the fan of $\PP^n$, and let
	\[
	\pi_\Sigma:X_\Sigma\to \PP^n
	\]
	be the induced toric morphism.  For each cone $\sigma\in\Sigma$, let
	$O(\sigma)\subseteq X_\Sigma$ denote the corresponding torus orbit.
	Let $\Gamma_F^\Sigma\subseteq X_\Sigma\times \PP^r$ denote the normalization of the
	closure of the graph of $F$.  Write
	\[
	\pi_\Sigma:\Gamma_F^\Sigma\to X_\Sigma \text{ and }
	\rho_\Sigma:\Gamma_F^\Sigma\to \PP^r
	\]
	for the two natural projections.
\end{notation}

\begin{lemma}[Values along divisors at infinity]\label{lem:bridge-nnd}
	Assume that $F:\AA^n_\kk\to\AA^r_\kk$ is Newton nondegenerate at infinity in the sense of
	Definition~\ref{def:newton-infty}.
	Let $\xi\in D_X\cap D_Y$ and set $R_\xi:=\cO_{\overline{\Gamma}_F,\xi}$.
	Then, for every divisorial valuation $v$ of $R_\xi$,
	\[
	v(\mathscr I_{X,\xi})>0\quad\Longrightarrow\quad v(\mathscr I_{Y,\xi})>0.
	\]
	Equivalently, there is no divisorial valuation $v$ with $v(\mathscr I_{X,\xi})>0$ and $v(\mathscr I_{Y,\xi})=0$.
\end{lemma}

\begin{proof}
	We will work with the notations in \Cref{def:newton-infty} and \Cref{subsec:algebraic-infinity}.
	
	\smallskip
	\noindent\emph{Step 1: toric boundary data and face systems.}
	Let $\Sigma$ and $\pi_\Sigma:X_\Sigma\to \PP^n$ be as above (Khovanskii's compactification).
	Write $D_{\tor}:=X_\Sigma\setminus(\kk^*)^n$ for the toric boundary.
	Each torus--invariant prime divisor $E\subseteq D_{\tor}$ corresponds to a primitive
	ray vector $w_E\in \ZZ_{\ge 0}^n\setminus\{0\}$.
	For every Laurent monomial $x^\alpha$ one has $\ord_E(x^\alpha)=\langle w_E,\alpha\rangle$,
	and for every polynomial $f$ one has $\ord_E(f)=m_{w_E}(f)$ with initial form $f_{w_E}$
	on the orbit $O(E) := O(w_E)$.

	\smallskip
	\noindent\emph{Step 2: boundary intersections of the graph and Newton nondegeneracy.}
	Let $\Gamma_F^\Sigma\subseteq X_\Sigma\times \PP^r$ be as before, and let
	$\mu_\Sigma:\Gamma_F^\Sigma\to \overline{\Gamma}_F$ be the induced proper birational morphism.
	We first claim that near $\mu_\Sigma^{-1}(\xi)$, one has
    \begin{align}
    \Supp (\mathscr{I}_X \cO_{\Gamma_F^\Sigma}) \subseteq \Supp (\mathscr{I}_Y \cO_{\Gamma_F^\Sigma}). \label{eq.supportGamma}
    \end{align}
    
    Indeed, consider any point $z \in \Supp(\mathscr{I}\cO_{\Gamma_F^\Sigma})$. Then, $\mu_\Sigma(z) \in D_X$, so the image of $z$ in $X_\Sigma$ lies over the hyperplane at infinity $H_X$. Let $O(\sigma) \subseteq X_\Sigma$ be the torus orbit containing the image of $z$, where $\sigma \in \Sigma$. Since $O(\sigma)$ maps into $H_X$, the cone $\sigma$ is not contained in any coordinate face
	of $\RR_{\ge0}^n$. Equivalently, $\text{relint}(\sigma)\cap \RR_{>0}^n\neq \varnothing.$ Choose $w\in \text{relint}(\sigma)\cap \RR_{>0}^n.$

We shall work in the affine toric chart $U_\sigma=\Spec k[\sigma^\vee\cap M]\subseteq X_\Sigma$ containing $O(\sigma)$. In toric coordinates on $U_\sigma$, each polynomial $f_j$ has an
	expansion of the form
	\[
	f_j=u^{m_j}\Bigl((f_j)_w(u)+\sum_{\langle w,\alpha\rangle>m_j} c_{j,\alpha}u^\alpha\Bigr),
	\]
where $m_j = m_w(f_j) = \min_{\alpha \in \Gamma(f_j)}\langle w,\alpha\rangle$ and $(f_j)_w$ is the face polynomial consisting of those terms achieving this minimum.
	By the standard initial-form criterion for strict transforms in toric compactifications
	(see, for example, \cite{Khovanskii,Oka}), the graph closure and its normalization in $U_\sigma \times \AA^r$ meet $O(\sigma)\times \{W_0\neq 0\}$ if and only if the face system
	\[
	(f_1)_w=\cdots=(f_r)_w=0
	\]
	has a solution in the torus $O(\sigma)\cong (k^*)^{\,n-\dim\sigma}.$
	Since $w\in \RR_{>0}^n$ and $F$ is Newton nondegenerate at infinity, this system has no
	solution. Therefore, we conclude that
	\[
	\Gamma_F^\Sigma\cap \bigl(O(\sigma)\times\{W_0\neq 0\}\bigr)=\varnothing.
	\]
	
	Now, suppose that $z\notin \Supp (\mathscr{I}_Y\cO_{\Gamma_F^\Sigma}).$
	Then,$\rho_\Sigma(z)\notin H_Y$, equivalently, $W_0\neq 0$ at $z$. Since the image of $z$ in
	$X_\Sigma$ lies in $O(\sigma)$, this would give $z\in \Gamma_F^\Sigma\cap \bigl(O(\sigma)\times\{W_0\neq 0\}\bigr),$ a contradiction. Hence, we necessarily have
	\[
	z\in \Supp (\mathscr{I}_Y\cO_{\Gamma_F^\Sigma}).
	\]
	Since $z$ was arbitrary, this proves (\ref{eq.supportGamma}) near $\mu_\Sigma^{-1}(\xi)$.
	
	\smallskip
	\noindent\emph{Step 3: descend to divisorial valuations over $\xi$.}
	Let $v$ be a divisorial valuation of $R_\xi$ with $v(\mathscr{I}_{X,\xi})>0.$
	Choose a normal modification $\nu:W\to \overline{\Gamma}_F$ and a prime divisor $E'\subseteq W$
	with center $\xi$ such that $v=\ord_{E'}$ on $R_\xi$.
	Take a common normal refinement $\eta: W'\to W$ and $\tau: W'\to \Gamma_F^\Sigma$ dominating both models,
	and let $\widetilde E'\subseteq W'$ be the strict transform of $E'$.
	
Since $\mathscr{I}_X$ is the ideal sheaf of the Cartier divisor $D_X$ on $\overline{\Gamma}_F$, its
	pullback to $W'$ is invertible, and
	\[
	\ord_{\widetilde E'}(\tau^*\mathscr{I}_X)
	=
	\ord_{E'}(\mathscr{I}_{X,\xi})
	=
	v(\mathscr{I}_{X,\xi})
	>0.
	\]
	Therefore, the generic point of $\widetilde E'$ maps into $\Supp (\mathscr{I}_X\cO_{\Gamma_F^\Sigma}).$
	By (\ref{eq.supportGamma}), near $\mu_\Sigma^{-1}(\xi)$ this support is contained in $\Supp (\mathscr{I}_Y\cO_{\Gamma_F^\Sigma}).$ Hence, the generic point of $\widetilde E'$ also maps into $\Supp (\mathscr{I}_Y\cO_{\Gamma_F^\Sigma}).$ This implies that $\ord_{\widetilde E'}(\tau^*\mathscr{I}_Y)>0.$

Moreover, $\mathscr{I}_Y$ is also the ideal sheaf of a Cartier divisor on $\overline{\Gamma}_F$. Thus,
	\[
	\ord_{\widetilde E'}(\tau^*\mathscr{I}_Y)
	=
	\ord_{E'}(\mathscr{I}_{Y,\xi})
	=
	v(\mathscr{I}_{Y,\xi}).
	\]
	It follows that $v(\mathscr{I}_{Y,\xi})>0$, as desired.
The second statement of the lemma is immediate.
\end{proof}

\begin{corollary}\label{cor:meta-infty-nondeg-finite}
	Assume that $F$ is Newton nondegenerate at infinity.
	Then, for every $\xi\in D_X\cap D_Y$ one has
	\[
	\Loj^{\text{an}}_{\infty,\xi}(F) > 0.
	\]
\end{corollary}

\begin{proof}
	Fix $\xi\in D_X\cap D_Y$. By Lemma~\ref{lem:bridge-nnd}, there is no divisorial valuation $v$ of $R_\xi$
	with $v(\mathscr I_{X,\xi})>0$ and $v(\mathscr I_{Y,\xi})=0$. The conclusion follows immediately from
	Theorem~\ref{thm:meta-infty-rees}(1).
\end{proof}

\begin{remark}\label{rem:meta-infty-to-toric}
	At this stage, \Cref{thm:meta-infty-rees} guarantees finiteness and attainment by a finite set of
	Rees valuations of $\mathscr I_{X,\xi}$. In the Newton nondegenerate regime,
	we can do better: we identify the relevant divisors with explicit toric boundary divisors arising
	from a toric compactification adapted to $\Gamma(F)$, leading to concrete Newton polyhedral
	formulas.
\end{remark}

\subsection{Newton nondegenerate at infinity and \L{}ojasiewicz exponent}\label{subsec:newton-recovery}

We will show how classical formulas for \L{}ojasiewicz exponents at infinity in the Newton nondegenerate case are recovered within our algebraic framework. Let $F:\AA^n_\kk\to\AA^r_\kk$ denote a polynomial mapping, that is Newton nondegeneracy at infinity as in Definition~\ref{def:newton-infty}.

\begin{proposition}\label{prop:toric-control-at-infty}
	Assume that $F$ is Newton nondegenerate at infinity. Then, there exists a toric compactification
	$X_\Sigma$ adapted to $\Gamma(F)$ and a proper birational morphism
	\[
	\mu:Z \longrightarrow \overline{\Gamma}_F
	\]
	with $Z$ smooth such that:
	\begin{enumerate}
		\item the reduced divisor $(D_X)_\text{red}+(D_Y)_\text{red}$ pulls back to an SNC divisor
		\[
		E_1+\cdots+E_N
		\quad\text{on } Z;
		\]
		\item for every $\xi\in D_X\cap D_Y$, after base change to $\Spec(R_\xi)$ the ideals
		$\mathscr I_{X,\xi}$ and $\mathscr I_{Y,\xi}$ pull back to invertible ideals supported on the
		pullbacks of those $E_i$ meeting $\mu^{-1}(\xi)$.
	\end{enumerate}
\end{proposition}

\begin{proof}
	Choose a toric compactification $X_\Sigma$ adapted to $\Gamma(F)$.
	Since $X_\Sigma$ is proper over $\PP^n$, the normalization $\Gamma_F^\Sigma$ of the graph closure in
	$X_\Sigma\times \PP^r$ is proper over $\overline{\Gamma}_F$, and so there is a proper birational morphism
	$\Gamma_F^\Sigma\to \overline{\Gamma}_F$.
	Now take a log resolution of the reduced divisor
	\[
	\bigl(\mu_\Sigma^{-1}(D_X)\bigr)_\text{red} + \bigl(\mu_\Sigma^{-1}(D_Y)\bigr)_\text{red}
	\]
	on $\Gamma_F^\Sigma$; composing this resolution with $\Gamma_F^\Sigma\to\overline{\Gamma}_F$ produces
	a proper birational morphism $\mu:Z \to\overline{\Gamma}_F$ with $Z$ smooth, such that the reduced pullback
	$(D_X)_\text{red}+(D_Y)_\text{red}$ has SNC support. Denote its irreducible components by $E_1,\dots,E_N$, proving (1).
	
For (2), fix $\xi\in D_X\cap D_Y$ and set
$Z_\xi:=Z\times_{\overline{\Gamma}_F}\Spec(R_\xi)$.
Since $D_X$ and $D_Y$ are Cartier divisors on $\overline{\Gamma}_F$,
their ideal sheaves $\mathscr I_X$ and $\mathscr I_Y$ are invertible.
Hence, their pullbacks to $Z_\xi$ remain invertible ideal sheaves and define
Cartier divisors supported on the pullback of
$(D_X)_{\mathrm{red}}+(D_Y)_{\mathrm{red}}$.

By (1), the divisor $(D_X)_{\mathrm{red}}+(D_Y)_{\mathrm{red}}$ is simple normal
crossings on $Z$, with irreducible components $E_1,\dots,E_N$.
Thus, locally on $Z$, it is defined by an equation $z_1\cdots z_r=0$, where
$z_j=0$ cuts out one of the components $E_i$. After base change to
$\Spec(R_\xi)$, the pullback divisor is still locally defined by
$z_1\cdots z_r=0$, so its irreducible components are exactly the pullbacks
$E_i\times_{\overline{\Gamma}_F}\Spec(R_\xi)$ for those $i$ with
$E_i\cap \mu^{-1}(\xi)\neq\varnothing$. Therefore the pullbacks of
$\mathscr I_X$ and $\mathscr I_Y$ to $Z_\xi$ are invertible ideal sheaves
whose supports are contained in those components.
\end{proof}

\begin{theorem}\label{thm:newton-finite-infty}
	Assume that $F$ is Newton nondegenerate at infinity. Then, for every $\xi\in D_X\cap D_Y$ one has
	\[
	0<\Loj^{\text{an}}_{\infty,\xi}(F)<\infty.
	\]
	Moreover, on the fixed model $Z$ of Proposition~\ref{prop:toric-control-at-infty}, one has
	\[
	\Loj^{\text{an}}_{\infty,\xi}(F)
	=
	\min\left\{
	\frac{\ord_{E_i}(\mathscr I_Y)}{\ord_{E_i}(\mathscr I_X)}
	\ \middle|\
	1\le i\le N,\ \xi\in \mu(E_i)
	\right\}.
	\]
\end{theorem}

\begin{proof}
	Fix $\xi\in D_X\cap D_Y$. By Proposition~\ref{prop:toric-control-at-infty}, on the base change of $Z$
	to $\Spec(R_\xi)$ both $\mathscr I_{X,\xi}$ and $\mathscr I_{Y,\xi}$ pull back to invertible ideals
	supported on the strict transforms of $E_1,\dots,E_N$ meeting $\mu^{-1}(\xi)$. Therefore, the local
	finite--minimum formula Theorem~\ref{thm:13.3-local-min} applies on this fixed model and yields
	\[
	\Loj^{\text{an}}_{\infty,\xi}(F)
	=
	\min\left\{
	\frac{\ord_E(\mathscr I_Y)}{\ord_E(\mathscr I_X)}
	\ \middle|\
	E \subseteq Z \text{ prime divisor with }\ord_E(\mathscr I_X)>0,\ \xi\in\mu(E)
	\right\}.
	\]
	Since every such $E$ is among the components $E_i$, this gives the stated formula.
	Positivity follows from Lemma~\ref{lem:bridge-nnd}: if $\ord_{E_i}(\mathscr I_X)>0$ and
	$\xi\in\mu(E_i)$, then the corresponding divisorial valuation on $R_\xi$ satisfies
	$v(\mathscr I_{X,\xi})>0$, and so $v(\mathscr I_{Y,\xi})>0$, i.e.\ $\ord_{E_i}(\mathscr I_Y)>0$.
	Thus, every ratio in the above minimum is a positive real number. Hence, $0<\Loj^{\text{an}}_{\infty,\xi}(F)<\infty$ as claimed.
\end{proof}

\begin{corollary}[Complex Newton--nondegenerate polynomial mappings]\label{cor:complex-newton-general}
	Let $F : \CC^n \to \CC^m$ be a polynomial mapping that is Newton nondegenerate at infinity in the sense
	of Definition~\ref{def:newton-infty}. Then, the \L{}ojasiewicz exponent at infinity $\Loj^{\text{an}}_\infty(F)$ is
	finite and is computed by finitely many divisorial valuations coming from a toric compactification
	adapted to the Newton polyhedron $\Gamma(F)$.
\end{corollary}

\begin{proof}
	By Newton nondegeneracy at infinity, Lemma~\ref{lem:bridge-nnd} shows that no divisorial valuation $v$
	of $\cO_{\overline{\Gamma}_F,\xi}$ satisfies $v(\mathscr I_{X,\xi})>0$ and $v(\mathscr I_{Y,\xi})=0$.
	Therefore, $\Loj^{\text{an}}_{\infty,\xi}(F)>0$ for every $\xi\in D_X\cap D_Y$ by
	Theorem~\ref{thm:meta-infty-rees}(1). The finite--minimum formula of
	Theorem~\ref{thm:newton-finite-infty} then expresses $\Loj^{\text{an}}_{\infty,\xi}(F)$, and hence $\Loj^{\text{an}}_\infty(F)$,
	as a minimum over finitely many divisorial valuations on a fixed toric model associated to
	$\Gamma(F)$.
\end{proof}

\smallskip 
When $n=2$, the toric divisorial valuations appearing in Corollary~\ref{cor:complex-newton-general}
correspond bijectively to the edges of the Newton polygon at infinity. 

\begin{corollary}[Plane polynomial mappings]\label{cor:plane-mapping-recovery}
	Let $F : \CC^2 \to \CC^2$ be a polynomial mapping that is Newton nondegenerate at infinity.
	Then, $\Loj^{\text{an}}_\infty(F)$ is finite and can be computed explicitly from the Newton polygon at infinity
	$\Gamma(F)$ by a finite set of rational numbers determined by its edges.
\end{corollary}

\begin{proof}
	By Corollary~\ref{cor:complex-newton-general}, $\Loj^{\text{an}}_\infty(F)$ is given by a finite minimum over
	divisorial valuations arising from a toric compactification adapted to $\Gamma(F)$.
	In dimension two, these divisorial valuations are exactly the toric valuations associated to the
	edges of the Newton polygon. By results in \Cref{sec:toric},
	the relevant orders $\ord_E(\mathscr I_X)$ and $\ord_E(\mathscr I_Y)$ are computed by the support
	functions of $\Gamma(F)$ along those edges. Substituting into the finite--minimum formula of
	Theorem~\ref{thm:newton-finite-infty} yields the stated Newton--polygon expression.
\end{proof}

\Cref{cor:complex-newton-general,cor:plane-mapping-recovery} recover classical results on \L{}ojasiewicz exponent at infinity in the Newton nondegenerate case; see, for instance, \cite{Ploski}.

\begin{remark}\label{rem:finite-max-infinity}
	Over $\RR$, the finiteness of \L{}ojasiewicz exponent at infinity is more subtle. This is mainly because we do not have the equivalence in \Cref{thm:analytic-algebraic-local-infty} nor a single birational model which realizes all relevant valuations. Nevertheless, similar results in Newton nondegenerate situations, where divisors at infinity are torus-invariant and relevant valuations are toric valuations associated to finitely many rays/facets (supporting hyperplanes of the corresponding Newton polyhedra), were obtained under various additional hypotheses (see, for example, \cite{VuiThao,OleRoz,Son2008}).
\end{remark}


\part{Necessity and Limits of the Hypotheses} \label{part:necessity}

In this part, we examine verification, uniformity, and sharpness issues for the
finite--max principle and for the Newton--type formulas arising from the
valuative framework.


\section{Uniformity, verification, and sharpness}\label{sec:verification-uniformity-sharpness}

This section addresses when the finite testing hypotheses from \Cref{sec:valuative,sec:toric} can be verified uniformly in families, and why these hypotheses are close to optimal.  On the positive side, in the Newton nondegenerate regime we prove the existence of a single toric modification simultaneously principalizing finitely many ideals, producing a uniform finite divisorial testing set across an entire collection of data (\Cref{thm:toric-simultaneous-principalization}); as consequences we obtain automatic finite testing for broad Newton-controlled families (\Cref{thm:automatic-finite-testing-newton}) and a genuine polyhedral wall--chamber structure governing the variation of $\Loj$ (\Cref{thm:wall-chamber}).  On the sharpness side, we construct monomial examples where successive principalizations introduce new relevant facets, showing that outside verifiable finite-testing regimes one cannot expect any a priori finite reduction of the valuative optimization problem (\Cref{thm:sharpness-monomial}).

\subsection{Newton polyhedra for ideals}\label{subsec:newton-poly-ideal}

Throughout this section we fix a regular local ring $(R,\m)$ of dimension $n$
and a regular system of parameters $x=(x_1,\dots,x_n)$.  When convenient for
notation we write $R=\Bbbk[[x_1,\dots,x_n]]$, but all constructions below depend
only on the chosen parameters; see \cite{HNP25} for more details on Newton nondegenerate ideals with respect to a regular system of parameters in a regular local ring.

\subsection{Uniform verification on a fixed toric model}\label{subsec:uniform-verification}

We now explain how, in Newton nondegenerate situations, the hypotheses of the
log-resolution criterion in Section~\ref{sec:valuative} (and its integral-closure reduction
\Cref{thm:IC-reduction}) are \emph{automatically} satisfied for a large
class of filtrations.

\begin{theorem}[Toric model simultaneously principalizing finitely many Newton nondegenerate ideals]\label{thm:toric-simultaneous-principalization}
	Let $(R,\m)$ be an excellent equicharacteristic regular local ring with chosen parameters
	$x=(x_1,\dots,x_n)$.  Let $I^{(1)},\dots,I^{(r)}\subseteq R$ be ideals such that each
	$I^{(j)}$ is Newton nondegenerate with respect to $x$ (i.e.,\ the relevant face ideals define smooth complete intersections on the torus).
	Then, there exists a proper birational morphism $\pi:Y\to \Spec R$
	such that:
	\begin{enumerate}
		\item $Y$ is regular and the exceptional divisor $E=\sum_{i=1}^N E_i$ has simple normal crossings.
		\item For each $j=1,\dots,r$, the total transform of $I^{(j)}$ is an invertible ideal:
		\[
		I^{(j)}\cO_Y\ =\ \cO_Y(-F_j),
		\quad
		F_j=\sum_{i=1}^N a_{ij}E_i,\ \ a_{ij}\in\NN.
		\]
		\item The morphism $\pi$ can be chosen toric with respect to the chosen parameters $x$,
		coming from a regular subdivision of a fan refining the Newton fans of $\Gamma(I^{(j)})$.
	\end{enumerate}
\end{theorem}

\begin{proof} We retain the notation and terminology of \cite{HNP25} for Newton nondegeneracy with respect to a chosen regular system of parameters.
	
	\medskip
	\noindent
	\emph{Step 1: reduction to the complete case.}
	Since $R$ is excellent regular local, the $\widehat R$ is faithfully flat over $R$. All Newton polyhedra and Newton nondegeneracy conditions are expressed in terms of supports with respect to the fixed regular system of parameters $x$, and so pass to $\widehat{R}$ without change. Moreover, the birational toric modifications we construct are defined combinatorially from a fan, hence descend to $R$ when they exist over $\widehat{R}$.
	Thus, we may assume $R=\Bbbk[[x_1,\dots,x_n]]$.
	
	\medskip
	\noindent
	\emph{Step 2: construction of a common toric modification.}
	For each $j$, let $\Gamma(I^{(j)})\subseteq \RR_{\ge0}^n$ be the Newton polyhedron
	with respect to the chosen parameters $x=(x_1,\dots,x_n)$,
	and let $\Sigma^{(j)}$ denote its Newton fan.
	Since we have finitely many such fans,
	there exists a common rational polyhedral refinement $\Sigma'$
	refining all $\Sigma^{(j)}$.
	Choose a regular subdivision $\Sigma$ of $\Sigma'$, and let
	\[
	\pi=\pi_\Sigma:Y=X_\Sigma\longrightarrow \Spec R
	\]
	be the associated toric morphism.
	
	Because $\Sigma$ is regular,
	$Y$ is regular and its toric boundary divisor $E=\sum_{i=1}^N E_i$
	is a simple normal crossings divisor.
	This proves (1).
	Property (3) holds by construction.
	
	\medskip
	\noindent
	\emph{Step 3: principalization of each $I^{(j)}$.}
	Fix $j\in\{1,\dots,r\}$.
	
	Since $\Sigma$ refines the Newton fan of $\Gamma(I^{(j)})$,
	Khovanskii’s toroidal resolution theorem (see \cite{Khovanskii,Oka}) for Newton nondegenerate systems applies.
	Specifically, for Newton nondegenerate ideals with respect to the parameters $x$,
	the toric modification associated to a regular subdivision of the Newton fan
	gives a good resolution and principalization.
	In particular, $I^{(j)}\cO_Y$
	is an invertible ideal sheaf supported on the toric boundary.
	
	\medskip
	\noindent
	\emph{Step 4: computation of the coefficients $a_{ij}$.}
	Let $E_1,\dots,E_N$ be the torus-invariant prime divisors of $Y$
	lying over the closed point,
	and let $w_i\in\ZZ_{\ge0}^n$ be the primitive ray vector corresponding to $E_i$.
	Then, the order of vanishing of a polynomial $f$ along $E_i$ is given by (see \cite{Khovanskii,Oka}):
	\[
	\ord_{E_i}(f)
	=
	\min_{\alpha\in\Gamma(f)}\langle w_i,\alpha\rangle.
	\]
	
	Applying this to generators of $I^{(j)}$ yields
	\[
	\ord_{E_i}\bigl(I^{(j)}\cO_Y\bigr)
	=
	\min_{\alpha\in\Gamma(I^{(j)})}\langle w_i,\alpha\rangle
	=: a_{ij}\in\NN.
	\]
	
	Since $I^{(j)}\cO_Y$ is invertible and determined by these divisorial orders,
	we obtain
	\[
	I^{(j)}\cO_Y=\cO_Y\!\left(-\sum_{i=1}^N a_{ij}E_i\right).
	\]
	
	This proves (2) for each $j$.
\end{proof}

\begin{theorem}[Automatic finite testing for Newton controlled families]\label{thm:automatic-finite-testing-newton}
	Assume the setup of \Cref{thm:toric-simultaneous-principalization}, and fix real numbers
	$\lambda_1,\dots,\lambda_r\ge 0$.  Define a graded family of ideals
	\[
	\a_t\ :=\ \overline{\prod_{j=1}^r \big(I^{(j)}\big)^{\lceil \lambda_j t\rceil}}
	\quad (t\in\RR_{\ge 0}),
	\]
	where $\overline{(\cdot)}$ denotes integral closure.
	Then, on the same model $\pi:Y\to\Spec R$ of \Cref{thm:toric-simultaneous-principalization},
	one has for every $t$:
	\[
	\a_t\ =\ \pi_*\cO_Y\Big(-\sum_{i=1}^N d_i(t)E_i\Big),
	\]
	where $d_i(t) = \sum_{j=1}^r a_{ij} \lceil \lambda_j t\rceil$.
	In particular, for any other family $\bbul$ of the same form (possibly with different coefficients),
	the containment relations $\b_q\subseteq \a_p$ are determined by the finitely many divisorial valuations
	$\{\ord_{E_i}\}_{i=1}^N$, i.e.,\ the hypotheses of the finite-testing criterion
	in \Cref{thm:finite-testing-implies-max} are satisfied with testing set $\{\ord_{E_i}\}$.
\end{theorem}

\begin{proof}
	Let $\pi:Y\to \Spec R$ be as in \Cref{thm:toric-simultaneous-principalization}. As constructed, $Y$ is regular, so $Y$ is normal. Write
	\(
	I^{(j)}\cO_Y=\cO_Y(-F_j)
	\)
	with
	\(F_j=\sum_i a_{ij}E_i\).
	Set
	\[
	D_\abul(t)\ :=\ \sum_{j=1}^r \lceil \lambda_j t\rceil\,F_j\ =\ \sum_{i=1}^N d_i(t)E_i, \text{ with } 
	d_i(t):=\sum_{j=1}^r a_{ij} \lceil \lambda_j t\rceil \in\NN.
	\]
	
	\smallskip\noindent
	\emph{Step 1: compute integral closures on a fixed principalization.}
	Since $I^{(j)}\cO_Y$ is invertible, so is $\prod_j (I^{(j)})^{\lceil \lambda_j t\rceil}\cO_Y$,
	and in fact
	\[
	\Big(\prod_{j=1}^r (I^{(j)})^{\lceil \lambda_j t\rceil}\Big)\cO_Y\ =\ \cO_Y\big(-D_\abul(t)\big).
	\]
	By the fixed-model integral-closure reduction (\Cref{thm:IC-reduction}),
	the integral closure of $\prod_j (I^{(j)})^{\lceil \lambda_j t\rceil}$ equals the pushforward
	of the corresponding invertible sheaf:
	\[
	\a_t\ =\ \overline{\prod_{j=1}^r (I^{(j)})^{\lceil \lambda_j t\rceil}}
	\ =\ \pi_*\cO_Y\big(-D_\abul(t)\big).
	\]
	
	\smallskip\noindent
	\emph{Step 2: finite testing.}
	For any $p,q$, the condition $\b_q\subseteq\a_p$ is equivalent to the inequality of divisors
	on $Y$:
	\(
	D_\bbul(q)\ \ge\ D_\abul(p)
	\),
	i.e.,  coefficientwise inequalities along $E_i$.
	Equivalently,
	\(
	\ord_{E_i}(\b_q)\ge \ord_{E_i}(\a_p)
	\)
	for all $i$.
	This is precisely the finite-testing hypothesis in \Cref{thm:finite-testing-implies-max} with testing set
	$\{\ord_{E_i}\}$.
\end{proof}

\subsection{Uniformity in families and wall--chamber decomposition}\label{subsec:wall-chamber}

We now push the previous verification to a genuine structural statement in parameter families.

\begin{theorem}[Wall--chamber stratification]\label{thm:wall-chamber}
	Assume the setup of \Cref{thm:automatic-finite-testing-newton}, and let $S \subseteq \RR^m$ be
	a parameter set.  Suppose that for each $s\in S$ we have two families
	\[
	\abul(s)=\{\a(s)_t\}_{t\ge0}
	\text{ and }
	\bbul(s)=\{\b(s)_t\}_{t\ge0},
	\]
	each of the form in \Cref{thm:automatic-finite-testing-newton}, with coefficients
	$\lambda_j=\lambda_j(s)$ depending \emph{piecewise linearly} on $s$.
	Assume moreover that the Newton polyhedra $\Gamma(I^{(j)}_s)$ are independent of $s$ and the
	nondegeneracy is uniform, so that a single toric model $\pi:Y\to\Spec R$ principalizes all
	$I^{(j)}_s$ simultaneously.
	
	Then, there exists a finite stratification $S=\bigsqcup_{\alpha} C_\alpha$
	by locally closed semialgebraic subsets such that on each stratum $C_\alpha$ one of the following holds:
	\begin{enumerate}
		\item[(i)] $\Loj_{\bbul(s)}(\abul(s))=\infty$ for all $s\in C_\alpha$;
		\item[(ii)] there exists a nonempty subset $J(\alpha)\subseteq \{1,\dots,N\}$
		such that for every $s\in C_\alpha$,
		\[
		\Loj_{\bbul(s)}(\abul(s))
		=
		\max_{i\in J(\alpha)}
		\frac{\ord_{E_i}(\abul(s))}{\ord_{E_i}(\bbul(s))},
		\]
		and the maximizer set
		\[
		\left\{
		i\in J(\alpha)\ \middle|\
		\frac{\ord_{E_i}(\abul(s))}{\ord_{E_i}(\bbul(s))}
		=
		\Loj_{\bbul(s)}(\abul(s))
		\right\}
		\]
		is independent of $s\in C_\alpha$.
	\end{enumerate}
	In particular, on every stratum $C_\alpha$ for which the maximizer set is a singleton
	$\{i(\alpha)\}$, one has
	\[
	\Loj_{\bbul(s)}(\abul(s))
	=
	\frac{\ord_{E_{i(\alpha)}}(\abul(s))}{\ord_{E_{i(\alpha)}}(\bbul(s))}
	\quad\text{for all } s\in C_\alpha.
	\]
	Hence, on such a stratum the function
	\[
	s\longmapsto \Loj_{\bbul(s)}(\abul(s))
	\]
	is given by a single ratio of affine linear functions.
\end{theorem}

\begin{proof}
	Fix the common model $\pi:Y\to\Spec R$ as in \Cref{thm:automatic-finite-testing-newton} and its exceptional prime divisors $E_1,\dots,E_N$.
	
	\smallskip\noindent
	\emph{Step 1: affine linearity of the asymptotic divisor coefficients.}
	By \Cref{thm:automatic-finite-testing-newton}, for each $s\in S$ and each $t\ge0$ we have
	\[
	\a(s)_t=\pi_*\cO_Y\bigl(-D^{(s)}_\a(t)\bigr) \text{ and }
	\b(s)_t=\pi_*\cO_Y\bigl(-D^{(s)}_\b(t)\bigr),
	\]
	where
	\[
	D^{(s)}_\a(t)=\sum_j \lceil \lambda_{\a,j}(s)t\rceil\,F_{\a,j} \text{ and }
	D^{(s)}_\b(t)=\sum_j \lceil \lambda_{\b,j}(s)t\rceil\,F_{\b,j},
	\]
with $F_{\a,j}=\sum_{i=1}^N a_{\a,ij}E_i \text{ and } F_{\b,j}=\sum_{i=1}^N a_{\b,ij}E_i.$
	Hence,
	\[
	\ord_{E_i}\bigl(\a(s)_t\bigr)=\sum_j a_{\a,ij}\,\lceil \lambda_{\a,j}(s)t\rceil \text{ and }
	\ord_{E_i}\bigl(\b(s)_t\bigr)=\sum_j a_{\b,ij}\,\lceil \lambda_{\b,j}(s)t\rceil.
	\]
	Dividing by $t$ and letting $t\to\infty$, we obtain
	\[
	\ord_{E_i}\bigl(\abul(s)\bigr)=\sum_j a_{\a,ij}\lambda_{\a,j}(s) \text{ and }
	\ord_{E_i}\bigl(\bbul(s)\bigr)=\sum_j a_{\b,ij}\lambda_{\b,j}(s).
	\]
	Since the coefficient functions $\lambda_{\a,j}(s)$ and $\lambda_{\b,j}(s)$ are piecewise linear,
	it follows that each function
	\[
	s\longmapsto \ord_{E_i}\bigl(\abul(s)\bigr) \text{ and }
	s\longmapsto \ord_{E_i}\bigl(\bbul(s)\bigr)
	\]
	is piecewise affine linear, and so continuous on each member of a finite decomposition of $S$.
	
	\smallskip\noindent
	\emph{Step 2: separate the infinite-valued region.}
	Refining the above decomposition if necessary, we may assume that for each fixed $i$ the sign of $\ord_{E_i}\bigl(\bbul(s)\bigr)$
	is constant on each piece. Since these quantities are nonnegative, on each piece $P$ there is a well-defined subset
	\[
	J_P:=\bigl\{\,i\in\{1,\dots,N\}\mid \ord_{E_i}\bigl(\bbul(s)\bigr)>0
	\text{ for all } s\in P\,\bigr\}.
	\]
	If on such a piece $P$ there exists $i\notin J_P$ such that $\ord_{E_i}\bigl(\abul(s)\bigr)>0 \text{ for all } s\in P,$
	then for every $s\in P$, one has
	\[
	\ord_{E_i}\bigl(\bbul(s)\bigr)=0 \text{ and }
	\ord_{E_i}\bigl(\abul(s)\bigr)>0.
	\]
	By the finite-testing criterion in \Cref{thm:automatic-finite-testing-newton}, no finite slope can satisfy the containment inequalities along $E_i$. Hence,
	\[
	\Loj_{\bbul(s)}(\abul(s))=\infty \text{ for all } s\in P.
	\]
	
	\smallskip\noindent
	\emph{Step 3: reduction to \Cref{thm:stratification} on the finite-valued region.}
	Now restrict to a piece $P$ on which no such index occurs. Then, for every $i\notin J_P$ one has
	\[
	\ord_{E_i}\bigl(\bbul(s)\bigr)=0 \quad \Longrightarrow \quad \ord_{E_i}\bigl(\abul(s)\bigr)=0 \quad \text{ for all } s\in P.
	\]
	Thus, only the divisors $E_i$ with $i\in J_P$ contribute to the finite-testing formula of
	\Cref{thm:automatic-finite-testing-newton}. 
	
	For each $s\in P$, define
	$
	\alpha_i(s):=\ord_{E_i}\bigl(\abul(s)\bigr) \text{ and }
	\beta_i(s):=\ord_{E_i}\bigl(\bbul(s)\bigr)
	\quad (i\in J_P).
	$
	Then,
	\[
	\Loj_{\bbul(s)}(\abul(s))
	=
	\max_{i\in J_P}\frac{\alpha_i(s)}{\beta_i(s)},
	\]
	and, by construction, each $\alpha_i$ and $\beta_i$ is continuous on $P$, while
	$
	\beta_i(s)\in (0,\infty) \text{ for all } i\in J_P,\ s\in P.
	$
	Therefore, the hypotheses of \Cref{thm:stratification} are satisfied on $P$ with finite candidate set
	$
	V_P:=\{\ord_{E_i}\mid i\in J_P\}.
	$
	Applying \Cref{thm:stratification}, we obtain a finite partition of $P$ into locally closed strata on each of which the maximizer set is constant. In particular, on each such stratum the function
	$
	s\longmapsto \Loj_{\bbul(s)}(\abul(s))
	$
	is given by a single valuative ratio whenever the maximizer set is a singleton.
	
	Taking the common refinement over all pieces $P$ gives the required finite stratification of $S$.
\end{proof}


\subsection{Sharpness: why one needs verifiable hypotheses}\label{subsec:sharpness}

We close with an explicit sharpness statement showing that without a uniform principalization
mechanism, no \emph{a priori} fixed finite testing set can govern all containments in a large class.

\begin{theorem}[Sharpness via monomial ideals with new facets]\label{thm:sharpness-monomial}
	Let $R=\Bbbk[[x,y]]$.  For each integer $N\ge 2$, set
	\[
	J_N\ :=\ \overline{(x^N,\ y^{N^2})}\ \subseteq R.
	\]
	Let $\mathcal V$ be any finite set of real-valued rank-one valuations on $R$ centered at $\m$.
	Then, there exist integers $N\ge 2$ and $p,q\ge 1$ such that:
	\begin{enumerate}
		\item For every $v\in \mathcal V$, one has $v(\m^q)\ge v(J_N^p)$.
		\item Nevertheless, $\m^q\not\subseteq J_N^p = \overline{J_N^p}$.
	\end{enumerate}
	In particular, no single finite set of valuations can serve as a universal testing set for all
	containments of the form $\m^q\subseteq \overline{J_N^p}$ as $N$ varies.
\end{theorem}

\begin{proof}
	Fix a finite set $\mathcal V$ as in the statement.  For each $v\in\mathcal V$, set
	$a_v:=v(x)>0$ and $b_v:=v(y)>0$.
	
	\smallskip\noindent
	\emph{Step 1: choose $N$ so that $v(J_N)=Na_v$ for all $v\in\mathcal V$.}
	For $J_N=\overline{(x^N,y^{N^2})}$, the valuative criterion for integral closure gives
	\[
	v(J_N)\ =\ \min\{N a_v,\ N^2 b_v\}.
	\]
	Choose an integer $N\ge 2$ such that
	\[
	N\ >\ 1+\max_{v\in\mathcal V}\frac{a_v}{b_v}.
	\]
	Then, $N b_v>a_v$ for all $v\in\mathcal V$, hence $N^2 b_v>Na_v$ and, therefore,
	\[
	v(J_N)=Na_v\quad\text{for all }v\in\mathcal V.
	\]
	
	\smallskip\noindent
	\emph{Step 2: choose $q$ so that $v(\m^q)\ge v(J_N^p)$ for all $v\in\mathcal V$.}
	Set
	\[
	M\ :=\ \max_{v\in\mathcal V}\left\lceil \frac{a_v}{\min\{a_v,b_v\}}\right\rceil.
	\]
	Then, $M\ge 1$, and by the choice of $N$ we have $\frac{a_v}{b_v}<N-1$ for all $v$, and so
	\[
	\frac{a_v}{\min\{a_v,b_v\}}\ \le\ \max\left\{1,\frac{a_v}{b_v}\right\}\ <\ N-1,
	\]
	so in particular $M\le N-1$.
	
	Now choose any integer $p\ge 1$ and set $q\ :=\ p N M.$
	Since $v(\m)=\min\{a_v,b_v\}$ and $v(J_N)=Na_v$ for all $v\in\mathcal V$, we obtain
	\[
	v(\m^q)\ =\ q\,\min\{a_v,b_v\}\ \ge\ pN\cdot \frac{a_v}{\min\{a_v,b_v\}}\cdot \min\{a_v,b_v\}
	\ =\ pN a_v\ =\ v(J_N^p)
	\]
	for every $v\in\mathcal V$.  This proves~(1).
	
	\smallskip\noindent
	\emph{Step 3: explicit witness that $\m^q\not\subseteq J_N^p$.}
	We claim that $y^q\in\m^q$ but $y^q\notin J_N^p$.  The first inclusion is clear.
	
	For the second, note that $J_N$ is a monomial integrally closed ideal, and membership in $J_N^p$
	is determined by its Newton polyhedron.  The Newton polyhedron of $J_N$ is given by
	\[
	\frac{u}{N}+\frac{v}{N^2}\ \ge\ 1,
	\]
	hence that of $J_N^p$ is
	\[
	\frac{u}{N}+\frac{v}{N^2}\ \ge\ p.
	\]
	For $y^q$, the exponent vector is $(u,v)=(0,q)$, so
	\[
	\frac{u}{N}+\frac{v}{N^2}\ =\ \frac{q}{N^2}\ =\ \frac{pNM}{N^2}\ =\ p\cdot\frac{M}{N}.
	\]
	Since $M\le N-1$, we have $\frac{M}{N}<1$, and so $\frac{q}{N^2}<p$. Therefore, $y^q\notin J_N^p$.
	Consequently, $\m^q\not\subseteq J_N^p$, proving~(2).
	
	This shows that no fixed finite set of valuations can test all containments of the form
	$\m^q\subseteq J_N^p = \overline{J_N^p}$ uniformly as $N$ varies.
\end{proof}


\part{Outlook and the Appendix} \label{part:outlook}

In this final part, we collect open questions and directions suggested by the
valuative framework developed in this paper.


\section{Concluding remarks and open problems}

This paper develops a valuative framework for understanding the \L{}ojasiewicz exponent, with a focus on finite--max mechanisms.
At the same time, Section~\ref{sec:verification-uniformity-sharpness} shows that finite
testing cannot be expected in full generality: there exist natural filtrations for which no
finite set of valuations suffices to verify containment.  This highlights a fundamental
dichotomy between situations where the normalized Rees algebra provides complete control and
those where the valuative supremum remains intrinsically infinite.

\medskip

\noindent
\textbf{Finite testing and birational models.}
A central question is to identify intrinsic conditions guaranteeing finite verification.

\begin{question}
	Which properties of a graded family or filtration ensure that integral--closure containments
	can be verified on a fixed birational model?  In particular, can finite testing be detected
	purely from the structure of the normalized Rees algebra?
\end{question}

Answering this would clarify when the \L{}ojasiewicz exponent is governed by finitely many
divisorial valuations and when more subtle valuation-theoretic phenomena occur.

\medskip

\noindent
\textbf{Extremal valuations and geometric structure.}
When finite testing holds, the valuations computing the exponent play a distinguished role.

\begin{question}
	What geometric or intrinsic properties characterize valuations that compute the
	\L{}ojasiewicz exponent?  For instance, are such valuations uniquely determined by the
	normalized Rees algebra, or do they satisfy stability or minimality properties within the
	valuation space?
\end{question}

Understanding these extremal valuations would further clarify the stratification and
rigidity phenomena described in Part~\ref{part:structuralFinite-max}.

\medskip

\noindent
\textbf{Variation in families.}
Part \ref{part:structuralFinite-max} discusses rigidity and stratification results for the \L{}ojasiewicz exponents of families.

\begin{question}
	More generally, under what conditions is the \L{}ojasiewicz exponent constructible or
	semicontinuous in algebraic or analytic families?  Can weaker hypotheses than finite testing
	still ensure partial rigidity or controlled variation?
\end{question}

Such results would contribute to a broader algebraic understanding of equisingularity and
asymptotic invariants in families.

\medskip

\noindent
\textbf{Role of non--divisorial valuations.}
While divisorial valuations suffice in many situations, Section~\ref{sec:verification-uniformity-sharpness}
shows that more general valuations may be needed in principle.

\begin{question}
	Are there natural geometric or algebraic settings in which non--divisorial valuations
	necessarily compute the \L{}ojasiewicz exponent?  Conversely, what conditions ensure that
	divisorial valuations always suffice?
\end{question}

Clarifying the role of non--divisorial valuations would help delineate the limits of
birational and Rees--algebraic methods.

\medskip

These questions suggest that the valuative and Rees--algebraic framework developed here
provides not only computational tools, but also a structural foundation for studying
asymptotic containment invariants.  We expect that further development of this perspective
will lead to new insights into singularity theory, equisingularity, and related asymptotic
invariants.

\section*{Appendix} \label{sec:appendix}

\subsection*{A.1 Detailed proof of Lemma \ref{lem:analytic-IC}, following \cite{LT}} \label{app:LT-proof}

\begin{proof}[Proof of Lemma \ref{lem:analytic-IC}]
	Set $I=\a^p$.
	We use the following one--function statement, which is a specialization of
	\cite[Th\'eor\`eme~7.2]{LT} (equivalence of conditions (1),(5),(6)):
	
	\medskip
	\noindent\emph{(LT--S)}
	For $h\in R$,
	\[
	h\in\overline I
	\quad\Longleftrightarrow\quad
	\exists U\ni0,\exists C>0:\ |h(x)|\le C\|f(x)\|^p\text{ on }U.
	\]
	
	For completeness, we reproduce the proof of (LT--S).
	Let $\pi:Y\to U_0$ be the normalized blow--up of $I$ over a sufficiently small
	neighborhood $U_0$ of $0$.
	Then, $Y$ is normal, $\pi$ is proper, and $I\mathcal O_Y$ is invertible.
	Hence, for every $y\in\pi^{-1}(0)$ there exist a neighborhood $V\ni y$ and a
	holomorphic function $s$ on $V$ such that
	\[
	I\mathcal O_Y(V)=(s).
	\]
	Fix such $V$, and choose $V'\Subset V$ a relatively compact subset of $V$.
	
	\medskip
	\noindent\emph{Step 1: Two--sided comparison.}
	Since $(f_i\circ\pi)^p\in(s)$, there exist holomorphic functions $u_i$ on $V$
	such that
	\[
	(f_i\circ\pi)^p=s u_i.
	\]
	Let $A=\max_i\sup_{V'}|u_i|$. Then, for $z\in V'$,
	\[
	\|f(\pi(z))\|^p\le A|s(z)|.
	\]
	
	Since $I=\a^p$ is generated by finitely many degree -- $p$ monomials
	$F_\nu=f_1^{e_{1,\nu}}\cdots f_s^{e_{s,\nu}}$,
	their pullbacks generate $(s)$.
	Hence, there exist holomorphic $v_\nu$ on $V$ such that
	\[
	s=\sum_\nu v_\nu(F_\nu\circ\pi).
	\]
	Let $B=\sum_\nu\sup_{V'}|v_\nu|$. Then, for $z\in V'$,
	\[
	|s(z)|\le B\|f(\pi(z))\|^p.
	\]
	Thus, on $V'$,
	\[
	A^{-1}\|f\circ\pi\|^p\le|s|\le B\|f\circ\pi\|^p.
	\]
	
	\medskip
	\noindent\emph{Step 2: $(LT\text{--}S)$, first direction.}
	Assume $h\in\overline I$. Then, there exists a neighborhood $\tilde W$
	of $\pi^{-1}(0)$ such that
	\[
	h\circ\pi\in I\mathcal O_Y\quad\text{on }\tilde W.
	\]
	On $V \subseteq \tilde W$, write $h\circ\pi=sH$ with $H$ holomorphic.
	Let $M=\sup_{V'}|H|$. Then, for $z\in V'$,
	\[
	|h(\pi(z))|\le M|s(z)|\le MB\|f(\pi(z))\|^p.
	\]
	
	Cover $\pi^{-1}(0)$ by finitely many such $V'_k$ and let $C'$ be the maximum
	of the corresponding constants.
	There exists a neighborhood $W$ of $\pi^{-1}(0)$ such that
	\[
	|h\circ\pi|\le C'\|f\circ\pi\|^p\quad\text{on }W.
	\]
	Since $\pi$ is proper, there exists a neighborhood $U\ni0$ with
	$\pi^{-1}(U)\subseteq W$.
	For $x\in U$,
	\[
	|h(x)|\le C'\|f(x)\|^p.
	\]
	
	\medskip
	\noindent\emph{Step 3: $(LT\text{--}S)$, converse direction.}
 Assume $|h(x)| \le C\|f(x)\|^p$ on a neighborhood $U$ of $0$. Pulling back, we get
 $$
 |h\circ\pi| \le C\|f\circ\pi\|^p
 $$
 on $\pi^{-1}(U)$.
 
 Choose finitely many relatively compact open sets $V_1,\dots,V_r\Subset Y$ covering $\pi^{-1}(0)$ such that for each $\alpha$ the invertible ideal $I\cO_Y|_{V_\alpha}$ is principal, say
 $
 I\cO_Y(V_\alpha)=(s_\alpha),
 $
 with $s_\alpha\in \cO_Y(V_\alpha)$. Shrinking if necessary, we may assume that each $V_\alpha$ is contained in one of the neighborhoods $V$ used in Step 1. Hence, for each $\alpha$, there exists a constant $A_\alpha>0$ such that on $V_\alpha$ one has
 $
 \|f\circ\pi\|^p \le A_\alpha |s_\alpha|.
 $
 Therefore, on $(V_\alpha\cap \pi^{-1}(U))\setminus \{s_\alpha=0\}$,
 $$
 |h\circ\pi| \le C\|f\circ\pi\|^p \le C A_\alpha |s_\alpha|,
 $$
 so
 $$
 \left|\frac{h\circ\pi}{s_\alpha}\right| \le C A_\alpha.
 $$
 
 Since $Y$ is normal and $s_\alpha$ is a local generator of the invertible ideal $I\cO_Y|_{V_\alpha}$, the germ of $s_\alpha$ is a non-zero-divisor at every point of $V_\alpha$. Thus $(h\circ\pi)/s_\alpha$ defines a meromorphic function on $V_\alpha\cap \pi^{-1}(U)$, and the boundedness above implies that it is holomorphic there. Hence
 $$
 h\circ\pi \in s_\alpha \cO_Y(V_\alpha\cap \pi^{-1}(U))
 = I\cO_Y(V_\alpha\cap \pi^{-1}(U))
 $$
 for every $\alpha$.
 
 Since the sets $V_\alpha$ cover $\pi^{-1}(0)$, it follows that
 $
 h\circ\pi \in I\cO_Y
 $
 on a neighborhood of $\pi^{-1}(0)$.
 Let $E$ be any exceptional divisor of the normalized blow-up. Then
 $$
 \ord_E(h\circ\pi)\ge \ord_E(I\cO_Y).
 $$
 These divisors correspond to the Rees valuations of $I$, so the above inequalities imply $v(h)\ge v(I)$ for every such valuation $v$. Hence, $h\in \overline{I}$ by \Cref{thm:valuative-ic}.
 
	\medskip
	\noindent\emph{Step 4: Deduction of the lemma.}
	
	\smallskip
	\noindent$(1)\Rightarrow(2)$.
	If $\b^q\subseteq \overline{\a^p}$, then for each $j$,
	$g_j^q\in\overline{\a^p}$.
	Applying $(LT\text{--}S)$ to $h=g_j^q$ yields
	\[
	\|g_j(x)\|^q\le C_j\|f(x)\|^p
	\]
	on some neighborhood $U_j$.
	Let $U=\cap_jU_j$ and $C=\max_jC_j$.
	Then, for $x\in U$,
	\[
	\|g(x)\|^q=\max_j|g_j(x)|^q\le C\|f(x)\|^p.
	\]
	
	\smallskip
	\noindent$(2)\Rightarrow(1)$.
	Assume $\|g(x)\|^q\le C\|f(x)\|^p$ on a neighborhood $U$.
	For any $q$--tuple $(j_1,\dots,j_q)$,
	\[
	|g_{j_1}(x)\cdots g_{j_q}(x)|
	\le \|g(x)\|^q
	\le C\|f(x)\|^p.
	\]
	Each monomial $g_{j_1}\cdots g_{j_q}$ is a generator of $\b^q$.
	Applying $(LT\text{--}S)$ to each such monomial gives
	$\b^q\subseteq \overline{\a^p}$.
\end{proof}


\end{document}